\documentclass[a4paper,reqno,11pt]{amsart}
\usepackage[left=1 in, right=1 in,top=1 in, bottom=1 in]{geometry}
\usepackage{amsfonts}
\usepackage{amssymb} 
\usepackage{amsthm} 
\usepackage{amsmath} 
\usepackage{mathrsfs} 
\usepackage{color}
\usepackage{enumerate}
\usepackage{hyperref}
\usepackage[numbers,sort&compress]{natbib}
\usepackage{pdfsync} 
\usepackage{esint}
\usepackage{graphicx}
\usepackage{float}
\usepackage{caption}
\usepackage{subfigure}
\usepackage{url}
\usepackage{dsfont}
\usepackage{enumitem}
\usepackage{bm}

\usepackage{setspace}
\usepackage{inputenc} 
\newtheorem{theorem}{Theorem}[section]
\newtheorem{definition}{Definition}[section]
\newtheorem{lemma}{Lemma}[section]
\newtheorem{remark}{Remark}[section]
\newtheorem{proposition}{Proposition}[section]
\newtheorem{corollary}{Corollary}[section]
\numberwithin{equation}{section}

\renewcommand{\d}{\operatorname{d}}

\renewcommand{\u}{\mathbf{u}}

\newcommand{\pt}{\partial}
\newcommand{\ls}{\leqslant}
\newcommand{\gs}{\geqslant}
\newcommand{\odiv}{\operatorname{div}}
\newcommand{\ra}{\rightarrow}

\newcommand{\vp }{\varphi}

\newcommand{\eps}{\varepsilon}
\newcommand{\R}{{\mathbb R}}
\renewcommand{\O}{\mathbb{T}^3}
\allowdisplaybreaks[4]
\begin{document}
\title[Martingale Solutions to SCNS with Density-dependent Viscosity]{The Global Existence of Martingale Solutions to Stochastic Compressible Navier-Stokes Equations with Density-dependent Viscosity}

\author{Yachun Li}
\address[Y. Li]{School of Mathematical Sciences, CMA-Shanghai, MOE-LSC, and SHL-MAC, Shanghai Jiao Tong University, Shanghai 200240, P. R. China} \email{\tt ycli@sjtu.edu.cn}

\author{Lizhen Zhang}
\address[L. Zhang]{School of Mathematical Sciences, Shanghai Jiao Tong University, Shanghai 200240, P. R. China}
\email{\tt ZhangLizhen@sjtu.edu.cn}

\begin{abstract}
The global existence of martingale solutions to the compressible Navier-Stokes equations driven by stochastic external forces, with density-dependent viscosity and vacuum, is established in this paper. This work can be regarded as a stochastic version of the deterministic Navier-Stokes equations \cite{Vasseur-Yu2016} (Vasseur-Yu, Invent. Math., 206:935--974, 2016.), in which the global existence of weak solutions was established for adiabatic exponent $\gamma > 1$. For the stochastic case, the regularity of density and velocity is even worse for passing the limit in nonlinear terms. We design a regularized system to approximate the original system. To make up for the lack of regularity of velocity, we need to add an artificial Rayleigh damping term besides the artificial viscosity and damping forces in \cite{Vasseur-Yu-q2016,Vasseur-Yu2016}. Moreover, we have to send the artificial terms to $0$ in a different order.
\end{abstract}

\date{\today}
\subjclass[2010]{35A01, 35Q30, 35Q35, 35Q40, 35R60, 60H15. \\
{\bf Acknowledgements. } The authors' research was supported in part
by Chinese National Natural Science Foundation under grants 12371221, 12161141004, and 11831011. The authors were also grateful to the supports by Shanghai Frontiers Science Center of Modern Analysis and the Fundamental Research Funds for the Central Universities. The authors appreciate Prof. Hermano Frid and Prof. Deng Zhang for valuable discussion.}

\maketitle
{ \small{\small \tableofcontents}}
 \setcounter{tocdepth}{2}

\section{Introduction}
Compressible Navier-Stokes equations describe the motion of compressible viscous Newtonian fluid. Practically, besides the structural vibration of the fluid, the fluid will also be affected by random external forces such as humidity, wind, solar radiation, industrial pollution, etc. Therefore it is reasonable to add stochastic forces to Navier-Stokes equations. 
The system of stochastic compressible Navier-Stokes equations (SCNS for short) in torus $\mathbb{T}^{3}$ reads as
\begin{equation}\label{sto NS}
\left\{\begin{array}{l}
\rho_{t}+\operatorname{div}\left(\rho \mathbf{u}\right)=0,\\
\d (\rho \mathbf{u})+\left(\operatorname{div}\left(\rho \mathbf{u} \otimes \mathbf{u}\right)+\nabla p-\operatorname{div}\mathcal{T}\right)\d t = \rho \mathbb{F}(\rho,\mathbf{u}) \d W,
\end{array}\right.
\end{equation}
where $(t,x) \in \mathbb{R}^+ \times \O$, $\rho$ is the density, $\mathbf{u}=(u_{1}, u_{2}, u_{3})\in \R^{3}$ denotes the velocity, $p$ is the pressure. For polytropic isentropic gas, the pressure can be expressed as $\displaystyle p=a\rho^{\gamma}$, where $a$ is a constant. $\mathcal{T}=2 \mu \mathbb{D}(\u)+\lambda \operatorname{div} \u \mathbb{I}_{3}$ is the viscous stress tensor, $\mu>0$ is the shear viscosity coefficient, $\lambda+\frac{2}{3}\mu >0 $ is the bulk viscosity coefficient, $\mathbb{I}_{3}$ is the $3\times 3$ identity matrix,
\begin{equation}
\mathbb{D}\mathbf{u}=\frac{\nabla \u+\nabla \u^{\top}}{2}
\end{equation}
is the deformation tensor.
 $\rho \mathbb{F}\left(\rho,\mathbf{u}\right)\d W$ is the external multiplicative noise. $W$ is a cylindrical $\mathcal{F}_{t}$-adapted Wiener process in the stochastic basis $\left(\Omega,\mathcal{F},\mathbb{P}\right)$, $\left(\mathcal{F}_{t}\right)_{t\geqslant 0}$ is the right-continuous filtration, see \S \ref{def in app} in Appendix.
\begin{equation}
W=\sum\limits_{k=1}^{+\infty}e_{k}\beta_{k}, \quad \d W=\sum\limits_{k=1}^{+\infty}e_{k}\d \beta_{k},
\end{equation}
where $\beta_{k}$ is the standard real Brownian motion, $\{e_{k}\}_{k=1}^{+\infty}$ is an orthonormal basis in an auxiliary separable Hilbert space $\mathcal{H}$, which is isometrically isomorphic to $l^{2}$, the space of square-summable sequences. $\mathcal{H}$ is independent of domain $\O$.
Let $H$ be a Bochner space, $\mathbb{F}\left(\rho,\mathbf{u}\right)$ is an $H$-valued operator from $ \mathcal{H}$ to $ \mathcal{H}$. Denoting the inner product in $\mathcal{H}$ as $  \langle\cdot,\cdot \rangle $,
the inner product
\begin{equation}
\langle\mathbb{F}\left(\rho,\mathbf{u}\right),e_{k}\rangle = \mathbf{F}_{k}\left(t,x,\rho,\u\right)
\end{equation}
 is a 3D $H$-valued vector function, which shows the strength of the external stochastic forces,
 \begin{equation}
 \mathbb{F}\left(\rho,\mathbf{u}\right)\d W=\sum\limits_{k=1}^{+\infty}\mathbf{F}_{k}\left(t,x,\rho,\u\right)\d \beta_{k}, \quad \mathbb{F}\left(\rho,\mathbf{u}\right)=\sum\limits_{k=1}^{+\infty}\mathbf{F}_{k}\left(t,x,\rho,\u\right) e_{k}.
 \end{equation}

We will write $\mathbf{F}_{k}\left(t,x,\rho,\u\right)$ as $\mathbf{F}_{k}\left(\rho,\u\right)$ for short in the sequel.
Besides, $\rho$ and $\mathbf{u}$ are actually random variables depending on $\omega \in \Omega$ : $\rho=\rho(\omega,t,x)$, $\mathbf{u}=\mathbf{u}(\omega,t,x)$, for $\omega\in \Omega$; however, we write $\rho(t,x)$ and $\mathbf{u}(t,x)$ for short to make the notations consistent with the usual deterministic studies in fluid dynamic equations.

\subsection{Some progresses}
 When $\rho\mathbb{F}(\rho,\u)\equiv0$, system \eqref{sto NS} reduces to deterministic Navier-Stokes system. There are numerous results concerning the well-posedness of weak solutions, among which we select in this subsection the ones most relevant to our work. Kazhikhov-Shelukhin \cite{Kazhikhov-Shelukhin1977}
obtained the existence of global weak solutions to the initial boundary problem of 1D Navier-Stokes-Fourier system in 1977. Lions \cite{Lions1998} proved the global existence with large initial data in 1998 for  $\gamma \geqslant \frac{3}{2}$ as $n=2$ and $\gamma \geqslant \frac{9}{5}$ as $n=3$. Then in 2001, Feireisl-Novotn\'y-Petzeltov\'a \cite{Feireisl-Novotny-Petzeltova2001} extended the range of adiabatic exponent to $\gamma>\frac{n}{2}$ in Leray sense. 
 In 2003, Jiang-Zhang \cite{Jiang-Zhang2003} extended to $\gamma>1$ for the spherically symmetric solutions. In 2021, Hu \cite{Huxianpeng2021} obtained the renormalized global weak solutions for $\frac{6}{5}<\gamma<\frac{3}{2}$ up to a closed set with zero parabolic Hausdorff measure.

The above results are all for the case of constant viscosities.
However, the viscosity coefficients $\mu$ and $\lambda$ may depend on the temperature or the density, or both, which can be mathematically derived from the Chapman-Enskorg expansion in Boltzman equation \cite{TatsienLi}. Physically, one can also find this dependence in the  inverse cut-off power law models and Maxwellian molecules \cite{bookChanpmanCowling}. In the isentropic case, by Boyle law and Gay-Lussac law, such dependence is reduced to the dependence on the density, denoted as $\mu (\rho)$ and $\lambda(\rho)$. In 1995, Kazhikhov-Vaigant \cite{Kazhikhov-Vaigant1995} considered the global existence of solutions to Navier-Stokes-Fourier system for the 2D case with the viscosity coefficient satisfying $\mu\left(\rho\right)=1$, $\lambda\left(\rho\right)=\rho^{\beta}$, $\beta>3$.
In 2006, Bresch and Desjardins \cite{BD2006} first introduced Bresch-Desjardins entropy to get further more regularity of density so as to solve the shallow water problem, in which the Bresch-Desjardins relation (B-D relation) $\lambda(\rho)=2(\rho\mu^{\prime}(\rho)-\mu(\rho))$ is necessary to balance the terms in deducing the B-D entropy. One year later, Mellet-Vasseur  \cite{Mellet-Vasseur2007} deduced the Mellet-Vasseur type inequality which shows that $\rho\u^{2}$ is bounded in $L^{\infty}\left(0,T; L\log L\left( \O \right)\right)$, and established the compactness of $\sqrt{\rho}\u$ in $L^{2}([0,T]\times \O)$.
 In 2016 Vasseur-Yu \cite{Vasseur-Yu2016} proved the global existence when the viscosity coefficients satisfy $\mu(\rho)=\frac{\rho}{2}$, $\lambda(\rho)=0$ without extra constraint on the range of adiabatic exponent $\gamma$ by means of vanishing viscosity method, the Bresch-Desjardins entropy (B-D entropy) and Mellet-Vasseur type inequality.\\
Now we turn to the stochastic case: $\rho\mathbb{F}(\rho,\u)\not\equiv 0$. We consider the so-called martingale solutions of the system, which can be interpreted as weak solutions in the stochastic version \cite{Stroock-Varadhan,Karatzas1988}. The existence of 2D periodic martingale solutions to Navier-Stokes-Fourier equations was first studied by Tornatore \cite{Tornatore2000} for Cauchy problem in 2000. In 2013, Feireisl-Malowski-Notovn\'y \cite{Feireisl-Maslowski-Novotny2013} get the global existence for the 3D Dirichlet problem with the same stochastic force $\rho\d W$ in \cite{Tornatore2000}. For multiplicative noise, Wang-Wang \cite{Wang-Wang2015} obtained the global martingale solutions in a bounded domain with non-slip boundary conditions in 2015. In 2016, Breit-Hofmanov\'a \cite{Breit-Hofmanova2016} showed the existence of global martingale solutions on a torus with periodic boundary conditions.
 In 2017, Smith \cite{Smith2017} got the global martingale solutions with Lipschitz continuous stochastic force. There are also results on singular limit \cite{BreitFeireislHofmanova2016}, the stationary solutions \cite{BreitFeireislHofmanovaMalowski2019}, and the weak-strong uniqueness \cite{BreitFeireislHofmanova2017}, see also \cite{bookFeireislNovotny}. The above results are all for the case of constant viscosity coefficients except in \cite{Tornatore2000} where the case  $\lambda(\rho)=1+\rho^{\beta}$ was considered but no degeneracy is involved.
In 2020, Breit-Feireisl \cite{BreitFeireisl2020} considered the case with temperature-dependent viscosities but no degeneracy:  $\underline{\mu} \cdot \left(1+\theta\right)<\mu(\theta)<\bar{\mu} \cdot\left(1+\theta\right)$, where $\underline{\mu}$ and $\bar{\mu}$ are specific positive constants, they proved the existence of martingale solution to Navier-Stokes-Fourier equations. Very recently in 2022, for the density-dependent viscosities same as in  \cite{Vasseur-Yu2016}, Brze\'zniak-Dhariwal-Zatorska \cite{Brzeziak-Dhariwal-Zatorska2022} proved the sequential stability of the martingale solutions to \eqref{sto NS} with density-dependent viscosities, based on the assumption that there exist global martingale solutions. Our purpose in this paper is to establish the global existence of martingale solutions.

\subsection{Our main result}

We consider the stochastic isentropic compressible Navier-Stokes equations with density-dependent viscosities. Once the viscosity coefficients depend on density, $\d\left(\rho \mathbf{u}\right)$ will degenerate at vacuum $\rho(t,x)=0$, which causes the uncertainty of the evolution of velocity $\d \u$. Moreover, $\operatorname{div}\mathcal{T}$ also degenerates at vacuum, which has an influence on the regularity of $\u$. Like \cite{Vasseur-Yu2016}, we consider
\begin{equation}\label{viscosity condition}
\mu(\rho)=\frac{\rho}{2},\quad \lambda(\rho)=0,
\end{equation}
which satisfies the B-D relation $\lambda(\rho)=2(\rho\mu^{\prime}(\rho)-\mu(\rho))$.
With \eqref{viscosity condition}, \eqref{sto NS} turns into
\begin{equation}\label{sto NS system}
\left\{\begin{array}{l}\vspace{1.2ex}
\rho_{t}+\odiv\left(\rho \mathbf{u}\right)=0,\\
\d\left(\rho \mathbf{u}\right)+\left(\odiv\left(\rho \mathbf{u} \otimes \mathbf{u}\right)+\nabla p-\odiv(\rho \mathbb{D} \mathbf{u})\right)\d t
=\rho \mathbb{F}\left(\rho,\mathbf{u}\right)\operatorname{d}W.\\
\end{array}\right.
\end{equation}
We aim at the global existence of martingale solutions to \eqref{sto NS system} in torus $\O$, given the initial conditions
\begin{equation}\label{initial}
\rho|_{t=0}=\rho_{0},\quad \rho \mathbf{u}|_{t=0}=\mathbf{q}_{0}.
\end{equation}

We assume that there exists $f=\left(f_{1},f_{2},\cdots,f_{k},\cdots\right)$, the constant $f_{k}\geqslant 0$, such that
\begin{align}\label{property of rho F}
 \left|\rho \mathbf{F}_{k}\left(\rho,\mathbf{u}\right)\right|\ls f_{k}\cdot \left(|\rho|+|\rho\u|\right), \quad \sum\limits_{k=1}^{+\infty} f_{k}^{2}<+\infty.
\end{align}
When $\rho=0$, we define $\mathbf{F}_{k}\left(\rho,\u\right)=0$, when $\rho>0$, we write $\mathbf{F}_{k}\left(\rho,\u\right)=\mathbf{F}_{k}\left(\rho,\frac{\rho\u}{\rho}\right)$, $\mathbf{q}=\rho\u$.

We first give the definition of martingale solutions.
Let $\Lambda$ be the law (see \S \ref{def in app} in Appendix) of $\rho_{0}$ and $\mathbf{q}_{0}$, i.e. $\Lambda= \mathcal{L}\left[\rho_{0},\mathbf{q}_{0}\right]$.
\begin{definition}
$\left(\left(\Omega,\mathcal{F},\mathbb{P} \right),\rho, \u, W\right)$ is called a martingale solution to Cauchy problem \eqref{sto NS system}-\eqref{initial} with initial law $\Lambda$, where $\left(\rho, \u\right)$ is defined in probability space $\left(\Omega,\mathcal{F},\mathbb{P} \right)$, if
\begin{enumerate}
  \item $\left(\Omega,\mathcal{F},\mathbb{P} \right)$ is a stochastic basis with a complete right-continuous filtration $\mathcal{F}=\left(\mathcal{F}_{t}\right)_{t\geq 0}$;
  \item the density $\rho$ and the velocity $\u$ are stochastic processes adapted to $\mathcal{F}$;
  \item  $\mathcal{L}\left[\rho(0,x),\mathbf{q}(0,x)\right]=\Lambda$;
  \item the equation of continuity
\begin{align}
  -\int_0^T  \partial_{t} \varphi (t) \int_{\O}  \rho \psi(x) \d x \d t = \int_{\O} \varphi(0,x) \rho(0,x)\d x +\int_0^T  \varphi(t)  \int_{\O}\rho\u \cdot \nabla \psi(x) \d x \d t,
\end{align}
  holds $\mathbb{P}$ {\rm a.s.} for all $\varphi(t)\in C_{c}^{\infty}\left([0,T)\right)$, and all $\psi(x)\in C^{\infty}\left(\O\right)$;
  \item the momentum equation
\begin{align}
&-\varphi(0)\int_{\O}\mathbf{q}(0,x) \cdot \psi(x) \d x-\int_{0}^{T}\varphi_{t}(t) \int_{\O} \rho\u \psi(x) \d x \d t \notag\\
 &-\int_{0}^{T}\varphi(t) \int_{\O} \rho\u\otimes \mathbf{u}: \nabla \psi(x) \d x \d t-\int_{0}^{T} \varphi(t)\int_{\O} \rho^{\gamma} \odiv \psi(x) \d x \d t\\
 &-\int_{0}^{T}\varphi(t) \int_{\O} \rho \mathbb{D} \mathbf{u}: \nabla \psi(x) \d x \d t =\int_{0}^{T} \varphi(t)\int_{\O} \rho\mathbb{F}(\rho,\u)\psi(x)\d x\d W,\notag
\end{align}
 holds $\mathbb{P}$ {\rm a.s.} for all $\varphi(t)\in C_{c}^{\infty}\left([0,T)\right)$, and all $\psi(x)\in C^{\infty}\left(\O\right)$.
\end{enumerate}
\end{definition}
For the convenience of stating our main theorem, we introduce the notation of the energy
\begin{equation}\label{energy estimate1}
e(t)=\int_{\O}\left(\frac{1}{2}\rho|\mathbf{u}|^{2} + \frac{a}{\gamma}\rho^{\gamma}\d z\right)\d x=\int_{\O}\left(\frac{1}{2}\frac{|\mathbf{q}|^{2}}{\rho} + \frac{a}{\gamma}\rho^{\gamma}\d z\right)\d x,
\end{equation}
the B-D entropy
\begin{align}
\tilde{e}(t)=&\int_{\O}\left(\frac{1}{2}\rho\left|\u+\nabla \ln\rho\right|^{2}+\frac{a}{\gamma}\rho^{\gamma}\right)\d x\notag\\
= &\int_{\O}\left(\frac{1}{2}\frac{|\mathbf{q}|^{2}}{\rho}+\left|\mathbf{q}\cdot\nabla \ln\rho\right|+\frac{1}{2}\rho\left|\nabla \ln\rho\right|^{2}+\frac{a}{\gamma}\rho^{\gamma}\right)\d x,
\end{align}
and the Mellet-Vasseur quantity
\begin{align}
\tilde{\tilde{e}}(t)=\int_{\O} \rho\left(1+\frac{\left|\mathbf{q}\right|^{2}}{\rho^{2}}\right) \ln \left(1+\frac{\left|\mathbf{q}\right|^{2}}{\rho^{2}}\right) \d x.
\end{align}
 Our main theorem is as follows.
\begin{theorem}\label{my theorem}
Let $\gamma>1$. If there exists a positive constant $C$, such that, for some constant $r> 4$,
\begin{align}
&\mathbb{E}\left[e(0)^{r}\right]\ls C, \quad \mathbb{E}\left[\tilde{e}(0)^{r}\right]\ls C, \quad \mathbb{E}\left[\tilde{\tilde{e}}(0)^{r}\right]\ls C,
\end{align}
and
\begin{equation}\label{Condition for rho}
\Lambda\left(\left\{\left.\omega\in\Omega \right|\rho_{0}(x)\gs 0\right\}\right)=1,\quad \Lambda \left(\left\{\omega\in\Omega \left|0<\underline{m}\ls \int_{\O}\rho_{0}(x) \d x \ls \overline{m}<\infty\right.\right\}\right)=1,
\end{equation}
where $\underline{m}$ and $\overline{m}$ are constants, then there exists a martingale solution $\left(\left(\Omega,\mathcal{F},\mathbb{P} \right),\rho, \u, W\right)$ to (\ref{sto NS system}) in a completed probability space $\left(\Omega,\mathcal{F},\mathbb{P} \right)$. Moreover, there hold
\begin{enumerate}
\item the energy inequality
  \begin{equation}
\mathbb{E}\left[\sup\limits_{t\in[0,T]}e(t)^{r}\right] \ls C(1+\mathbb{E}\left[e(0)^{r}\right]),
\end{equation}
where the constant $C>0$ depends on $r$, $T$, and $\sum\limits_{k=1}^{+\infty}f^{2}_{k}$;
\item the B-D entropy estimates:
\begin{align}
\mathbb{E}\left[\sup\limits_{t\in[0,T]}\tilde{e}(t)^{r}\right] \ls C\left(1+\mathbb{E}\left[\tilde{e}(0)^{r}\right] \right),
\end{align}
where the constant $C>0$ depends on $r,T$, $\sum\limits_{k=1}^{+\infty}f^{2}_{k}$ and the initial data;
  \item the stochastic Mellet-Vasseur inequality:
\begin{align}
     \mathbb{E}\left[\sup\limits_{t\in[0,T]}\left(\tilde{\tilde{e}}(t)\right)^{r}\right]
\ls C\left(1+\mathbb{E}\left[\left(\tilde{\tilde{e}}(0)\right)^{r}\right]\right),
\end{align}
where  the constant $C>0$ depends on $r, T$, $\sum\limits_{k=1}^{+\infty}f^{2}_{k}$ and the initial data.
\end{enumerate}
\end{theorem}
 Our proof is based on the vanishing viscosity method. We regularize the original system to construct approximate solutions with better regularity as follows:
 \begin{align}\label{sto quantum damping NS system}
\left\{\begin{aligned}\vspace{1.2ex} 
&\rho_{t}+\operatorname{div}(\rho \mathbf{u})=\eps \triangle \rho,\\
&\d(\rho \mathbf{u})+\left(\operatorname{div}\left(\rho \mathbf{u} \otimes \mathbf{u}\right)+\nabla p-\operatorname{div}\left(\rho \mathbb{D} \mathbf{u}\right)+ \eps \triangle^{2} \u + r_{0} |\u|^{2}\u+r_{1} \rho|\mathbf{u}|^{2} \mathbf{u}+r_{2}\u\right)\d t\\
&=\left(\delta \rho\nabla \triangle^{9}\rho + \eta \rho^{-10}+\kappa \rho \left(\nabla\left(\frac{\triangle \sqrt{\rho}}{\sqrt{\rho}}\right)\right)\right)\d t+\rho \mathbb{F}\left(\rho,\mathbf{u}\right)\d W,\\
\end{aligned}\right.
\end{align}
where $\eps \triangle \rho$ is artificial viscosity, $r_{0} |\u|^{2}\u$ is Rayleigh damping force, $r_{1} \rho|\mathbf{u}|^{2} \mathbf{u}+r_{2}\u$ are the drag forces, $\delta \rho\nabla \triangle^{9}\rho + \eta \rho^{-10}$ are artificial pressures, $\kappa \rho \left(\nabla\left(\frac{\triangle \sqrt{\rho}}{\sqrt{\rho}}\right)\right)$ is the quantum term.
The initial conditions are $\rho(t,x)|_{t=0}=\rho_{0}(x), \left(\rho \mathbf{u}\right)(t, x)|_{t=0}=\mathbf{q}_{0}(x)$. 
By Galerkin approximation, we project \eqref{sto quantum damping NS system} in $m$D approximated space $H_{m}$ and get the global existence of approximated solutions in $H_{m}$. Employing Burkholder-Davis-Gundy's inequality, we estimate the $r$-th moment of the stochastic term. Then we do the energy estimates by It\^o's formula so as to figure out the tightness of the laws of approximated solutions $\Pi_{m}\left(\rho\right)$ and $\Pi_{m}\left(\rho\u\right)$. By Jakubowski's extension of Skorokhod's representation theorem and the energy estimates, we pass to the limit $m\ra +\infty$ to obtain the global existence of martingale solutions of \eqref{sto quantum damping NS system}. It is worth mentioning that some estimates of the higher order derivatives of density are not uniform in $\eps$. So we would establish the stochastic B-D entropy estimates uniformly in $\eps$. More specifically, the stochastic B-D entropy estimates give the $H^{1}$ regularity of $\rho^{\frac{1}{2}}$. Then we vanish the artificial viscosity by sending $\eps\ra 0$. By It\^o's formula, passing to the limit $\kappa \ra 0$, we derive the Mellet-Vasseur type inequality in stochastic version. The Mellet-Vasseur's inequality gives the strong convergence of $\rho^{\frac{1}{2}}\u$, so that the convergence of the nonlinear term $\rho\u\otimes \u$ holds  $\mathbb{P}$ a.s. Let $n\ra \infty$, $\delta \ra 0$, $\eta\ra 0$, $r_{0}\ra 0$ at the same time, where $n$ is the truncation of $\u$, with $|\u|\ls n $. Finally, we take the limit $r_{1}\ra 0$ and $r_{2}\ra 0$, then we obtain the global existence of martingale solutions of (\ref{sto NS system}) for $\gamma>1$, i.e., Theorem \ref{my theorem}.
In summary, the order of taking these limits is as follows:
\begin{enumerate}
  \item The artificial viscosity $\eps \triangle \rho$ vanishes, i.e., $ \eps \ra 0 $.
  \item The quantum term $ \kappa \rho \left(\nabla\left(\frac{\triangle \sqrt{\rho}}{\sqrt{\rho}}\right)\right)$ vanishes, i.e., $\kappa \ra 0$.
  \item The artificial pressure terms $\delta \rho\nabla \triangle^{9}\rho$ and $\eta \rho^{-10}$ vanish at the same time with the Rayleigh damping term $r_{0}|\u|^{2}\u$, i.e., $\delta \ra 0$, $\eta\ra 0$, $r_{0}\ra 0$, $n\ra \infty$ at the same time, where $n$ is the truncation of $\u$, $|\u|\ls n $.
  \item The artificial damping terms $r_{1}\rho|\u|^{2}\u$ and $r_{2}\u$ vanish, i.e., $r_{1}\ra 0$ and $r_{2}\ra 0$.
\end{enumerate}

The difficulties and the strategies for our problem are analyzed as follows.
\begin{itemize}
  \item [\textbf{1.}] \textbf{Loss of regularity of $\rho\u$ with respect to time $t$ caused by the stochastic external forces.} The Gaussian process in the right-hand side of the $(\ref{sto NS system})_2$ is at most H\"older-$\frac{1}{2}$ continuous. Moreover, the Aubin-Lions' lemma does not apply to $\rho\u$ in this case either. Therefore, we should consider the time continuity of $\rho\u$ so as to obtain the compactness of $\rho\u$ in $C\left([0,T];L^{\frac{3}{2}}\left( \O \right)\right)$. Compared with the regularity of $\rho\u\in C\left([0,T];L^{2}\left( \O \right)\right)$ in \cite{Breit-Hofmanova2016}, $\rho\u$ exhibits lower regularity in spatial variables in this paper.
  \item [\textbf{2.}] \textbf{Low regularity of $\u$ on account of degenerate viscosities}. Compared with the case of constant viscosity coefficients in \cite{Breit-Hofmanova2016}, the equation $(\ref{sto NS system})_{2}$ becomes degenerate since the diffusion term $\odiv (\rho\mathbb{D}\u)$ will vanish at vacuum. 
 We cannot obtain the estimate of $\left\|\mathbb{D}\u\right\|_{L^{2}([0,T]\times\left( \O \right))}$ from the term $\odiv (\rho\mathbb{D}\u)$ since $\rho$ does not have positive  lower bound. The lack of regularity of $\u$ is an obstacle to the convergence of approximations for the nonlinear term $\rho\u\otimes \u$. To make up the lack of regularity of $\u$, motivated by \cite{BD2006}, we construct an approximated scheme by adding artificial damping terms $-r_{1}\rho|\u|^{2}\u-r_{2}\u$ to the momentum equation. $-r_{1}\rho|\u|^{2}\u-r_{2}\u$ are drag forces in shallow water model, where $r_{1}\rho|\u|^{2}\u$ is in the turbulent regime, and $r_{2}\u$ ($r_{2}\geqslant 0$) is for the laminar case. We add the term $\eps\triangle^{2}\u$ for the estimate of $\triangle\u$, see \cite{Zatorska2012}. Inspired by \cite{Vasseur-Yu-q2016}, we added the quantum term $\kappa \rho \left(\nabla\left(\frac{\triangle \sqrt{\rho}}{\sqrt{\rho}}\right)\right)$, the artificial pressure term $\delta \rho\nabla \triangle^{9}\rho+\eta \rho^{-10}$  to provide more regularity for $\rho$ and $\u$. For the theory about quantum Navier-Stokes equations, please refer to \cite{Dong2010, Jungle2010, Gisclon-Violet2015, Jiang2011}. 
However, for our stochastic case, since $\rho\u \in C\left([0,T];L^{\frac{3}{2}}\left( \O \right)\right)$, these regularities are still not enough for the convergence of approximations of the nonlinear term $\rho\u\otimes \u$, and we find that the desired regularity of $\u$ can be obtained by adding an artificial damping term $-r_{0}|\u|^{2}\u$ which is exactly Rayleigh damping term \cite{dampingintro,bookRayleigh}. When we estimate the stochastic integral, the high regularity of $\rho$ from the stochastic B-D entropy estimates makes up the lack of regularity of $\u$.
  \item [\textbf{3.}]\textbf{Order of passing to the limit in stochastic case.} When vanishing these artificial terms to recover the original system, we have to take the order of passing to the limit different from the deterministic case. For example, Vasseur-Yu \cite{Vasseur-Yu2016} vanished artificial pressure terms before vanishing quantum terms. For the stochastic case, if taking the same order, we would not have the uniform bound of the Rayleigh damping term $-r_{0}|\u|^{2}\u$ in the stochastic B-D entropy. Hence we vanish the quantum term $\kappa \rho \left(\nabla\left(\frac{\triangle \sqrt{\rho}}{\sqrt{\rho}}\right)\right)$ first, then let the Rayleigh damping term $-r_{0}|\u|^{2}\u$ vanish at the same time with the artificial pressure terms $\delta \rho\nabla \triangle^{9}\rho$ and $\eta \rho^{-10}$. In addition, we use It\^o's formula instead of the renormalization of momentum equation to deduce the stochastic Mellet-Vasseur type inequality. Vasseur-Yu \cite{Vasseur-Yu2016} renormalized the momentum equation to obtain the convergence of terms for the Mellet-Vasseur's inequality. 
       Without the ``renormalized'' solutions for our stochastic case, the arising trouble of convergence is dealt by the regularity of artificial pressure terms and the subtle order of passing limit. 
\end{itemize}

The paper is organized as follows. In \S 2, we first show the global-in-time existence of solutions to stochastic quantum Navier-Stokes system with the damping terms and the artificial pressures. In \S 3, we establish the stochastic B-D entropy estimates uniformly in $\eps$, then we take the limit $\eps\ra 0$. In \S 4, we derive the Mellet-Vasseur inequality in stochastic version, meanwhile let $\delta \ra 0$, $\eta\ra 0$, $r_{0}\ra 0$, $n\ra \infty$. Finally we prove Theorem \ref{my theorem} by sending the limit $r_{1}\ra 0$ and $r_{2}\ra 0$. \S 5 is Appendix, the basic theories in Stochastic analysis is listed for the convenience of readers.

\section{Global existence of martingale solutions to the regularized system}
In this section, we first construct a parabolic approximation structure to $\eqref{sto quantum damping NS system}_{1}$, and then apply Galerkin approximation to $\eqref{sto quantum damping NS system}_{1}$. Then like \cite{BreitFeireislHofmanova-book2018,Karatzas1988}, for the convenience of energy estimates, we truncate $|\u|$ by $R$, dividing the time interval $[0,T]$ into $\frac{T}{h}$ parts, next we construct an iteration scheme (\ref{iteration of rho and u}). Next, we pass to the limit $h\ra 0$, which exactly shows that the limit of solution of (\ref{iteration of rho and u}) satisfies the system (\ref{Galerkin approximation scheme}) for $\forall t$. Then we let $R\ra +\infty$ to achieve the global existence of the Galerkin approximated regularized system. At the end of this section, we aim to obtain the global existence of martingale solutions in infinite dimensional space by sending $m\ra +\infty$.

\subsection{Galerkin approximation of the regularized  system in finite dimensional space $H_{m}$}
Let $\{w_{\mathbf{k}}\}_{|\mathbf{k}|=1}^{+\infty}$ be an orthogonal basis of $L^{2}(\O)$, $\mathbf{k}=(k_{1},k_{2},k_{3})$, $|\mathbf{k}|=\max\limits_{j=1,2,3}k_{j}$, more precisely, $w_{\mathbf{k}}(\mathbf{x})=e^{i2\pi \mathbf{k}\cdot \mathbf{x}}$.
The finite-dimensional space is defined as
\begin{equation}
H_{m}=\left\{ w=\sum\limits_{\mathbf{m},\max |m_{i}|\ls m}a_{\mathbf{m}}\cos2\pi \mathbf{m}\cdot \mathbf{x}+b_{\mathbf{m}}\sin 2\pi \mathbf{m}\cdot \mathbf{x} \right\}^{3}.
\end{equation}
For any $f\in W^{k, p}(\O)$, $\{w_{\mathbf{k}}\}_{|\mathbf{k}|=1}^{+\infty}$ is also an orthogonal basis of $W^{k, p}(\O)$,  $\Pi_{m}\left(f\right)=\sum\limits_{\left|\mathbf{k}\right|=1}^{m}\langle f, w_{\mathbf{k}}\rangle w_{\mathbf{k}}$. By choosing appropriate $a_{\mathbf{m}}$ and $b_{\mathbf{m}}$ (cf. \cite{Grafakos2008}), it holds $\Pi_{m}\left(f\right)\rightarrow f$ in $W^{k, p}\left( \O \right)$ as $m\rightarrow \infty$, $1 \ls p < +\infty$.

We employ the following Galerkin approximation scheme for \eqref{sto quantum damping NS system}:
\begin{align}\label{Galerkin approximation scheme}
\left\{\begin{array}{l}\vspace{1.2ex}
\rho_{t}+\operatorname {div}(\rho \mathbf{u})= \eps \triangle \rho,\\
\d  \Pi_{m}\left(\rho \mathbf{u}\right)+\Pi_{m}\left(\operatorname{div}(\rho \mathbf{u} \otimes \mathbf{u})+ \nabla\left( a\rho^{\gamma}\right)-\operatorname{div}(\rho
\mathbb{D} \mathbf{u})-\frac{11}{10}\eta \nabla \rho^{-10}\right)\d t\\
=\Pi_{m}\left(\left(-r_{0}|\u|^{2}\u-r_{1} \rho|\u|^{2} \u-r_{2}\u\right)\right)\d t                                                                                     -\Pi_{m}\left(\eps \nabla \rho\nabla \u\right)\d t-\Pi_{m}\left(\eps \triangle^{2} \u\right)\d t\\
+\Pi_{m}\left(\delta \rho\nabla \triangle^{9}\rho \right)\d t+\kappa \Pi_{m}\left(\rho\left(\nabla\left(\frac{\triangle \sqrt{\rho}}{\sqrt{\rho}}\right)\right)\right)\d t+\Pi_{m}\left(\rho \mathbb{F}(\rho,\u)\right)\d W. \\
\end{array}\right.
\end{align}
By the definition of $\Pi_{m}$, $\rho_{0}$ is approximated by $\Pi_{m}\left(\rho_{0}\right)$, by choosing suitable coefficients $a_{\mathbf{m}}, b_{\mathbf{m}}$, there exists a positive constant $\underline{\rho}$ and $\Pi_{m_{k}}\left(\rho_{0}\right)$, a subsequence of $\Pi_{m}\left(\rho_{0}\right)$, $\Pi_{m_{k}}\left(\rho_{0}\right)\geqslant \underline{\rho}$, where $\underline{\rho}$ changes with $m_{k}$. The subsequence is still denoted by $\Pi_{m}\left(\rho_{0}\right)$ for convenience.
By the maximum principle of the parabolic structure and $\eqref{Galerkin approximation scheme}_1$, we have
\begin{equation}\label{maximum principle}
\inf\limits_{x\in \O}\Pi_{m}\left(\rho_{0}\right) e^{-\int_{0}^{t}\left\| \odiv \mathbf{u} \right\|_{L^{\infty}}\d s}\ls\rho\ls \sup\limits_{x\in \O}\Pi_{m}\left(\rho_{0}\right) e^{-\int_{0}^{t}\left\| \odiv \mathbf{u} \right\|_{L^{\infty}}\d s}.
\end{equation}

%
In the finite dimensional space $H_{m}$, it naturally holds 
  $\Lambda\left(\{ \Pi_{m}\left(\rho_{0}\right)\ls \bar{\rho}\}\right)=1$, where $\bar{\rho}$ is a constant; $\Lambda\left(\left\{\left\| \Pi_{m}\left(\mathbf{q}_{0}\right)\right\|_{H_{m}}\ls \bar{q}\right\}\right)=1$, where $\bar{q}$ is a positive constant, and $\underline{\rho}$, $\bar{q}$ and $\bar{\rho}$ may change with $m$.
For the sake of (\ref{maximum principle}), we define a cut-off function in $C^{\infty}\left(\mathds{R}^{3}\right)$:
\begin{equation}
\chi_{R}(z)=\left\{\begin{array}{l}
1,\quad \text{for} \quad \left\| z\right\|_{H_{m}}\ls R,\\
0\ls\chi_{R}(z)\ls 1,\quad \text{for} \quad R \ls \left\| z\right\|_{H_{m}}\ls R+1,\\
0,\quad \text{for} \quad \left\| z\right\|_{H_{m}}\gs R+1.\\
\end{array}\right.
\end{equation}
$[\u]_{R}=\chi_{R}(\u)\u$ represents the truncation of $\u$.

\subsection{The global existence of martingale solutions of the regularized approximated system in $H_{m}$}
 We get the solutions of (\ref{Galerkin approximation scheme}) by linearization, iteration, uniform energy estimates and stochastic tightness theory.

\subsubsection{Linearization}

We split the time interval $[0,T]$ into grid, setting $h$ as the step length. 
We define the approximated iteration scheme recursively,
\begin{align}\label{iteration of rho and u}
\left\{\begin{aligned}\vspace{1.2ex}
&\rho_{t}+\odiv\left(\rho(t)[\mathbf{u}(nh)]_{R}\right) = \eps \triangle \rho(t), \quad t\in [nh,(n+1)h],\\
&\d \Pi_{m}\left(\rho \mathbf{u}\right)+\Pi_{m}\left(\odiv (\rho(t)\left[\mathbf{u}(nh)\right]_{R}\otimes \mathbf{u}(nh))\right)\d t+ \Pi_{m} \left(\chi_{R}(\u(nh))\nabla p(\rho(t))\right)\d t\\
&-\left(\Pi_{m} \left(\chi_{R}(\u(nh))\odiv \left(\rho\mathbb{D} \mathbf{u}(nh)\right)\right)\right)\d t \\
&+\left(\Pi_{m} \left(\frac{11}{10} \eta \nabla \rho^{-10}\chi_{R}\left(\u(nh)\right)+\chi_{R}(\u(nh))\delta \rho \nabla \triangle^{9}\rho \right)\right)\d t \\
&=\Pi_{m} \left((-r_{0} |\mathbf{u}(nh)|^{2}\left[\mathbf{u}(nh)\right]_{R}-r_{1} \rho|\mathbf{u}(nh)|^{2}\left[\mathbf{u}(nh)\right]_{R}- r_{2} \left[\mathbf{u}(nh)\right]_{R})\right)\d t\\
&-\eps \Pi_{m}\left(\chi_{R}(\u(nh))\nabla \rho\nabla \mathbf{u}(nh)\right)\d t -\eps \Pi_{m}\left(\chi_{R}(\u(nh))\triangle^{2} \mathbf{u}(nh)\right)\d t\\
&+\kappa \Pi_{m}\left(\chi_{R}(\u(nh))\rho\left(\nabla\left(\frac{\triangle \sqrt{\rho}}{\sqrt{\rho}}\right)\right)\right)\d t
+ \Pi_{m}\left(\chi_{R}(\u(nh))\rho \Pi_{m}\left[\mathbb{F}(\rho,\mathbf{u}(nh))\right]\right)\d W, \\
&\rho(nh)=\rho(nh-)=\lim\limits_{s\ra nh} \rho(s),\quad \rho(0)=\Pi_{m}\left(\rho_{0}\right),\\
&\u(nh)=\u(nh-)=\lim\limits_{s\ra nh} \u(s), \quad \rho\u(0)=\mathbf{q}_{0}.
\end{aligned}\right.
\end{align}

We introduce a linear mapping $\mathscr{M}\left(\rho\right)$,
\begin{equation}
\mathscr{M}\left(\rho\right): H_{m} \rightarrow H_{m}^{*}, \quad \mathscr{M}\left(\rho\right)(\mathbf{z})=\Pi_{m}(\rho \mathbf{z}),
\end{equation}
or
\begin{equation}
\int_{\O} \mathscr{M}\left(\rho\right) \mathbf{z} \cdot \mathbf{\vp}_{m} \d  x \equiv \int_{\O} \Pi_{m}(\rho \mathbf{z}) \cdot \mathbf{\varphi}_{m} \d  x \quad \text { for all } \mathbf{\vp}_{m} \in H_{m}.
\end{equation}
From (\ref{iteration of rho and u}), $\mathscr{M}\left(\rho\right)$ is invertible for $\rho\neq 0$,  we write $\rho=S(\mathbf{u}(nh), \Pi_{m}\left(\rho_{0}\right))$, $S$ is a Lipschitz-continuous operator. This gives a representation of $\u$ explicitly. 
For any $h$, the stochastic equation admits a unique solution in $[nh,(n+1)h]$, since the stochastic equation admits a unique solution in a time interval $[0,t)$ for any $t$.  $\left\|\u\right\|_{H_{m}}\ls C $ is necessary for taking the limit $h\ra 0$, where $C$ is independent of $R$, $h$, time $\tau$, $\tau \in [0,T]$, which will be showed in the following.
 The solution of \eqref{iteration of rho and u} are concerned with $h, R, m, \eps, \kappa, \eta, \delta, r_{0}, r_{1}$, and $ r_{2}$, since the limit $h\ra 0$ is mainly focused in this subsection, and for simplicity, we denote the solutions as $\rho_{m,h}$ and $\mathbf{u}_{m,h}$. During the limit taking process, we will use the following notations:
\begin{align}
\rho_{m,h} \stackrel{h\ra 0 }{\longrightarrow} \rho_{m,R} \stackrel{R\ra \infty }{\longrightarrow}\rho_{m} \stackrel{m\ra \infty }{\longrightarrow}\rho_{\rm reg} \stackrel{\eps \ra 0 }{\longrightarrow} \rho_{{\rm I}}\stackrel{\kappa \ra 0 }{\longrightarrow} \rho_{{\rm II}}\stackrel{\eta, \delta, r_{0} \ra 0}{\longrightarrow} \rho_{{\rm III}} \stackrel{r_{1},r_{2} \ra 0}{\longrightarrow} \rho,\\
\u_{m,h} \stackrel{h\ra 0 }{\longrightarrow} \u_{m,R} \stackrel{R\ra \infty }{\longrightarrow}\u_{m} \stackrel{m\ra \infty }{\longrightarrow} \u_{\rm reg} \stackrel{\eps \ra 0 }{\longrightarrow} \u_{{\rm I}}\stackrel{\kappa \ra 0 }{\longrightarrow} \u_{{\rm II}}\stackrel{\eta, \delta, r_{0} \ra 0}{\longrightarrow} \u_{{\rm III}} \stackrel{r_{1},r_{2} \ra 0}{\longrightarrow} \u,
\end{align}
where $\rho_{m,R}$ and $\u_{m,R}$ are concerned with $ R, m, \eps, \kappa, \eta, \delta, r_{0}, r_{1}$, and $ r_{2}$; $\rho_{m}$ and $\u_{m}$ are concerned with $m,\eps, \delta, \eta, r_{0}, r_{1}, r_{2}$ and $\kappa$; $\rho_{\rm reg}$ and $\u_{\rm reg}$ are concerned with $ \eps, \kappa, \eta, \delta, r_{0}, r_{1}$, and $ r_{2}$; $\rho_{{\rm I}}$ and $\u_{{\rm I}}$ are concerned with $\kappa, \eta, \delta, r_{0}, r_{1}$, and $ r_{2}$; $\rho_{{\rm II}}$ and $\u_{{\rm II}}$ are concerned with $ \eta, \delta, r_{0}, r_{1}$, and $ r_{2}$;  $\rho_{{\rm III}}$ and $\u_{{\rm III}}$ are concerned with $ r_{1}$ and $ r_{2}$; $\rho$ and $\u$ are solutions to the original system.

From \eqref{iteration of rho and u}, by Burkh\"older-Davis-Gundy's inequality (see Appendix \ref{appendix 2nd part}) and the fact that norms are equivalent in $H_{m}$,
 the following estimate holds uniformly in $h$ and $\tau$:
\begin{equation}
\begin{split}
\mathbb{E}[\left\| \rho_{m,h} \u_{m,h}(\tau)\right\|_{H_{m}}^{r}]+\mathbb{E}[\left\| \u_{m,h}(\tau)\right\|_{H_{m}}^{r}]\ls C\left(\mathbb{E}\left[\left\| \mathbf{q}_{0}\right\|_{H_{m}}^{r}\right] +1\right),\\
\end{split}
\end{equation}
where $C$ depends on $m, R, T, \underline{\rho}(m), \bar{\rho}$ and $r$ \cite{BreitFeireislHofmanova-book2018}. Moreover, it holds $\rho_{m,h}|_{t=0}=\rho_{0}$, $\rho_{m,h}\u_{m,h}|_{t=0}=\mathbf{q}_{0}$, $ \mathbb{P}$ a.s.

To determine the space of $\rho_{m,h}$ and $\mathbf{u}_{m,h}$, for \eqref{Galerkin approximation scheme}, we estimate
\begin{equation}
\begin{split}
\mathbb{E}\left[\Pi_{m}\left\|\rho_{m,h} \mathbf{u}_{m,h}(t_{1},\cdot)-\rho_{m,h} \mathbf{u}_{m,h}(t_{2},\cdot)\right\|^{r}_{H_{m}}\right]\ls C(t_{1}-t_{2})^{\frac{r}{2}} \left(\mathbb{E}\left[\left\| \mathbf{q}_{0}\right\|_{H_{m}}^{r}\right] +1\right),\\
\end{split}
\end{equation}
 for $r>1$, $0\ls t_{1},t_{2}\ls T$, since the deterministic nonlinearity terms are linearized and $\u_{m,h}$ is truncated by $R$. Since $\rho_{m,h}(t)$ is lower bounded by $\underline{\rho}(m)$, upper bounded by $\bar{\rho}$, by Kolmogorov-Centov's continuity theorem, it holds
\begin{equation}
\begin{aligned}
   \mathbb{E}\left[\left\| \tilde{\u}_{m,h}\right\|^r_{{C_{t}^{0,\beta}}H_{m}}\right]\ls C \left(\mathbb{E}\left[\left\| \mathbf{q}_{0}\right\|_{H_{m}}^{r}\right] +1\right),\\
\end{aligned}
\end{equation}
where $\tilde{\u}_{m,h}$ is a modification of $\u_{m,h}$, and we still use the notation $\u_{m,h}$ in the following, $r>2$, $\beta\in\left(0,\frac{1}{2}-\frac{1}{r}\right)$ and $C$ only depends on $m, R, \underline{\rho}(m), \bar{\rho},T$ and $r$.
 For $\rho_{m,h}$, \eqref{iteration of rho and u} 
implies
\begin{align}
&\left\| \rho_{m,h}(\tau_{1})-\rho_{m,h}(\tau_{2}) \right\|_{H_{m}}
\ls \left\|\int_{\tau_{1}}^{\tau_{2}}\operatorname {div}\left(\rho_{m,h} [\mathbf{u}_{m,h}(nh)]_{R}\right)\d t\right\|_{H_{m}}+ \left\| \int_{\tau_{2}}^{\tau_{1}}\eps \triangle \rho_{m,h} \d t\right\|_{H_{m}}\notag\\
\ls &C \left\| \rho_{m,h}\right\|_{H_{m}}\left(\tau_{1}-\tau_{2}\right),
\end{align}
where $C$ depends on $R,m$, so $\rho_{m,h} \in C^{0,1}\left([0,T];H_{m}\right)$, the sequence $\{\left\|\rho_{m,h}\right\|_{H_{m}}\}_{h}$ is equiv-continuous in $C([0,T])$.

We define the solution space $\mathcal{X}_{1}=\mathcal{X}_{\rho}^{1}\times \mathcal{X}_{\mathbf{u}}^{1}\times \mathcal{X}_{W}^{1}=C\left([0,T];H_{m}\right)$ $\times C^{0,\nu}\left([0,T];H_{m}\right)\times C\left([0,T];\mathcal{H}\right)$, $0<\nu<\beta$, $\beta\in\left(0,\frac{1}{2}-\frac{1}{r}\right)$, for some $r>2$.
We consider the set
\begin{equation}
B_{\rho,m}=\left\{\left.\rho_{m,h} \in C^{0,1}\left([0,T];H_{m}\right)\right|\|\rho_{m,h}\|_{C^{0,1}\left([0,T];H_{m}\right)} \leq L\right\},
\end{equation}
which is relatively compact in $\mathcal{X}_{\rho}^{1}$ due to Arzel\`a-Ascoli's theorem. In fact, by Arzel\`a-Ascoli's continuity theorem, we gain that $\{\left\|\rho_{m,h}\right\|_{H_{m}}\}_{h}$ is precompact in a subset of $C\left([0,T]\right)$, i.e., $C^{0,1}\left([0,T];H_{m}\right)$ $\hookrightarrow C\left([0,T];H_{m}\right)$.
Due to Chebyshev's inequality, we get
\begin{equation}
\mathcal{L}\left[\rho_{m,h}\right]\left(B_{\rho,m}^c\right)=\mathbb{P}\left[\left\{\left\|\rho_{m,h}\right\|_{C_t^{0,1} H_m}>L\right\}\right] \leq \frac{1}{L^r} \mathbb{E}\left[\left\|\rho_{m,h}\right\|_{C_t^{0,1} H_m}^r \right] \leq \frac{C}{L^r},
\end{equation}
where $C$ is uniformly in $h$. Choosing $L$ sufficiently large yields the tightness of $\mathcal{L}\left[\rho_{m,h}\right]$. Similarly,
\begin{equation}
B_{\u,m}=\left\{\left.\u_{m,h} \in C^{0,\beta}\left([0, T] ; H_m\right) \right|\|\u_{m,h}\|_{C_t^{0,\beta} H_m} \leq L\right\},
\end{equation}
by Chebyshev's inequality,
\begin{equation}
\mathcal{L}\left[\u_{m,h}\right]\left(B_{\u,m}^c\right)=\mathbb{P}\left[\left\{\left\|\u_{m,h}\right\|_{C_t^{0,\beta} H_m}>L\right\}\right] \leq \frac{1}{L^r} \mathbb{E}\left[\left\|\u_{m,h}\right\|_{C_t^{0,\beta} H_m}^r \right] \leq \frac{C}{L^r},
\end{equation}
which shows the tightness of $\mathcal{L}\left[\u_{m,h}\right]$.
 $\mathcal{L}[W]$ is tight on $\mathcal{X}_W$ since it is a Radon measure on a Polish space.
To conclude, for any given small $\iota>0$, there exist compact sets $K_{\varrho} \subset \mathcal{X}_{\varrho}^{1}, K_{\mathbf{u}} \subset$ $\mathcal{X}_{\mathbf{u}}^{1}$, $K_W \subset \mathcal{X}_W^{1}$ such that
\begin{equation}
\mathcal{L}\left[\rho_{m,h}\right]\left(K_{\rho}\right) \geq 1-\frac{\iota}{3}, \quad \mathcal{L}\left[\mathbf{u}_{m,h}\right]\left(K_{\mathbf{u}}\right) \geq 1-\frac{\iota}{3}, \quad \mathcal{L}[W]\left(K_W\right) \geq 1-\frac{\iota}{3}.
\end{equation}
Applying de Morgan's law $A \cap B=\left(A^c \cup B^c\right)^c$, we get
\begin{equation}
\begin{aligned}
& \mathcal{L}\left[\rho_{m,h}, \mathbf{u}_{m,h}, W\right]\left(K_{\rho} \times K_{\mathbf{u}} \times K_W\right)=\mathbb{P}\left[\left\{\rho_{m,h} \in K_{\rho}, \mathbf{u}_{m,h} \in K_{\mathbf{u}}, W \in K_W\right\}\right] \\
& \quad=1-\mathbb{P}\left[\left\{\rho_{m,h} \notin K_{\rho}\right\} \cup\left\{\mathbf{u}_{m,h} \notin K_{\mathbf{u}}\right\} \cup\left\{W \notin K_W\right\}\right] \\
& \quad \geq 1-\mathbb{P}\left[\left\{\rho_{m,h} \notin K_{\rho}\right\}\right]-\mathbb{P}\left[\left\{\mathbf{u}_{m,h} \notin K_{\mathbf{u}}\right\}\right]-\mathbb{P}\left[\left\{W \notin K_W\right\}\right] \geq 1-\iota,
\end{aligned}
\end{equation}
this gives the tightness of $\mathcal{L}[\rho_{m,h},\mathbf{u}_{m,h},W]$.

Based on the tightness of $\mathcal{L}[\rho_{m,h},\mathbf{u}_{m,h},W]$ and Jakubowski-Skorokhod's representation theorem (see \ref{Skorokhod thm} in Appendix \ref{appendix 2nd part}), we have the following convergence theorem.
\begin{proposition}  There exist $\left\{\bar{\rho}_{m,h},\bar{\mathbf{u}}_{m,h}, \bar{W}_{m,h}\right\}$ and $\left\{\rho_{m,R}, \mathbf{u}_{m,R}, W_{m,R}\right\}$, two families of $\mathcal{X}_{1}$-valued Borel measurable random variables, defined on a complete probability space $\left(\bar{\Omega}, \bar{\mathcal{F}}, \mathbb{\bar{P}}\right)$, such that
\begin{enumerate}
  \item $\mathcal{L}\left[\bar{\rho}_{m,h}, \bar{\mathbf{u}}_{m,h}, \bar{W}_{m,h}\right]$ on $\mathcal{X}_{1}$ is given by $\mathcal{L}\left[\rho_{m,h}, \mathbf{u}_{m,h}, W\right]$;
  \item  $\mathcal{L}\left[\rho_{m,R}, \mathbf{u}_{m,R}, W_{m,R}\right]$ on $\mathcal{X}_{1}$ is a Radon measure;
  \item $\left\{\bar{\rho}_{m,h}, \bar{\mathbf{u}}_{m,h}, \bar{W}_{m,h}\right\}$ converges to $\left\{\rho_{m,R}, \mathbf{u}_{m,R}, W_{m,R}\right\}$ $\bar{\mathbb{P}}$ {\rm a.s.} in the topology of $\mathcal{X}_{1}$, i.e.,
\begin{equation}
\begin{aligned}
\bar{\rho}_{m,h} \rightarrow \rho_{m,R} & \text { in } C\left([0,T];H_{m,R}\right) \quad \bar{\mathbb{P}} \text{ {\rm a.s.}};\\
\bar{\mathbf{u}}_{m,h} \rightarrow \mathbf{u}_{m,R} & \text { in } C^{0,\nu}\left([0, T] ; H_{m,R}\right) \quad\bar{\mathbb{P}}\text{ {\rm a.s.}}; \\
\bar{W}_{m,h} \rightarrow W_{m,R} & \text { in } C\left([0, T];\mathcal{H}\right) \quad \bar{\mathbb{P}}\text{ {\rm a.s.}}\\
\end{aligned}
\end{equation}
\end{enumerate}
\end{proposition}
Here we omit the proof which can be obtained by using similar arguments as in \cite{BreitFeireislHofmanova-book2018}.

\begin{proposition} The processes $\rho_{m,R}$ and $\u_{m,R}$ satisfy
\begin{equation}\label{m layer mass equ}
\Pi_{m}\left(\left(\rho_{m,R}\right)_{t}\right)+\Pi_{m}\left(\operatorname {div}\left(\rho_{m,R}\left[\u_{m,R}\right]_{R}\right)\right)= \Pi_{m}\left(\eps \triangle \rho_{m,R}\right)\\
\end{equation}
in distribution and $\bar{\mathbb{P}}$ {\rm a.s.}
\end{proposition}
\begin{proof}
Firstly, the new processes hold $\mathcal{L}[\bar{\u}_{m, h}]=\mathcal{L}[\u_{m, h}],\mathcal{L}[\bar{\rho}_{m, h}]=\mathcal{L}[\rho_{m, h}]$, hence $\bar{\rho}_{m, h}$ and $\bar{\u}_{m h}$ still satisfy the continuity equation $\bar{\mathbb{P}}$ a.s., namely,
\begin{align}\label{mlayer mass equ distribution1}
&\int_{0}^{t}\int_{\O}\Pi_{m}\left(\left(\bar{\rho}_{m, h}\right)_{t}+\odiv\left(\bar{\rho}_{m, h} \left[\bar{\mathbf{u}}_{m, h}\right]_{R}\right)\right)\varphi(t)\psi(x)\d x\d s\\
=&\int_{0}^{t}\int_{\O} \Pi_{m}\left(\eps \triangle \bar{\rho}_{m, h}\right)\varphi(t)\psi(x)\d x \d s,\quad t \in[0, T],\notag
\end{align}
 for all $\varphi\in C_{c}^{\infty}\left([0,T)\right)$, and all $\psi(x)\in C^{\infty}\left(\O\right)$, $\bar{\mathbb{P}}$ a.s. 
For $\int_{\O}\Pi_{m}\left(\bar{\rho}_{m,h}(t)\right) \varphi(t)\psi(x) \mathrm{d}x$, $\bar{\rho}_{m,h} \rightarrow \rho_{m,R} \text { in } C\left([0,T];H_{m}\right)$, $ \bar{\mathbb{P}}$ a.s., therefore $\Pi_{m}\left(\bar{\rho}_{m, h}(t,x)\right)\varphi(t)\psi(x)$ is uniformly integrable in $L^{1}(\Omega\times[0,T]\times\O)$, by Vitali's convergence theorem, we have
 \begin{align}\int_{\O}\Pi_{m}\left(\bar{\rho}_{m, h}(t,x)\right) \varphi(t)\psi(x) \mathrm{d}x\rightarrow\int_{\O}\Pi_{m}\left(\rho_{m,R}(t,x)\right)\varphi(t)\psi(x) \mathrm{d}x\quad\bar{\mathbb{P}}\text{ a.s.}
 \end{align}
Similarly, we have the convergence of other integrals in \eqref{mlayer mass equ distribution1}. So we have
 \begin{align}\label{mlayer mass equ distribution3}
&\varphi(t) \int_{\O}\Pi_{m}\left(\rho_{m,R}(t,x)\right) \psi(x) \mathrm{d}x- \varphi(0)\int_{\O}\Pi_{m}\left[\rho_{m,R}(0,x)\right]\psi(x)\mathrm{d}x\notag\\
&-\int_{0}^{t} \varphi_{s}(s) \int_{\O}\Pi_{m}\left(\rho_{m,R}\right)(s,x)\psi(x) \d x\d s \\
&+\int_{0}^{t}\varphi(s) \int_{\O}\Pi_{m}\left[\operatorname{div}(\rho_{m,R}\left[\u_{m,R}\right]_{R})\right]\psi(x)\d x\d s \notag \\
=&\int_{0}^{t}\varphi(s)\int_{\O} \Pi_{m}\left(\eps \triangle \rho_{m,R}\right)\psi(x)\d x\d s, \notag
\end{align}
for all $\varphi\in C_{c}^{\infty}\left([0,T)\right)$, and all $\psi(x)\in C^{\infty}\left(\O\right)$, and $t \in[0, T), \quad \bar{\mathbb{P}}$ a.s. \hfill $\square$
\end{proof}
For simplicity, we give the new notations here:
\begin{equation}
\begin{aligned}
\left\|\cdot\right\|_{L^{p}(0,T; L^{q}\left( \O \right))}=\left\|\cdot\right\|_{L_{t}^{p} L_{x}^{q}}, \text{ for } 1\ls p\ls \infty, 1\ls q\ls \infty,\\
\left\|\cdot\right\|_{L^{p}(0,T;H^{q}\left( \O \right))}=\left\|\cdot\right\|_{L_{t}^{p} H_{x}^{q}}, \text{ for } 1\ls p\ls \infty, 1\ls q< \infty.\\
\end{aligned}
\end{equation}

To show that the limits $\rho_{m}$ and $\mathbf{u}_{m}$ satisfy the approximated momentum equation, we first generalize the convergence theorem of stochastic integral \cite{DEBUSSCHE20111123} from ``in probability'' to ``$\bar{\mathbb{P}}$ a.s.'' as follows.
\begin{lemma}\label{lemma sto int convergence}
If $\mathbf{G}_{k}\ra \mathbf{G}$ in $ L^{2}\left([0,T] \times \O \right)$ $\mathbb{P}$ {\rm a.s.}, $W_{k}\ra W$ in $C\left([0,T]; \mathcal{H}\right)$ $\mathbb{P}$ {\rm a.s.}, then
 \begin{equation}
  \int_{0}^{t} \int_{\O} \mathbf{G}_{k}\d x \d W_{k} \ra \int_{0}^{t} \int_{\O} \mathbf{G}\d x \d W \quad  \mathbb{P} \text{ {\rm a.s.}}\\
  \end{equation}
\end{lemma}
\begin{proof}
Since $W_{k}\ra W$ as $k \ra \infty$, so there exists a constant $k_{1}=\max\{1, K_{1}\}>0$, such that for $k>k_{1}$, $\left|W_{k}-W\right|\ls 1 $. And there exists a constant $k_{2}=\max\{2, K_{2}\}>k_{1}$, such that for $k>k_{2}$, $\left|W_{k}-W\right|\ls \frac{1}{k_{1}^{2}}$. All the time all the way, there exists $k_{n+1}=\max\{n+1, K_{n+1}\}>k_{n}$, when $k>k_{n+1}$, it holds $|W_{k}-W|< \frac{1}{k_{n}^{2}}$.
\begin{align}
&\left|\int_{0}^{t} \int_{\O} \mathbf{G}_{k}\d x \d W_{k}-\int_{0}^{t} \int_{\O} \mathbf{G}\d x \d W\right|\notag\\
= &\left|\int_{0}^{t} \int_{\O} \mathbf{G}_{k}\d x \d W_{k}-\int_{0}^{t} \int_{\O} \mathbf{G}\d x \d W_{k}+\int_{0}^{t} \int_{\O} \mathbf{G}\d x \d W_{k}-\int_{0}^{t} \int_{\O} \mathbf{G}\d x \d W\right|\\
\ls & \left|\int_{0}^{t} \int_{\O} \left(\mathbf{G}_{k}-\mathbf{G}\right)\d x \d W_{k}\right|+\left|\int_{0}^{t} \int_{\O} \mathbf{G}\d x \d W_{k}-\int_{0}^{t} \int_{\O} \mathbf{G}\d x \d W\right|.\notag
\end{align}
The first term in the right hand side of the above formula goes to $0$, $\mathbb{P}$ a.s., as $k\ra +\infty$. Actually, for any $\eps > 0 $, for $k$ large enough, $\left|\mathbf{G}_{k}-\mathbf{G}\right|\ls \eps^{2}$,
 \begin{align}\label{conver of Gk-G Wk}
 & \mathbb{P}\left[\left|\int_{0}^{t} \int_{\O} \left(\mathbf{G}_{k}-\mathbf{G}\right)\d x \d W_{k}\right|\geq \eps \right]\\
 \ls &\frac{\mathbb{E}\left[\left|\int_{0}^{t} \int_{\O} \left(\mathbf{G}_{k}-\mathbf{G}\right)\d x \d W_{k}\right|^{2} \right]}{\eps^{2}}
 \ls  \frac{\mathbb{E}\left[\int_{0}^{t}\left|\int_{\O} \left(\mathbf{G}_{k}-\mathbf{G}\right)\d x\right|^{2}\d t\right]}{\eps^{2}}
 \ls   \eps^{2}.\notag
 \end{align}
Now we focus on the second term. We define a mollification $\bar{\mathbf{G}}=\frac{1}{\iota}\int_0^t \exp^{-\frac{t-s}{\iota}}\mathbf{G}(s)\d s$, here $\iota\in \left\{1,\frac{1}{k_{1}},\frac{1}{k_{2}},\cdots,\frac{1}{k_{n}},\cdots\right\}$. By Heine's theorem, $\bar{\mathbf{G}}\ra \mathbf{G}$ still holds as $\iota\ra 0$. We estimate
\begin{align}\label{GdWk-GdW}
&\left|\int_{0}^{t} \int_{\O} \mathbf{G}\d x \d W_{k}-\int_{0}^{t} \int_{\O} \mathbf{G}\d x \d W\right|\notag\\
\ls & \left|\int_{0}^{t} \int_{\O} \left(\mathbf{G}-\bar{\mathbf{G}}\right)\d x \d W_{k}+ \int_{0}^{t} \int_{\O} \left(\bar{\mathbf{G}}-\mathbf{G}\right)\d x \d W\right.\notag\\
&\left.+\left( \int_{0}^{t}\int_{\O} \bar{\mathbf{G}}\d x \d W_{k}- \int_{0}^{t}\int_{\O} \bar{\mathbf{G}}\d x \d W\right)\right|\\
\ls & \left|\int_{0}^{t} \int_{\O} \left(\mathbf{G}-\bar{\mathbf{G}}\right)\d x \d W_{k}\right|+\left| \int_{0}^{t} \int_{\O} \left(\bar{\mathbf{G}}-\mathbf{G}\right)\d x \d W\right|\notag\\
&+ \left|\int_{0}^{t}\int_{\O} \bar{\mathbf{G}}\d x \d W_{k}- \int_{0}^{t}\int_{\O} \bar{\mathbf{G}}\d x \d W\right|.\notag
\end{align}
Like \eqref{conver of Gk-G Wk}, the first term and the second term in the right-hand-side of \eqref{GdWk-GdW} will converge to $0$ as $\iota\ra 0$.
For the term $\left|\int_{0}^{t}\int_{\O} \bar{\mathbf{G}}\d x \d W_{k}- \int_{0}^{t}\int_{\O} \bar{\mathbf{G}}\d x \d W\right|$, we calculate
\begin{equation}
\frac{\d \bar{\mathbf{G}}}{\d t}= \frac{1}{\iota}\mathbf{G}(t)- \frac{1}{\iota^{2}}\int_0^t \exp^{-\frac{t-s}{\iota}}\mathbf{G}(s)\d s = \frac{1}{\iota}\mathbf{G}(t)-\frac{1}{\iota}\bar{\mathbf{G}}.\\
\end{equation}
 For the third integral in \eqref{GdWk-GdW}, we obtain
\begin{align}
 &\left|\int_{0}^{t}\int_{\O} \bar{\mathbf{G}}\d x \d W_{k}- \int_{0}^{t}\int_{\O} \bar{\mathbf{G}}\d x \d W\right|\notag\\
 =& \left| \int_{\O} \bar{\mathbf{G}}\d x W_{k} - \int_{\O} \bar{\mathbf{G}}\d x W- \int_{0}^{t}\int_{\O} \frac{\d \bar{\mathbf{G}}}{\d s}W_{k}\d x \d s+\int_{0}^{t}\int_{\O} \frac{\d \bar{\mathbf{G}}}{\d s}W \d x \d s\right|\\
 \ls &\left| \int_{\O} \bar{\mathbf{G}}\d x \left(W_{k}-W\right)\right|+ \left| \int_{0}^{t}\int_{\O}\left(W-W_{k}\right) \frac{\mathbf{G}}{\iota}\d x \d s \right|\notag\\
  & +\left|\int_{0}^{t}\int_{\O}\left(W-W_{k}\right) \frac{\bar{\mathbf{G}}}{\iota}\d x \d s \right|.\notag
\end{align}
by It\^o's formula and the integration by parts with respect to $t$ for the stochastic integral of semi-martingale (see \S \ref{def in app} in Appendix).
 $\left| \int_{\O} \bar{\mathbf{G}}\d x \left(W_{k}-W\right)\right|$ will converge to $0$. In fact, we calculate
\begin{align}
 & \left| \int_{\O} \bar{\mathbf{G}}\d x \left(W_{k}-W\right)\right|= \left| \int_{\O} \sum\limits_{k=1}^{+\infty}\bar{\mathbf{G}}_{i}e_{i}\d x \sum\limits_{i=1}^{+\infty}\left(W_{k}-W\right)_{i}e_{i}\right|\notag\\
=& \left| \sum\limits_{i=1}^{+\infty} \int_{\O} \bar{\mathbf{G}}_{i}\d x \left(W_{k}-W\right)_{i}\right|\ls \sum\limits_{i=1}^{+\infty}\left|\int_{\O} \bar{\mathbf{G}}_{i}\d x\right| \left|\left(W_{k}-W\right)_{i}\right|\\
\ls &\sum\limits_{i=1}^{+\infty}\left(\int_{\O} \left|\bar{\mathbf{G}}_{i}\right|^{2}\d x\right)^{\frac{1}{2}} \left|\left(W_{k}-W\right)_{i}\right|
\ls \|\bar{\mathbf{G}}\|_{L_{t}^{\infty}L_{x}^{2}}\sup_{i}\left|\left(W_{k}-W\right)_{i}\right|.\notag
\end{align}
  Since $W_{k}\ra W$ as $k \ra \infty$, so for any $\iota>0$, $k>k_{n+1}$,
\begin{align}
\left| \int_{0}^{t}\int_{\O}\left(W-W_{k}\right) \frac{\bar{\mathbf{G}}}{\iota}\d x \d s \right|\ls  k_{n}\frac{1}{k_{n}^{2}} \left| \int_{0}^{t}\int_{\O} \bar{\mathbf{G}}\d x \d s \right|\ra 0 \text{ as }k_{n} \ra +\infty .
\end{align}
 So does the term $\left| \int_{0}^{t}\int_{\O}\left(W-W_{k}\right) \frac{\mathbf{G}}{\iota}\d x \d s \right|$. \hfill $\square$
\end{proof}
Now we have the following proposition for the approximated momentum equation.
\begin{proposition} The processes $\left(\rho_{m,R}, \u_{m,R}, W_{m,R}\right)$ satisfy
\begin{align}
& \d \Pi_{m}\left(\rho_{m,R} \u_{m,R}\right)+\Pi_{m}\left(\operatorname{div}(\rho_{m,R} [\u_{m,R}]_{R} \otimes \u_{m,R})+\nabla \left(a\rho_{m,R}^{\gamma}\right)\chi_R(\left\| \u_{m,R}\right\|_{H_{m}})\right)\d t\notag \\
&-\Pi_{m}\left(\operatorname{div}(\rho_{m,R}\mathbb{D} \u_{m,R})\chi_R(\left\| \u_{m,R}\right\|_{H_{m}})\right)\d t\notag\\
=&\Pi_{m}\left((-r_{0} |\u_{m,R}|^{2}[\u_{m,R}]_{R}-r_{1} \rho_{m,R}|\u_{m,R}|^{2}[\u_{m,R}]_{R}-r_{2} [\u_{m,R}]_{R})\right)\d t \notag \\
&+\Pi_{m}\left(\eps \nabla \rho_{m,R}\nabla \u_{m,R}\chi_R(\left\| \u_{m,R}\right\|_{H_{m}})\right)\d t -\Pi_{m}\left(\eps \triangle^{2}\u_{m,R}\chi_R(\left\| \u_{m,R}\right\|_{H_{m}})\right)\d t\\
&+\Pi_{m}\left(\frac{11}{10}\eta \nabla \rho_{m,R}^{-10}\chi_R(\left\| \u_{m,R}\right\|_{H_{m}})\right)\d t + \Pi_{m}\left(\delta \rho_{m,R} \nabla \triangle^{9}\rho_{m,R}\chi_R(\left\| \u_{m,R}\right\|_{H_{m}})\right)\d t \notag\\
& +\Pi_{m}\left(\kappa \rho_{m,R} \left(\nabla\left(\frac{\triangle \sqrt{\rho_{m,R}}}{\sqrt{\rho_{m,R}}}\right)\right)\chi_R\left(\left\| \u_{m,R}\right\|_{H_{m}}\right)\right) \notag\\
&+\Pi_{m}\left(\rho_{m,R} \mathbb{F}(\rho_{m,R},\u_{m,R})\chi_R\left(\left\| \u_{m,R}\right\|_{H_{m}}\right)\d W_{m,R}\right), \notag
\end{align}
$\bar{\mathbb{P}}$ {\rm a.s.} in distribution sense, for $\forall t \in[0, T]$.
 \end{proposition}
\begin{proof}
We mainly need to show the convergence of stochastic term in this proof, because the other terms in momentum equation converges in distribution coming from the strong convergence of variables and Vitali's convergence theorem. Similarly as in \cite{BreitFeireislHofmanova-book2018}, by the strong convergence of $\bar{\rho}_{m,h}$ and $\bar{\u}_{m,h}$, and the continuity of $\Pi_{m}\left(\mathbf{F}_{k}\right)$, it holds
\begin{equation}
\begin{aligned}
\Pi_{m}\left(\bar{\rho}_{m,h}\Pi_{m}\left(\mathbf{F}_{k}\left(\bar{\rho}_{m,h},\bar{\mathbb{\u}}_{m,h}\right)\right)\right) \rightarrow \Pi_{m}\left(\rho_{m,R} \Pi_{m}\left(\mathbf{F}_{k}(\rho_{m,R}, \u_{m,R})\right)\right),
\end{aligned}
\end{equation}
$\bar{\mathbb{P}}$ a.s. in $L^{2}\left(\Omega; L^{2}\left([0, T]; L^{2}\left(\O\right)\right)\right)$. Combining this with the convergence of $\bar{W}_{m,h}$, by lemma \ref{lemma sto int convergence}, for all $\varphi(t)\in C_{c}^{\infty}\left([0,T)\right)$, and all $\psi(x)\in C^{\infty}\left(\O\right)$, we have
\begin{align}
&\int_{0}^{t}\int_{\O}\Pi_{m}\left(\bar{\rho}_{m,h}\Pi_{m}\left(\mathbb{F}\left(\bar{\rho}_{m,h},\bar{\mathbf{u}}_{m,h}\right)\right)\right)\varphi(t)\psi(x)\d \d x\bar{W}_{m,h}\\
&\rightarrow \int_{0}^{t}\int_{\O}\Pi_{m}\left(\rho_{m,R} \Pi_{m}\left(\mathbb{F}\left(\rho_{m,R},\u_{m,R}\right)\right)\right)\varphi(t)\psi(x)\d x \d W_{m,R}, \quad t\in [0,T], \quad \bar{\mathbb{P}} \text{ a.s.}\notag \qquad \quad \square
\end{align}
\end{proof}
\subsubsection{Energy balance and energy estimates for the regularized system}

 We first derive the energy estimates uniformly in $R$ for regularized system \eqref{Galerkin approximation scheme}. We use the general notation $\rho$ and $\u$ to take place of the notation $\rho_{m,R}$ and $\u_{m,R}$ in the process of the a priori energy estimates.

We apply It\^o's formula to the functional
\begin{equation}
\begin{aligned}
g:L^{\gamma}\left( \O \right)\times H_{m}^{*}\ra \mathds{R}, \quad
(\rho, \mathbf{q})\mapsto \frac{1}{2} \mathbf{q} \mathscr{M}^{-1}\left(\rho\right)\mathbf{q},
\end{aligned}
\end{equation}
 where $\mathbf{q}\in H_{m}^{*}$, and $\mathscr{M}^{-1}\left(\rho\right)$ maps $H_{m}^{*}$ to $H_{m}$. Then it holds 
\begin{align}
&\pt_{\mathbf{q}}g\left(\rho, \mathbf{q}\right)=\mathscr{M}^{-1}\left(\rho\right)\mathbf{q}\in H_{m}, \quad \pt^{2}_{\mathbf{q}}g\left(\rho, \mathbf{q}\right)=\mathscr{M}^{-1}\left(\rho\right),\\
&\pt_{\rho}g\left(\rho,\mathbf{q}\right)=-\frac{1}{2} \mathscr{M}^{-1}\left(\rho\right)\left(\mathbf{q}\cdot \mathscr{M}^{-1}\left(\rho\right)\mathbf{q}\right)\in \mathds{R}.\notag %
\end{align}
Therefore, for $g=\frac{1}{2}\rho |\mathbf{u}|^{2}=\frac{1}{2}\frac{|\mathbf{q}|^{2}}{\rho}$, according to It\^o's formula (see \S\ref{appendix 2nd part} in Appendix),
we deduce
\begin{align}
&\mathrm{d}\left(\int_{\O}\frac{1}{2}\rho |\u|^{2}\d x\right)=\frac{1}{2} \int_{\O}|\u|^{2} \mathrm{d} \rho \mathrm{d}x + \int_{\O}\u\cdot \d\left(\rho \mathbf{u}\right)\mathrm{d}x +\left(\frac{1}{2}\int_{\O}\rho\left|\mathbb{F}(\rho,\mathbf{u})\right|^{2}\mathrm{d}x\right)\mathrm{d}t.
\end{align}
Here and hereafter $"\d"$ outside the integral in the above equality means differentiation with respect to time $t$.
Substituting expression of $\d \left(\rho \u\right)$ of momentum equation into the above formula, it holds
\begin{align}
& \mathrm{d}\left(\int_{\O}\frac{1}{2}\rho |\u|^{2}\d x\right) \notag\\
=&-\frac{1}{2}\int_{\O}|\u|^{2} \mathrm{d}\rho \mathrm{d}x +\left(\int_{\O}\rho \mathbb{F}(\rho,\mathbf{u})\cdot \mathbf{u}\mathrm{d}x\right)\mathrm{d}W+\left(\frac{1}{2}\int_{\O}\rho \left|\mathbb{F}(\rho,\mathbf{u})\right|^{2}\mathrm{d}x\right)\mathrm{d}t\notag\\
&- \left(\int_{\O}\mathbf{u}\cdot \left(\odiv\left(\rho\mathbf{u}\otimes \mathbf{u}\right)+\nabla p\left(\rho\right)\right)\mathrm{d}x-\int_{\O}\eps \nabla \rho\nabla \mathbf{u}\cdot \mathbf{u}\mathrm{d}x-\eps\int_{\O}\left|\triangle \mathbf{u}\right|^{2}\mathrm{d}x\right.\\
&\quad \left.+ \int_{\O}\odiv\left(\rho\mathbb{D}\mathbf{u}\right)\mathbf{u}\mathrm{d}x+ \int_{\O} \frac{11}{10}\eta\nabla\rho^{-10}\cdot \mathbf{u}\mathrm{d}x+\delta \int_{\O} \rho \nabla\triangle^{9}\rho\cdot\mathbf{u}\mathrm{d}x\right.\notag\\
&\quad \left.-\int_{\O} r_{0}|\u|^{2}\mathbf{u}\cdot\mathbf{u}\d x-\int_{\O} r_{1}\rho|\u|^{2}\u\cdot \mathbf{u}\mathrm{d}x-\int_{\O} r_{2}\mathbf{u}\cdot\mathbf{u}\d x+\kappa \int_{\O} \rho \nabla \left(\frac{\triangle \sqrt{\rho}}{\sqrt{\rho}}\right)\cdot \mathbf{u}\mathrm{d}x\right)\mathrm{d}t.\notag
\end{align}
Since $\left(\int_{\O}\mathbf{u}\cdot \odiv\left(\rho\u\otimes\u\right)\d x\right)\d t=\left(\int_{\O}\odiv\left(\rho\mathbf{u}\right)\frac{1}{2}|\u|^{2}\d x\right)\d t,$
 we deduce that
\begin{equation}
-\frac{1}{2}\int_{\O}|\u|^{2} \mathrm{d}\rho \d x - \left(\int_{\O}\odiv\left(\rho\mathbf{u}\otimes \mathbf{u}\right)\cdot \mathbf{u}\d x\right)\d t - \left(\int_{\O}\eps \nabla \rho\nabla \mathbf{u}\cdot \mathbf{u}\d x\right)\d t=0.
\end{equation}
Multiplying $-\eta \rho_{\eta}^{-11}$ to the mass conservation equation, and integrating on $\O$, we have
\begin{equation}
\begin{array}{l}
\left(\int_{\O}\frac{11}{10}\eta \nabla \rho^{-10}\cdot \u \d x\right)\d t=-\d \left(\int_{\O}\frac{\eta}{10}\rho^{-10}\d x\right)-\eps\eta\frac{11}{25}\left(\int_{\O}|\nabla \rho^{-5}|^{2}\d x\right)\d t.
\end{array}
\end{equation}
Similarly, we have
\begin{equation}
-\left(\int_{\O} \nabla p\cdot\mathbf{u} \d x\right)\d t = -\d \left(\int_{\O} \frac{p}{\gamma-1}\d x\right) - \eps\frac{4a}{\gamma}\left(\int_{\O}|\nabla (\rho^{\frac{\gamma}{2}})|^{2}\d x\right)\d t,\\
\end{equation}
\begin{equation}
-\left(\int_{\O} \delta \rho \nabla\triangle^{9}\rho\cdot \mathbf{u} \d x\right)\d t = - \d \left(\int_{\O} \frac{\delta}{2}|\nabla \triangle^{4}\rho|^{2} \d x\right) - \eps\delta \left(\int_{\O}|\triangle^{5} \rho|^{2}\d x\right)\d t.\\
\end{equation}
For the term $\kappa \rho\nabla \left(\frac{\triangle \sqrt{\rho}}{\sqrt{\rho}}\right)$, we use the equality relation $2\rho\nabla \left(\frac{\triangle \sqrt{\rho}}{\sqrt{\rho}}\right)=\odiv \left(\rho \nabla^{2}\ln\rho\right)$. Therefore, we have
\begin{equation}
\int_{\O}\kappa \rho\nabla \left(\frac{\triangle \sqrt{\rho}}{\sqrt{\rho}}\right)\cdot \u \d x 
=-\kappa\frac{\d}{\d t}\int_{\O}|\nabla\sqrt{\rho}|^{2}\d x- \kappa\eps\int_{\O}\rho|\nabla^{2}\ln\rho|^{2}\d x.
\end{equation}
To conclude, we obtain the energy balance
\begin{align}\label{energy balance}
&\int_{\O}\left(\frac{1}{2}\rho|\mathbf{u}|^{2}+\frac{a}{\gamma}\rho^{\gamma} +\frac{\eta}{10}\rho^{-10}+\kappa|\nabla\sqrt{\rho}|^{2}+\frac{\delta}{2}|\nabla \triangle^{4} \rho|^{2}\right)\d x\notag\\
&+\eps \frac{4a}{\gamma}\int_{0}^{t}\int_{\O}|\nabla \rho^{\frac{\gamma}{2}}|^{2}\d x\d s+\eps\eta \frac{11}{25} \int_{0}^{t}\int_{\O}|\nabla \rho^{-5}|^{2}\d x\d s\notag\\
&+\eps \delta \int_{0}^{t}\int_{\O}|\triangle^{5} \rho|^{4}\d x\d s+\frac{\eps \kappa}{2} \int_{\O} \rho|\nabla^{2}\ln\rho|^{2}\d x\d t+\int_{\O}\rho|\mathbb{D}\mathbf{u}|^{2}\d x\d s \\
&+\eps\int_{0}^{t}\int_{\O}\left|\triangle \mathbf{u}\right|^{2}\d x\d s + r_{0}\int_{0}^{t}\int_{\O}|\mathbf{u}|^{4}\d x\d s+r_{1}\int_{\O}\rho |\mathbf{u}|^{4}\d x\d s+ r_{2}\int_{0}^{t}\int_{\O}|\mathbf{u}|^{2}\d x\d s\notag\\
&=\int_{\O}\left(\frac{1}{2}\frac{|\mathbf{q}_{0}|^{2}}{\rho_{0}}+\frac{a}{\gamma}\rho_{0}^{\gamma} +\frac{\eta}{10}\rho_{0}^{-10}+\kappa|\nabla\sqrt{\rho_{0}}|^{2}+\frac{\delta}{2}|\nabla \triangle^{4} \rho_{0}|^{2}\right)\d x\notag\\
&\quad +\frac{1}{2}\int_{0}^{t}\int_{\O}\rho|\mathbb{F}_{k}(\rho,\mathbf{u})|^{2}\d x\d s+\frac{1}{2}\int_{0}^{t}\int_{\O}\rho\mathbb{F}_{k}(\rho,\mathbf{u})\cdot \mathbf{u}\d x\d W.\notag
\end{align}
The $C_{r}$ inequality states that: for real-valued $r>1$, $a$ and $b$, by the property of convex functions, it holds $  |a+b|^{r}\ls \max (1,2^{r-1})(|a|^{r}+|b|^{r})$. Hence to estimate $r-$th expectation of the right-hand side, for some $r\geqslant 1$, we only need to consider
\begin{align}
\mathbb{E}\left[\left|\int_{0}^{t}\int_{\O} \rho|\mathbb{F}(\rho,\mathbf{u})|^{2} \d x\d s\right|^{r}\right]
=\mathbb{E}\left[\left|\sum\limits_{k=1}^{+\infty}\int_{0}^{t}\int_{\O} \rho\left|\mathbf{F}_{k}(\rho,\mathbf{u})\right|^{2} \d x\d s\right|^{r}\right]
\end{align}
 and
\begin{align}\label{the second term in the right-hand side of energy balance}
&\mathbb{E}\left[\left|\int_{0}^{t}\int_{\O} \rho \mathbb{F}\left(\rho,\mathbf{u}\right)\cdot \mathbf{u} \d x\d W\right|^{r}\right]
\ls C\mathbb{E}\left[\left|\int_{0}^{t}\left|\int_{\O} \rho  \sum\limits_{k=1}^{+\infty}\mathbf{F}_{k}\left(\rho,\mathbf{u}\right)e_{k}\cdot \mathbf{u} \d x\right|^{2}\d s\right|^{\frac{r}{2}}\right]\\
=& C\mathbb{E}\left[\left|\sum\limits_{k=1}^{+\infty}\int_{0}^{t}\left| \int_{\O} \rho \mathbf{F}_{k}\left(\rho,\mathbf{u}\right)\cdot \mathbf{u} \d x\right|^{2}\d s\right|^{\frac{r}{2}}\right]\notag
\end{align}
respectively, where $C$ depends on $r$.

For any function $\xi$, $1<\gamma<3$, $\xi^{\frac{1}{\gamma}}<(1+\xi)^{\frac{1}{\gamma}}<(1+\xi)$. By the property of $\mathbf{F}_{k}$, we do the following energy estimates,
\begin{align}
&\sum\limits_{k=1}^{+\infty}\int_{0}^{t}\int_{\O}\rho\left|\Pi_{m}\left(\mathbf{F}_{k}(\rho,\u)\right)\right|^{2}\d x\d t \ls \sum\limits_{k=1}^{+\infty} \sum\limits_{k=1}^{+\infty}f^{2}_{k}\int_{0}^{t}\int_{\O} \frac{\left(\left|\rho\right|+ \left|\rho\u\right|\right)^{2}}{\rho} \d x\d s \notag\\
\ls &C \left\| \rho\right\|_{L_t^{1}L_x^{1}}+C \int_{0}^{t}\int_{\O} \rho\left|\u\right|^{2}\d x\d s \ls C \int_{0}^{t}\left(\int_{\O} \rho ^{\gamma}\d x\right)^{\frac{1}{\gamma}} \d s+C \int_{0}^{t}\int_{\O} \rho\left|\u\right|^{2}\d x\d s\\
\ls &C \int_{0}^{t} \left(\int_{\O}\rho ^{\gamma}\d x+1\right)\d s+C \int_{0}^{t}\int_{\O} \rho\left|\u\right|^{2}\d x\d s , \quad t\in\left[0,T\right], \notag
\end{align}
uniformly in $R$ and $m$, where $C$ depends on $  \sum\limits_{k=1}^{+\infty}f^{2}_{k}$. By Jensen's inequality, we have
\begin{align}\label{Jensen used example}
&\quad\mathbb{E}\left[\left|\sum\limits_{k=1}^{+\infty}\int_{0}^{t}\int_{\O} \rho\left|\Pi_{m}\left(\mathbf{F}_{k}(\rho,\u)\right)\right|^{2} \d x\d s\right|^{r}\right] \notag\\
&\ls C\mathbb{E}\left[\left(\int_{0}^{t}\left(\int_{\O}\left(\rho\left|\u\right|^{2}+\rho^{\gamma}\right)(s,x)\d x+1\right)\d s \right)^{r}\right]\\
&\ls C\mathbb{E}\left[t^{r}\int_{0}^{1}\left(\int_{\O}\left(\rho\left|\u\right|^{2}+\rho^{\gamma}\right)(ts,x)\d x+1\right)^{r} \d s\right]\notag\\
&\ls C \left(t^{r-1}\mathbb{E}\left[\int_{0}^{t}\left(\int_{\O}\left(\rho\left|\u\right|^{2}+\rho^{\gamma}\right)(s,x)\d x\right)^{r} \d s\right]+t^{r}\right),\quad t\in [0,T],  \notag
\end{align}
for some $r\geqslant 1$, where $C$ depends on $r$ and $ \sum\limits_{k=1}^{+\infty}f^{2}_{k}$.
For \eqref{the second term in the right-hand side of energy balance},
 $  \frac{1}{2\gamma}+\frac{1}{2}+\frac{1}{q}=1$, it holds
\begin{align}
 \sum\limits_{k=1}^{+\infty}\left|\int_{\O} \rho \Pi_{m}\left(\mathbf{F}_{k}\left(\rho,\u\right)\right)\cdot \u \d x\right|^{2}
\ls C \sum\limits_{k=1}^{+\infty} f_{k}^{2} \left|\left\|\sqrt{\rho}\right\|_{L_{x}^{2\gamma}}\left\|\sqrt{\rho}\u\right\|_{L_{x}^{2}}+ \int_{\O}\rho\left|\u\right|^{2}\d x \right|^{2},
\end{align}
\begin{align}
&\left| \left\|\sqrt{\rho}\right\|_{L_{x}^{2\gamma}}\left\|\sqrt{\rho}\u\right\|_{L_{x}^{2}}\right|^{2}
\ls C \left|\left\|\sqrt{\rho}\right\|_{L_{x}^{2\gamma}}\left\|\sqrt{\rho}\u\right\|_{L_{x}^{2}}\right|^{2}\\
\ls & C \left(\int_{\O}\rho^{\gamma}\d x\right)^{\frac{1}{\gamma}}\int_{\O}\rho\u^{2}\d x \ls C \left(\int_{\O}\rho^{\gamma}\d x+1\right)\left(\int_{\O}\rho\u^{2}\d x\right),\notag
\end{align}
where $C$ depends on $\sum\limits_{k=1}^{+\infty}f^{2}_{k}$.

Due to the property of convex functions, for some $r \geqslant2$, by Jensen's inequality, we have
\begin{align}\label{estimate of rho F u}
&\quad\mathbb{E}\left[\left|\int_{0}^{t}\int_{\O} \rho \Pi_{m}\left(\mathbb{F}\left(\rho \u\right)\right)\cdot \u\d x\d W\right|^{r}\right]\notag\\
&\ls C \mathbb{E}\left[\left(\int_{0}^{t}\left(\left(\int_{\O}\rho^{\gamma}(s)\d x+\int_{\O}\rho(s)\u^{2}(s)\d x\right)^{2}+1\right) \d s\right)^{\frac{r}{2}}\right] \\
&= C \mathbb{E}\left[\left(t\int_{0}^{1}\left(\int_{\O}\rho^{\gamma}(ts)\d x + \int_{\O}\rho(ts )\u^{2}(ts)\d x\right)^{2}\d s+t \right)^{\frac{r}{2}}\right] \notag\\
&\ls C \left(t^{\frac{r}{2}-1}\int_{0}^{t}\mathbb{E}\left[\left(\int_{\O}\rho^{\gamma}(s)\d x + \int_{\O}\rho(s, x)\u^{2}(s )\d x\right)^{r} \right]\d s+t^{\frac{r}{2}}\right), \notag
\end{align}
uniformly in $R$ and $m$, where $C$ depends on $r$. Now we recover the notation.
Let
\begin{equation}\label{energy estimate1}
e_{m,R}(t)=\int_{\O}\left(\frac{1}{2}\rho_{m,R}|\u_{m,R}|^{2}+\frac{a}{\gamma}\rho_{m,R}^{\gamma} +\frac{\eta}{10}\rho_{m,R}^{-10}+\frac{\kappa}{2}|\nabla\sqrt{\rho_{m,R}}|^{2}+\frac{\delta}{2}|\nabla \triangle^{4} \rho_{m,R}|^{2}\right)\d x. \\
\end{equation}
Combining the above estimates for the stochastic integrals, therefore, we have
\begin{align}
& \mathbb{E}\left[\left(e_{m,R}\left(t\right) + \eps \frac{4a}{\gamma}\int_{0}^{t}\int_{\O}\left|\nabla \rho_{m,R}^{\frac{\gamma}{2}}\right|^{2}\d x\d s+\eps\eta \frac{11}{25} \int_{0}^{t}\int_{\O}\left|\nabla \rho_{m,R}^{-5}\right|^{2}\d x \d s\right.\right.\notag\\
&+\left.\left.\eps\delta \int_{0}^{t}\int_{\O}|\triangle^{5} \rho_{m,R}|^{2}\d x\d s + \eps \kappa \int_{0}^{t}\int_{\O} \rho_{m,R}|\nabla^{2}\ln\rho_{m,R}|^{2}\d x\d s+\eps\int_{\O}\left|\triangle \u_{m,R}\right|^{2}\d x\d s\right.\right.\notag\\
&+\left.\left.\int_{0}^{t}\int_{\O}\rho_{m,R}|\mathbb{D}\u_{m,R}|^{2}\d x\d s+r_{0}\int_{0}^{t}\int_{\O}|\u_{m,R}|^{4}\d x\d s\right.\right.\\
&+\left.\left.r_{1}\int_{0}^{t}\int_{\O}\rho_{m,R} |\u_{m,R}|^{4}\d x\d s+r_{2}\int_{0}^{t}\int_{\O}|\u_{m,R}|^{2}\d x\d s\right)^{r}\right]\notag\\
&\ls C \left( (t^{\frac{r}{2}-1}+t^{r-1})\int_{0}^{t}\mathbb{E}\left[e_{m,R}(s)^{r}\right]\d s+t^{\frac{r}{2}}+t^{r}+\mathbb{E}\left[e_{m,R}(0)^{r}\right]\right).\notag
\end{align}
Hence it holds
\begin{equation}
\quad \mathbb{E} [\left(e_{m,R}(t)\right)^{r}] \ls C \left( (t^{\frac{r}{2}-1}+t^{r-1})\int_{0}^{t}\mathbb{E}\left[e_{m,R}(s)^{r}\right]\d s+t^{\frac{r}{2}}+t^{r}+\mathbb{E}\left[e_{m,R}(0)^{r}\right]\right),
\end{equation}
where $C$ depends on $r$ and $\sum\limits_{k=1}^{+\infty}f^{2}_{k}$.
By Gr\"onwall's inequality, $ t^{\frac{r}{2}}+t^{r}+\mathbb{E}\left[e_{m,R}(0)^{r}\right]$ is non-decreasing with respect to $t$, it holds
\begin{equation}\label{energy estimate2}
\mathbb{E}[\sup\limits_{t\in[0,T]}\left(e_{m,R}(t)\right)^{r}] \ls  C\left(t^{\frac{r}{2}}+t^{r}+\mathbb{E}\left[e_{m,R}(0)^{r}\right]\right)e^{C(t^{\frac{r}{2}}+t^{r})}\ls C(\mathbb{E}\left[e_{m,R}(0)^{r}\right]+1),
\end{equation}
where $C$ depends on $r$ $T$, and $\sum\limits_{k=1}^{+\infty}f^{2}_{k}$, $r \geqslant2$.
Thereupon, we have
\begin{align}\label{energy estimate3}
& \mathbb{E}\left[\left(\eps \frac{4a}{\gamma}\int_{0}^{t}\int_{\O}|\nabla \rho_{m,R}^{\frac{\gamma}{2}}|^{2}\d x\d s+\eps\eta \frac{11}{25} \int_{0}^{t}\int_{\O}|\nabla \rho_{m,R}^{-5}|^{2}\d x\d s\right.\right.\notag\\
&\left.\left.\quad  +\eps\delta \int_{0}^{t}\int_{\O}|\triangle^{5} \rho_{m,R}|^{2}\d x\d s + \eps \kappa \int_{0}^{t}\int_{\O} \rho_{m,R}|\nabla^{2}\ln\rho_{m,R}|^{2}\d x\d s\right.\right.\notag\\
&\left.\left.\quad +\eps\int_{\O}\left|\triangle \u_{m,R}\right|^{2}\d x\d s +\int_{0}^{t}\int_{\O}\rho_{m,R}|\mathbb{D}\u_{m,R}|^{2}\d x\d s\right.\right.\\
&\left.\left.\quad+ r_{0}\int_{0}^{t}\int_{\O}|\u_{m,R}|^{4}\d x\d s + r_{1}\int_{0}^{t}\int_{\O}\rho_{m,R} |\u_{m,R}|^{4}\d x\d s+ r_{2}\int_{0}^{t}\int_{\O}|\u_{m,R}|^{2}\d x\d s \right)^{r}\right]\notag\\
&\ls C(\mathbb{E}\left[e_{m,R}(0)^{r}\right]+1),\notag
\end{align}
where $C$ depends on $r$ $T$, and $\sum\limits_{k=1}^{+\infty}f^{2}_{k}$, $r \geqslant2$.

\subsubsection{The global existence of the classical solutions to the nonlinear regularized Galerkin scheme in $H_{m}$}

We define $\tau_{R}$, a stoping time (see \S \ref{def in app} in Appendix), until which $\left(\rho_{m,R},\u_{m,R}\right)$ is the well-defined in $C\left([0,T];H_{m}\right)\times C^{0,\nu}\left([0,T];H_{m}\right)$,
\begin{equation}
\tau_{R}=\inf\left\{t\in[0,T]\left|\left\|\u_{m,R}\right\|_{H_{m}}>R \right.\right\}.
\end{equation}
For this layer $R\ra +\infty$, we need show that
\begin{equation}
\bar{\mathbb{P}}\left[\left\{\sup\limits_{R\in\mathds{N}}\tau_{R}=T \right\}\right]=1.
\end{equation}
In other words, till time $T$, the solutions will not blow up as $R\ra \infty$.
Due to \eqref{maximum principle}, we only need consider $\left\|\u_{m,R}\right\|_{H_{m}}$ and then the well-posedness of $\rho_{m,R}$ follows.
From \eqref{energy estimate3} we know that $\left\|\triangle \u_{m,R}\right\|_{L_t^{2}L_x^{2}}\ls C$, $\bar{\mathbb{P}}$ a.s., this implies that $\left\|\nabla \u_{m,R}\right\|_{L_t^{2}L_x^{6}}\ls C $, $\bar{\mathbb{P}}$ a.s., where $C$ depends on $r$, $T$, and $\sum\limits_{k=1}^{+\infty}f^{2}_{k}$. Since $\mathbb{E}\left[\left(r_{2}\int_{0}^{t}\int_{\O}|\u_{m,R}|^{2}\d x\d s\right)^{r}\right]\ls C(\mathbb{E}\left[e_{m,R}(0)^{r}\right]+1)$, it holds
\begin{equation}
\left\|\u_{m,R}\right\|_{H_{m}}\ls C, \quad \left\|\odiv \u_{m,R}\right\|_{H_{m}}\ls C,\quad \bar{\mathbb{P}} \quad \text{a.s.}
\end{equation}
This shows that $\tau_{R}=T$ for $R$ large enough.

\subsection{$m$ goes to infinity}

For the space of $\u_{m}$, we try not to use the regular property from $\triangle^{2} \mathbf{u}_{m}$ because  $\eps$ converges to zero finally.

Let $R\ra \infty$ in the energy estimates (\ref{energy estimate1}), (\ref{energy estimate2}) and (\ref{energy estimate3}), then we have
\begin{equation}\label{energy squa rho and energy quarter rho}
\mathbb{E}\left[\left(\eps \kappa\left\|\rho_{m}^{\frac{1}{2}}\right\|_{L_{t}^{2}H^{2}_{x}}\right)^{r}\right] \ls C,\quad
\mathbb{E}\left[\left(\eps \kappa\left\|\nabla\left(\rho_{m}^{\frac{1}{4}}\right)\right\|_{L_t^{4}L_x^{4}}^{2}\right)^{r}\right] \ls C,
\end{equation}
where $C$ depends on $T,r$, $\sum\limits_{k=1}^{+\infty}f^{2}_{k}$ and $\mathbb{E}\left[e_{m}(0)^{r}\right]$.
To prove it, we used the following lemma which is from J\"ungle's original idea \cite{Jungle2010}.
\begin{lemma}\label{Jungle inequality} For any smooth positive function $f(x)$, it holds
\begin{equation}
\begin{aligned}
\int_{\O} f\left|\nabla^{2} \ln f\right|^{2} \d x \gs \frac{1}{7} \int_{\O}\left|\nabla^{2} f^{\frac{1}{2}}\right|^{2} \d x+\frac{1}{8} \int_{\O}\left|\nabla f^{\frac{1}{4}}\right|^{4} \d  x. \\
\end{aligned}
\end{equation}
\end{lemma}

To show the convergence of $\rho_{m}$ and $\u_{m}$, we will proceed in the following four steps.\\

\begin{flushleft}
\textbf{Step 1:} Choose the path space.\\
\end{flushleft}

We will roughly give the estimates for $\rho_{m}$, $\rho_{m} \u_{m}$, and $\rho_{m}^{\frac{1}{2}}$.
\begin{lemma}\label{strong conver of rho}
For some $r\geqslant 2$, there exists a constant $C$ such that
\begin{equation}\label{independent on delta as eps ra to 0}
 \mathbb{E}\left[\left(\left\|(\rho_{m})_{t}\right\|_{L_t^2 L_x^{\frac{3}{2}}}\right)^{r}\right]\ls  C, \quad \mathbb{E}\left[\left(\left\|\left(\rho_{m}^{\frac{1}{2}}\right)_{t}\right\|_{L_t^2 L_x^2}\right)^{r}\right] \ls C,\quad \mathbb{E}\left[\left(\left\|\rho_{m}\u_{m}\right\|_{L_t^2 L_x^2}\right)^{r}\right] \ls C,
\end{equation}
\begin{equation}\label{square root of rhom}
 \mathbb{E}\left[\left(\left\|\rho_{m}^{\frac{1}{2}}\right\|_{L_t^2 H_x^{2}}\right)^{r}\right] \ls C, \quad
 \mathbb{E}\left[\left(\left\|\nabla\left(\rho_{m}^{\frac{1}{4}}\right)\right\|_{L_t^4 L_x^4}\right)^{r}\right] \ls C.
 \quad \mathbb{E}\left[\left(\left\|\rho_{m}^{\gamma}\right\|_{L_t^{\frac{5}{3}} L_x^{\frac{5}{3}}}\right)^{r}\right]
\ls C,
\end{equation}
where $C$ is merely dependent on $T,r,r_{1},\delta,\eps\kappa$, and \eqref{independent on delta as eps ra to 0} will be independent on $\delta$ after $\eps\rightarrow 0$.
\end{lemma}
\begin{proof}
Recall that
\begin{equation}
\begin{aligned}
(\rho_{m})_{t}&=-\rho_{m}\odiv\u_{m}-\nabla\rho_{m}\cdot\u_{m}+\eps\triangle\rho_{m}\\
        &=-\rho_{m}^{\frac{1}{2}}\rho_{m}^{\frac{1}{2}}\odiv\u_{m}-\left(\rho_{m}^{\frac{1}{2}}\u_{m}\right)\cdot \left(2\nabla\rho_{m}^{\frac{1}{2}}\right)+\eps\triangle\rho_{m}.
\end{aligned}
\end{equation}
 On one hand, for some $r \geqslant2$,
\begin{equation}
\mathbb{E}\left[\left\|\rho_{m}^{\frac{1}{2}}\rho_{m}^{\frac{1}{2}}\odiv\u_{m}\right\|_{L_t^{2} L_x^{\frac{3}{2}}}^{r}\right]  \ls \mathbb{E}\left[\left(\left\|\rho_{m}^{\frac{1}{2}}\right\|_{L_t^{\infty}L_x^{6}}\left\|\rho_{m}^{\frac{1}{2}}\odiv\u_{m}\right\|_{L_t^{2}L_x^{2}}\right)^{r}\right]
\ls C.
\end{equation}
On the other hand,
\begin{equation}
\mathbb{E}\left[\left\| \nabla \rho_{m}^{\frac{1}{2}}\cdot \left(\rho_{m}^{\frac{1}{2}}\u_{m}\right)\right\|_{L_t^{2}L_x^{\frac{3}{2}}}^{r}\right]
\ls \mathbb{E}\left[\left(\left\|\rho_{m}^{\frac{1}{2}}\u_{m}\right\|_{L_t^{\infty}L_x^{2}}\left\|\nabla \rho_{m}^{\frac{1}{2}}\right\|_{L_t^{2} L_x^{6}}\right)^{r}\right]
\ls C,
\end{equation}
so
\begin{equation}
 \mathbb{E}\left[\left(\left\|(\rho_{m})_{t}\right\|_{L_t^{2}L_x^{\frac{3}{2}}}\right)^{r}\right]\ls  C,
\end{equation}
$C$ depends on $T,r, \eps$ and $\kappa$.
Meanwhile, the mass equation reduces to
\begin{equation}
\begin{aligned}
2\left(\rho_{m}^{\frac{1}{2}}\right)_{t}&=-\rho_{m}^{\frac{1}{2}}\odiv\u_{m}-2\nabla\rho_{m}^{\frac{1}{2}}\cdot\u_{m}+\eps \frac{\triangle\rho_{m}}{\rho_{m}^{\frac{1}{2}}}.\\
\end{aligned}
\end{equation}
Letting $R\ra \infty$ in the energy estimates (\ref{energy estimate1}), (\ref{energy estimate2}) and (\ref{energy estimate3}), 
 and by $H^{8}\left( \O \right)\hookrightarrow C^{6,\frac{1}{2}}\left( \O \right)$, we have
\begin{equation}
 \mathbb{E}\left[\left(\left\|\eps \frac{\triangle\rho_{m}}{\rho_{m}^{\frac{1}{2}}} \right\|_{L_t^2 L_x^2}\right)^{r}\right]\ls  C,
\end{equation}
where $C$ depends on $T,r,\eps\eta$ and $\eps\delta$.
Hence we obtain
\begin{equation}
 \mathbb{E}\left[\left(\left\|\left(\rho_{m}^{\frac{1}{2}}\right)_{t}\right\|_{L_t^2 L_x^2}\right)^{r}\right]\ls  C,
\end{equation}
for $C$ depends on $T,r,r_{1},\eps\eta,\eps\delta$ and $\eps\kappa$, $r \geqslant2$.
Notice that the above estimates are independent of $\delta$ as $\eps$ goes to zero firstly. The estimate
$\mathbb{E}\left[\left(\left\|\nabla\rho_{m}^{\frac{\gamma}{2}}\right\|_{L_t^2 L_x^2}\right)^{r}\right] \ls C$
implies
\begin{equation}
 \mathbb{E}\left[\left(\left\|\rho_{m}^{\gamma}\right\|_{L_{t}^{1}L_{x}^{3}}\right)^{r}\right] \ls C,
\end{equation}
together with $\rho_{m}^{\gamma}\in L^{\infty}(0,T;L^{1}\left( \O \right))$, by H\"older's inequality we have
\begin{equation}
\begin{aligned}
\quad \mathbb{E}\left[\left(\left\|\rho_{m}^{\gamma}\right\|_{L_t^{\frac{5}{3}}L_x^{\frac{5}{3}}}\right)^{r}\right]
\ls\mathbb{E}\left[\left\|\rho_{m}^{\gamma}\right\|^{\frac{2}{5}r}_{L_t^{\infty}L_x^{1}}\left\|\rho_{m}^{\gamma}\right\|^{\frac{3}{5}r}_{L_t^{1}L_x^{3}}\right]
\ls C,
\end{aligned}
\end{equation}
where $C$ depends on $T ,r$ and $\eps$, $r \geqslant2$. For $\rho_{m}\u_{m}$, for some $r \geqslant2$, we estimate
\begin{equation}
\begin{aligned}
\quad \mathbb{E}\left[\left(\left\|\rho_{m}\u_{m}\right\|_{L_t^2 L_x^2}\right)^{r}\right]
\ls\mathbb{E}\left[\left\|\rho_{m}^{\frac{1}{2}}\right\|_{L_t^{2}L_x^{\infty}}^{r}\left\|\rho_{m}^{\frac{1}{2}}\u_{m}\right\|_{L_t^{\infty}L_x^{2}}^{r}\right]
\ls C,
\end{aligned}
\end{equation}
due to $ H^{2}\left( \O \right)\hookrightarrow C^{0,\frac{1}{2}}\left( \O \right)$, $\rho_{m}^{\frac{1}{2}}\u_{m} \in L^{\infty}([0,T];L^{2}\left( \O \right))$, where $C$ depends on $T ,r$ and $\eps\kappa$, $r \geqslant2$.
We calculate
\begin{equation}
\nabla (\rho_{m}\u_{m})=2\nabla\rho_{m}^{\frac{1}{2}}\otimes \rho_{m}^{\frac{1}{2}}\u_{m}+\rho_{m}^{\frac{1}{2}}\cdot\rho_{m}^{\frac{1}{2}}\nabla\u_{m},\\
\end{equation}
 so we have
$\nabla \left(\rho_{m}\u_{m}\right)\in L^{r}\left(\Omega; L^{2}\left(0,T; L^{\frac{3}{2}}\left( \O \right)\right)\right)$, $W^{1,\frac{3}{2}}\left( \O \right)\hookrightarrow L^3\left( \O \right)$.  $\rho_{m}\u_{m}$ are bounded in $ \L^{r}\left(\Omega;L^{2}\left([0,T];L^2\left( \O \right)\right)\right)$, $r \geqslant2$. This is the end of the proof. \hfill $\square$
\end{proof}

We choose the path space
\begin{equation}
\mathcal{X}_{2}=\mathcal{X}_{\rho_{0}}^{2} \times \mathcal{X}_{\mathbf{q}_{0}}^{2} \times \mathcal{X}_{\frac{\mathbf{q}_{0}}{\sqrt{\rho_{0}}}}^{2} \times \mathcal{X}_{\rho}^{2} \times \mathcal{X}_{\u}^{2} \times \mathcal{X}_{\rho\u}^{2} \times \mathcal{X}_{W}^{2},
\end{equation}
where
$\mathcal{X}_{\rho_{0}}^{2}=L^{\gamma}\left(\O\right)\cap L^{1}\left(\O\right) \cap L^{-10}\left(\O\right)\cap H^{9}\left(\O\right)$,  
$\mathcal{X}_{\mathbf{q}_{0}}^{2}=L^{1}\left(\O\right))$,
$\mathcal{X}_{\frac{\mathbf{q}_{0}}{\sqrt{\rho_{0}}}}^{2}=L^{2}\left(\O\right)$,
$\mathcal{X}_{\rho}^{2}=L^{2}\left([0,T];H^{10}\left(\O\right)\right)\cap L^{2}\left([0,T];W^{1,3}\left( \O \right)\right)\cap L^{\frac{5}{3}\gamma}\left([0,T]\times\O\right)$, 
$\mathcal{X}_{\u}^{2}= L^{2}\left([0,T]\times\O\right)\cap L^{4}\left([0,T]\times\O\right)$,
$\mathcal{X}_{\rho\u}^{2}= L^{2}\left([0,T]; W^{1,\frac{3}{2}}\left(\O\right)\right)\cap C\left([0,T];L^{\frac{3}{2}}\left( \O \right)\right)$,
$\mathcal{X}_{W}^{2}= C\left([0,T];\mathcal{H}\right)$.\\

\begin{flushleft}
\textbf{Step 2:} Show the tightness of the laws.
\end{flushleft}
\begin{proposition}
$\left\{\mathcal{L}\left[\rho_{0,m}, \mathbf{q}_{0,m}, \frac{\mathbf{q}_{0,m}}{\sqrt{\rho_{0,m}}}, \rho_{m}, \u_{m},\rho_{m}\u_{m}, W_{m} \right]; m \in \mathds{N}\right\}$ is tight on $\mathcal{X}_{2}$.
\end{proposition}
 To prove this proposition, we need to prove proposition \ref{the first propo in m layer} -- \ref{compactness of Y}.
\begin{proposition}\label{the first propo in m layer}
 The set $\left\{\mathcal{L}\left[\rho_{0,m}\right]; m\in \mathds{N}\right\}$ is tight on $\mathcal{X}_{\rho_{0}}^{2}$. The set $\left\{\mathcal{L}\left[\mathbf{q}_{0,m}\right]; m\in \mathds{N}\right\}$ is tight on $\mathcal{X}_{\mathbf{q}_{0}}^{2}$,
and $\left\{\mathcal{L}\left[\frac{\mathbf{q}_{0,m}}{\sqrt{\rho_{0,m}}}\right]; m\in \mathds{N}\right\}$ is tight on $\mathcal{X}_{\frac{\mathbf{q}_{0}}{\sqrt{\rho_{0}}}}^{2}$.
\end{proposition}
Probability measures on the $\sigma$-algebra of Borel sets of any Polish space are Radon measures. Therefore probability measures are inner regular or tight, see \S \ref{Radon measure} in Appendix.
We recall that $L^{p}\left( \O \right)(1\ls p<+\infty)$ is separable, $L^{\infty}\left( \O \right)$ is inseparable, and $C\left([a,b]\right)$, space of continuous functions in $[a,b]$ is separable.
 $L^{p}(1\ls p<+\infty)$ is complete so that it is a Polish space. Consequently, $\mathcal{L}\left[\rho_{0,m}\right]$ is tight on
   $\mathcal{X}_{\rho_{0}}^{2}=L^{\gamma}\left( \O \right)$.
Similarly, we get the tightness of the measures generated by $\mathbf{q}_{0,m},\frac{\mathbf{q}_{0,m}}{\sqrt{\rho_{0,m}}}$.

\begin{proposition}
 The set $\left\{\mathcal{L}\left[\u_{m}\right]; m\in \mathds{N}\right\}$ is tight on $\mathcal{X}_{\u}^{2}$. The set $\left\{\mathcal{L}\left[\rho_{m}\right]; m\in \mathds{N}\right\}$ is tight on $\mathcal{X}_{\rho}$.
\end{proposition}
\begin{proof} For any given $\iota>0$, some $r \geqslant2$, $\mathbb{E}\left[\left\| \u_{m}\right\|_{L_t^{4}L_x^{4}}^{4r}\right]\ls C\left(r_{0}\right)$, let $L\gs \frac{\iota}{C}$, and we define the set
\begin{equation}
B_{L}=\left\{\u_{m}\in L^{4}\left([0,T]\times\O\right)|\left\| \u_{m}\right\|_{L_t^{4}L_x^{4}}\ls L\right\}.
\end{equation}
Then it holds
\begin{equation}
\mathcal{L}\left[\u_{m}\right](B^{c}_{L})=\mathbb{P}\left[\left\{\left\| \u_{m}\right\|_{L_t^{4}L_x^{4}}\gs L\right\}\right] \ls \frac{\mathbb{E}\left[\left\| \u_{m}\right\|_{L_t^{4}L_x^{4}}^{4r}\right]}{L^{4r}} \ls\frac{C}{L^{4r}} \ls \iota,
\end{equation}
which shows the tightness of $\left\{\mathcal{L}\left[\u_{m}\right]; m\in \mathds{N}\right\}$.
For the more information about $\rho_{m}$, 
 we calculate that
\begin{equation}
\triangle \rho_{m}=2\left|\nabla \rho_{m}^{\frac{1}{2}}\right|^{2}+2\rho_{m}^{\frac{1}{2}}\triangle \rho_{m}^{\frac{1}{2}}.\\
\end{equation}
Since $\rho_{m}^{\frac{1}{2}}\in L^{r}\left(\Omega;L^{\infty}\left(0,T;H^{1}\left( \O \right)\right)\right)$, $H^{1}\left( \O \right)\hookrightarrow L^{6}\left( \O \right)$,
$\rho_{m}^{\frac{1}{2}} \in L^{r}(\Omega;L^{2}(0,T;H^{2}\left( \O \right)))$, it holds
 $\rho_{m}^{\frac{1}{2}}\triangle \rho_{m}^{\frac{1}{2}} \in L^{r}\left(\Omega;L^{2}\left(0,T;L^{\frac{3}{2}}\left( \O \right)\right)\right)$ for some $r \geqslant2$.
 $\left|\nabla \rho_{m}^{\frac{1}{2}}\right|^{2}\in L^{r}\left(\Omega;L^{2}\left(0,T;L^{\frac{3}{2}}\left( \O \right)\right)\right)$ holds for some $r \geqslant2$ similarly. This proves the proposition. \hfill$\square$
 \end{proof}
\begin{proposition}\label{compactness of Y}
 The set $\left\{\mathcal{L}\left[\rho_{m}\u_{m}\right]; m\in \mathds{N}\right\}$ is tight on $\mathcal{X}_{\rho\u}^{2}$.
\end{proposition}
\begin{proof} We first consider the deterministic part, and we denote
\begin{align}
Y_{m}(t)
=&\Pi_{m}\left(\rho_{m}\u_{m}(0)\right)-\int_{0}^{t}\Pi_{m}\left(\odiv(\rho_{m}\u_{m}\otimes \u_{m})\right) \d s+\int_{0}^{t}\Pi_{m}\left(\odiv(\rho_{m}\mathbb{D}\u_{m})\right) \d s\notag\\
&-\int_{0}^{t}\Pi_{m}\left(\nabla \rho_{m}^{\gamma}\right) \d s +\int_{0}^{t}\frac{11}{10}\eta \Pi_{m}\left(\nabla \rho_{m}^{-10}\right)\d s
+\int_{0}^{t}\delta \Pi_{m}\left(\rho_{m} \nabla \triangle^{9}\rho_{m}\right)\d s\\
&-\int_{0}^{t} r_{0}\Pi_{m}\left( |\mathbf{u}_{m}|^{2}\mathbf{u}_{m}\right)\d s-\int_{0}^{t} r_{1} \Pi_{m}\left(\rho_{m}|\mathbf{u}_{m}|^{2}\mathbf{u}_{m}\right)\d s-\int_{0}^{t}r_{2}\Pi_{m}\left( \mathbf{u}_{m}\right)\d s\notag\\
&- \int_{0}^{t}\left(\eps \Pi_{m}\left(\triangle^{2} \u_{m}\right)+\eps \Pi_{m}\left(\nabla \rho_{m}\nabla \u_{m}\right)\right)\d s +\int_{0}^{t}\kappa
\Pi_{m}\left(\rho_{m}\left(\nabla\left(\frac{\triangle \sqrt{\rho_{m}}}{\sqrt{\rho_{m}}}\right)\right)\right)\d s.\notag
\end{align}

We claim that $\mathbb{E}\left[\left\| Y_{m}(t)\right\|_{C^{0,\varsigma}([0,T];W^{-l,2}\left( \O \right))}^{r}\right]\ls C(T,r)$ for some $r >4$, where $\varsigma < \frac{1}{4}-\frac{1}{r}$, $l>\frac{5}{2}$.
In fact, $\mathbb{E}\left[\left\| \rho_{m}\right\|_{L^{\infty}([0,T];L^{\gamma}\left( \O \right))}^{r}\right]\ls C(T,r)$,
$\mathbb{E}\left[\left\| \rho_{m}^{\frac{1}{4}}\u_{m}\right\|_{L^{4}([0,T];L^{4}\left( \O \right))}^{r}\right]\ls C(T,r,r_{1})$,
rewrite $\rho_{m}\u_{m}\otimes \u_{m}=\rho_{m}^{\frac{1}{2}}\rho_{m}^{\frac{1}{2}}\u_{m}\otimes \u_{m}$,
 so $\rho_{m}\u_{m}\otimes \u_{m} \in L^{2}\left([0,T];L^{\frac{2\gamma}{\gamma+1}}\left( \O \right)\right)$,
 $\left\{\Pi_{m}\left(\odiv \left(\rho_{m}\u_{m}\otimes \u_{m}\right)\right)\right\}$ is bounded in $L^{r}\left(\Omega ; L^{2}\left([0,T];W^{-1,\frac{2\gamma}{\gamma+1}}\left( \O \right)\right)\right)$, which is compactly embedded in $W^{-l,2}\left( \O \right)$, for $l>\frac{5}{2}$.
Similarly, by (\ref{energy estimate1}), (\ref{energy estimate2}), (\ref{energy estimate3}) with $R\ra \infty$, and Sobolev's embedding, we know that the families $\left\{\nabla p(\rho_{m})\right\}$, $\left\{\Pi_{m}\left(\odiv(\rho_{m}\mathbb{D}\u_{m})\right)\right\}$, $\left\{\Pi_{m}\left(\frac{11}{10}\eta\nabla\rho_{m}^{-10} \right)\right\}$, $\left\{\Pi_{m}\left(\eps \nabla \rho_{m}\nabla\u_{m}\right)\right\}$,
$\left\{\Pi_{m}\left(\kappa\odiv \left(\rho_{m}\nabla^{2}\ln\rho_{m}\right)\right)\right\}$, and
$\left\{ \Pi_{m}\left(\delta\rho_{m}\nabla \triangle^{9}\rho_{m}\right)\right\}$ are bounded in $L^{r}\left(\Omega ; L^{2}\left([0,T];W^{-l,2}\left( \O \right)\right)\right)$, for $l>\frac{5}{2}$.
 Therefore $\frac{\d}{\d t}Y_{m}(t)$ is bounded
 in $L^{r}\left(\Omega; L^{\frac{4}{3}}\left([0,T];W^{-l,2}\left( \O \right)\right)\right)$, for $l>\frac{5}{2}$.
 For any test function $\varphi(x) \in  W^{l,2}\left( \O \right)$, $\left\|\varphi\right\|_{W^{l,2}\left( \O \right)}=1$, by stochastic Fubini's theorem, we estimate
\begin{align}\label{the result of deterministic part}
&\quad \mathbb{E}\left[\left(\left\|Y_{m}(t_{2})\right\|_{W^{-l,2}\left( \O \right)}-\left\|Y_{m}(t_{1})\right\|_{W^{-l,2}\left( \O \right)}\right)^{r}\right]
\ls \mathbb{E}\left[\left(\left\|Y_{m}(t_{2})-Y_{m}(t_{1})\right\|_{W^{-l,2}\left( \O \right)}\right)^{r}\right]\notag\\
&= \mathbb{E}\left[\left(\sup\limits_{\varphi}\int_{\O}\int_{t_{1}}^{t_{2}} \d Y_{m}(t)\varphi(x)\d x\right)^{r}\right]
= \mathbb{E}\left[\left(\sup\limits_{\varphi}\int_{\O}\int_{t_{1}}^{t_{2}} \varphi(x)\d Y_{m}(t)\d x\right)^{r}\right]\notag\\
&= \mathbb{E}\left[\left(\sup\limits_{\varphi}\int_{t_{1}}^{t_{2}}\int_{\O} \varphi(x)\frac{\d}{\d t} Y_{m}(t)\d x \d t\right)^{r}\right] \\
&\ls \mathbb{E}\left[\left(\sup\limits_{\varphi}\left(\int_{t_{1}}^{t_{2}}\left(\left\|\varphi(x)\right\|_{W^{l,2}\left(\O\right)}\right)^{4}\d t\right)^{\frac{1}{4}}\left(\int_{t_{1}}^{t_{2}} \left\| \frac{\d}{\d t} Y_{m}(t)\d x\right\|_{W^{-l,2}\left( \O \right)}^{\frac{4}{3}} \d t\right)^{\frac{3}{4}}\right)^{r}\right]\notag\\
&\ls C (t_{2}-t_{1})^{\frac{r}{4}} \mathbb{E}\left[\sup\limits_{\varphi}\left(\int_{t_{1}}^{t_{2}} \left\| \frac{\d}{\d t} Y_{m}(t)\d x\right\|_{W^{-l,2}\left( \O \right)}^{\frac{4}{3}} \d t\right)^{\frac{3r}{4}}\right]
\ls C (t_{2}-t_{1})^{\frac{r}{4}}, \notag
\end{align}
by Kolmogorov-Centov's theorem, it holds $\left\| \bar{Y}_{m}(t)\right\|_{C^{0,\varsigma}([0,T];W^{-l,2}\left( \O \right))}\ls C$ uniformly in $m$, where $\bar{Y}_{m}(t)$ is a modification of $Y_{m}(t)$, and $C$ depends on $T$ and $r$, $\bar{\mathbb{P}}$ a.s., for $\varsigma < \frac{1}{4}-\frac{1}{r}$. We will use the notation $ Y_{m}(t)$ here and hereafter.
Next, we consider the time regularity of stochastic integral. Applying Burkh\"older-Davis-Gundy's inequality, we obtain
\begin{align}
&\quad \mathbb{E}\left[\left(\left\|\int_{\tau_{1}}^{\tau_{2}} \rho_{m} \Pi_{m}(\mathbb{F}(\rho_{m},\u_{m}))\d W_{m}\right\|_{L^{2}\left( \O \right)}\right)^{r}\right]\notag\\
&\ls C \mathbb{E}\left[\left(\int_{\O}\int_{\tau_{1}}^{\tau_{2}}\left|\rho_{m} \Pi_{m}(\mathbb{F}(\rho_{m},\u_{m}))\right|^{2}\d t\d x\right)^{\frac{r}{2}}\right]\\
&\ls C \sum\limits_{k=1}^{\infty}f_{k}^{2}\mathbb{E}\left[\left(\int_{\tau_{1}}^{\tau_{2}} \int_{\O}\left(\left|\rho_{m}\right|^{2}+\left|\rho_{m}\u_{m}\right|^{2}\right) \d x\d t\right)^{\frac{r}{2}}\right],\notag
\end{align}
where
\begin{align}
\mathbb{E}\left[\left(\int_{\tau_{1}}^{\tau_{2}} \int_{\O}\left|\rho_{m}\right|^{2}\d x\d t\right)^{\frac{r}{2}}\right] \ls &C (\tau_{2}-\tau_{1})^{\frac{r}{2}}\mathbb{E}\left[\left(\left\|\rho_{m}\right\|_{L_t^{\infty}L_x^{3}}\right)^{r}\right]\\
\ls &(\tau_{2}-\tau_{1})^{\frac{r}{2}}\mathbb{E}\left[\left(\left\|\rho_{m}^{\frac{1}{2}}\right\|_{L_t^{\infty}H_x^{1}}^{2}\right)^{r}\right] , \notag
\end{align}
\begin{align}
\mathbb{E}\left[\left(\int_{\tau_{1}}^{\tau_{2}} \int_{\O}\left|\rho_{m}\u_{m}\right|^{2}\d x\d t\right)^{\frac{r}{2}}\right] \ls &C (\tau_{2}-\tau_{1})^{\frac{r}{2}}\mathbb{E}\left[\left(\left\|\rho_{m}\right\|_{L_t^{\infty}L_x^{3}}\right)^{r}\right]\\
\ls &(\tau_{2}-\tau_{1})^{\frac{r}{2}}\mathbb{E}\left[\left(\left\|\rho_{m}^{\frac{1}{2}}\right\|_{L_t^{\infty}H_x^{1}}^{2}\right)^{r}\right] , \notag
\end{align}
for $r>1$, where $C$ depends on $\omega$ and $r$.
By Kolmogorov-Centov's continuity theorem (see \S \ref{Centov thm} in Appendix), up to a modification, there holds
\begin{equation}
\mathbb{E}\left[\left(\left\|\int_{0}^{t} \rho_{m} \Pi_{m}(\mathbb{F}(\rho_{m},\u_{m}))\d W_{m}\right\|_{C^{0,\varsigma}([0,T];L^{2}\left( \O \right))}\right)^{r}\right]\ls C,
\end{equation}
uniformly in $m$, for $r>2,\varsigma\in \left(0,\frac{1}{2}-\frac{1}{r}\right)$. Combined with the result of deterministic part \eqref{the result of deterministic part}, with $C$ dependent on $\omega$, $T$ and $r$, it holds
\begin{equation}
\mathbb{E}\left[\left(\left\|\Pi_{m}\left(\rho_{m}\u_{m}\right)\right\|_{C^{0,\varsigma}\left([0,T];L^{2}\left( \O \right)\right)}\right)^{r}\right]\ls C ,\quad \varsigma\in \left(0,\frac{1}{4}-\frac{1}{r}\right), \quad r>4.
\end{equation}
The embedding $C^{0,\varsigma}\left([0,T];W^{-l,2}\left( \O \right)\right)\hookrightarrow C^{0,\varsigma-\eps}\left([0,T];W^{-k,2}\left( \O \right)\right)$, $k>l,$
implies the tightness of $\mathcal{L}\left[\rho_{m}\u_{m}\right] $ in $C^{0,\varsigma}\left([0,T];W^{-k,2}\left( \O \right)\right)$, for $k>l>\frac{5}{2}$.

 Since the quantum term provides a better regularity in the energy estimates, we obtain
$\Pi_{m}\left(\rho_{m}\u_{m}\right)\in L^{\infty}\left([0, T]; L^{\frac{3}{2}}\left(\O\right)\right)\cap C^{0,\varsigma}\left([0,T];W^{-k,2}\left( \O \right)\right)$,
which is compactly embedded in $C\left([0,T];L^{\frac{3}{2}}\left( \O \right)\right)$.
In fact, 
\begin{align}
 &\mathbb{E}\left[\left(\left\| \Pi_{m}\left(\rho_{m}\u_{m}(t_{2})-\rho_{m}\u_{m}(t_{1})\right)\right\|_{L^{\frac{3}{2}}\left( \O \right)}\right)^{r}\right]\notag\\
 =&\mathbb{E}\left[\left(\sup\limits_{\phi\in L^{3}\left( \O \right),\left\|\phi\right\|_{L^{3}}=1}\int_{\O}\Pi_{m}\left(\rho_{m}\u_{m}(t_{2})-\rho_{m}\u_{m}(t_{1})\right)\phi \d x\right)^{r}\right]\\
 =&\mathbb{E}\left[\left(\sup\limits_{\phi\in L^{3}\left( \O \right),\left\|\phi\right\|_{L^{3}}=1}\left(\int_{\O}\Pi_{m}\left(\rho_{m}\u_{m}(t_{2})-\rho_{m}\u_{m}(t_{1})\right)\left(\phi-\phi \ast \eta_{\imath}\right) \d x \right.\right.\right.\notag\\
 &\left.\left.\left.\quad +\int_{\O}\Pi_{m}\left(\rho_{m}\u_{m}(t_{2})-\rho_{m}\u_{m}(t_{1})\right) \phi \ast \eta_{\imath} \d x \right)\right)^{r}\right],\notag
\end{align}
where $\eta_{\imath} $ is a mollifier and $\left\|\phi \ast \eta_{\imath}\right\|_{L^{3}\left( \O \right)}\ra \left\|\phi\right\|_{L^{3}\left( \O \right)}$ holds as $\imath$ goes to $0$. Since $L^{p}\left( \O \right), 1<p<\infty$ is uniformly convex space, by Hanner's inequality, weak convergence and convergence in norm, we imply that $\left\|\phi \ast \eta_{\imath}-\phi\right\|_{L^{3}\left( \O \right)}\ra 0$. By H\"older's inequality, the first term in the above formula will go to $0$. By stochastic Fubini's theorem, it holds
\begin{align}
&\int_{\O}\Pi_{m}\left(\rho_{m}\u_{m}(t_{2})-\rho_{m}\u_{m}(t_{1})\right) \phi \ast \eta_{\imath} \d x  \notag\\
&=\int_{\O}\left(\int_{t_{1}}^{t_{2}}\d Y_{m}(t)+ \int_{t_{1}}^{t_{2}}\rho_{m}\mathds{F}(\rho_{m},\u_{m})\d W_{m}\right) \phi \ast \eta_{\imath} \d x \notag\\
&=\int_{t_{1}}^{t_{2}}\int_{\O}\frac{\d Y_{m}(t)}{\d t}\phi \ast \eta_{\imath} \d x\d t + \int_{t_{1}}^{t_{2}}\int_{\O}\rho_{m}\mathds{F}(\rho_{m},\u_{m})\phi\ast \eta_{\imath} \d x\d W_{m}\\
 &\ls \left\|\frac{\d Y_{m}(t)}{\d t}\right\|_{L^{\frac{4}{3}}([0,T];W^{-l,2}\left( \O \right))} \left\|\phi \ast \eta_{\imath}\right\|_{W^{l,2}\left( \O \right)}(t_{2}-t_{1})^{\frac{1}{4}}\notag\\
 &+\int_{\O}\int_{t_{1}}^{t_{2}}\rho_{m}\mathds{F}(\rho_{m},\u_{m}) \d W_{m} \left(\phi\ast \eta_{\imath}\right)\d x.\notag
\end{align}
With $C$ dependent on $T$ and $r$, it holds
\begin{align}
&\mathbb{E}\left[\left\|\int_{t_{1}}^{t_{2}}\rho_{m}\mathbb{F}(\rho_{m},\u_{m})\d W_{m}\right\|^{r}_{L^{\frac{3}{2}}\left( \O \right)}\right]\\
\ls &\mathbb{E}\left[\left\|\int_{t_{1}}^{t_{2}}\rho_{m}\mathbb{F}(\rho_{m},\u_{m})\d W_{m}\right\|^{r}_{L^{2}\left( \O \right)}\right]\ls C (t_{2}-t_{1})^{\frac{r}{2}}.\notag
\end{align}
Therefore, we have
\begin{equation}
 \mathbb{E}\left[\left(\left\|\Pi_{m}\left(\rho_{m}\u_{m}(t_{2})-\rho_{m}\u_{m}(t_{1})\right)\right\|_{L^{\frac{3}{2}}\left( \O \right)}\right)^{r}\right]\ls C (t_{2}-t_{1})^{\frac{r}{4}},
\end{equation}
by Kolmogorov-Centov's theorem, up to a modification, it holds
\begin{align}
\mathbb{E}\left[\left\| \Pi_{m}\left(\rho_{m}\u_{m}(t)\right)\right\|_{C^{0,\varsigma}([0,T];L^{\frac{3}{2}}\left( \O \right)}^{r}\right]\ls C,
\end{align}
 uniformly in $m$, where $C$ depends on $T$ and $r$, for $\varsigma < \frac{1}{4}-\frac{1}{r}$. Alternatively, by Chebyshev's inequality, we have
\begin{align}
&\sup\limits_{m\geqslant 1} \mathbb{P}_{m}\left[\left\{\left\|\Pi_{m}\left(\rho_{m}\u_{m}(t_{2})-\rho_{m}\u_{m}(t_{1})\right)\right\|_{L^{\frac{3}{2}}\left( \O \right)}>\eps \right\} \right]\notag\\
\ls & \sup\limits_{m\geqslant 1} \frac{\mathbb{E}\left[\left(\left\|\Pi_{m}\left(\rho_{m}\u_{m}(t_{2})-\rho_{m}\u_{m}(t_{1})\right)\right\|_{L^{\frac{3}{2}}\left( \O \right)}\right)^{r}\right]}{\eps^{r}} \\
\ls & C \delta ^{\frac{r}{4}}\ra 0 ,\quad \left|t_{2}-t_{1}\right| \ls \delta,\text{ as }\delta\ra 0,\notag
\end{align}
 by Kolmogorov's theorem (see \S \ref{Kolmogorov thm} in Appendix), we get the tightness.
 This completes the proof. \hfill$\square$
\end{proof}
\begin{flushleft}
\textbf{Step 3:} Pass to the limit $m\ra \infty$.
\end{flushleft}

$\mathcal{X}_{2}$ is sub-Polish space instead of Polish space because weak topologies are generally not measurable. So our stochastic compactness argument
 is based on Jakubowski's extension of Skorokhod's representation theorem.
\begin{proposition} There exist
 $\left\{{\bar{\rho}_{0,m}}, \bar{\mathbf{q}}_{0,m}, \frac{\bar{\mathbf{q}}_{0,m}}{\sqrt{\bar{\rho}_{0,m}}}, \bar{\rho}_{m}, \bar{\u}_{m}, \bar{\rho}_{m}\bar{\u}_{m}, \bar{W}_{m} \right\},\quad m\in \mathds{N},$
 and \\
 $\left\{\rho_{0,{\rm reg}}, \mathbf{q}_{0,{\rm reg}},\frac{\mathbf{q}_{0,{\rm reg}}}{\sqrt{\rho_{0,{\rm reg}}}}, \rho_{\rm reg}, \u_{\rm reg},\rho_{\rm reg}\u_{\rm reg}, W_{\rm reg}\right\}$, two families of $\mathcal{X}_{2}$-valued Borel measurable random variables defined on a complete probability space, still denoted as $\left(\bar{\Omega},\bar{\mathcal{F}},\bar{\mathbb{P}} \right)$, such that (up to a subsequence):
\begin{enumerate}
  \item For all $m\in \mathds{N}$,
  $\mathcal{L}\left[\bar{\rho}_{0,m}, \bar{q}_{0,m}, \frac{\bar{\mathbf{q}}_{0,m}}{\sqrt{\bar{\rho}_{0,m}}}, \bar{\rho}_{m}, \bar{\u}_{m}, \bar{\rho}_{m}\bar{\u}_{m}, \bar{W}_{m} \right],\quad m\in \mathds{N}$\\
   coincides with $\mathcal{L}\left[\rho_{0,m}, \mathbf{q}_{0,m}, \frac{\mathbf{q}_{0,m}}{\sqrt{\rho_{0,m}}}, \rho_{m}, \u_{m},\rho_{m}\u_{m}, W_{m}\right], \quad m\in \mathds{N} $ on $\mathcal{X}_{2}$;
  \item $\mathcal{L}\left[\rho_{0,{\rm reg}}, \mathbf{q}_{0,{\rm reg}},\frac{\mathbf{q}_{0,{\rm reg}}}{\sqrt{\rho_{0,{\rm reg}}}}, \rho_{\rm reg}, \u_{\rm reg},\rho_{\rm{reg}}\u_{\rm{reg}}, W_{\rm{reg}}\right]$ on $\mathcal{X}_2$ is a Radon measure;
  \item  $\left\{\bar{\rho}_{0,m}, \bar{\mathbf{q}}_{0,m}, \frac{\bar{\mathbf{q}}_{0,m}}{\sqrt{\bar{\rho}_{0,m}}}, \bar{\rho}_{m}, \bar{\u}_{m}, \bar{\rho}_{m}\bar{\u}_{m}, \bar{W}_{m}\right\}$, $m\in \mathds{N}$,
   converges to\\
    $\left\{\rho_{0,{\rm reg}}, \mathbf{q}_{0,{\rm reg}},\frac{\mathbf{q}_{0,{\rm reg}}}{\sqrt{\rho_{0,{\rm reg}}}}, \rho_{\rm reg}, \u_{\rm reg},\rho_{\rm reg}\u_{\rm reg}, W_{\rm reg}\right\}$ in the topology of $\mathcal{X}_{2}$, $\bar{\mathbb{P}}$ {\rm a.s.} as $m\ra +\infty$.
\end{enumerate}
\end{proposition}
\begin{corollary}
 $L^{2}\left(0,T; W^{2,\frac{3}{2}}\left( \O \right)\right)$ is continuously embedded in $L^{2}\left([0,T]\times \O \right)$ and \\
  $L^{2}\left(0,T; L^{\frac{3}{2}}\left( \O \right)\right)$.
 Since $W^{2,\frac{3}{2}}\left( \O \right)$ and $ L^{\frac{3}{2}}\left( \O \right)$ are reflexive spaces, applying the Aubin-Lion's lemma,
  we have the strong convergence of $\rho_{m}$: 
\begin{equation}
\bar{\rho}_{m}\rightarrow \rho_{\rm reg} \quad \text{in} \quad L^{r}\left(\Omega;L^{2}\left([0,T];W^{1,3}\left( \O \right)\right)\right), \quad r \geqslant2 .
\end{equation}
In addition, by \eqref{square root of rhom}, $\bar{\rho}_{m}^{\frac{1}{2}} \in L^{r}\left(\Omega;L^{2}\left([0,T];H^{2}\left( \O \right)\right)\right)$, $H^{2}\left( \O \right)\hookrightarrow W^{1,6}\left( \O \right)$, $ W^{1,6}\left( \O \right)$ is continuously embedded
in $L^{2}\left( \O \right)$, therefore, the strong convergence of $\bar{\rho}_{m}^{\frac{1}{2}}$ holds:
\begin{equation}
\bar{\rho}_{m}^{\frac{1}{2}}\rightarrow \rho_{\rm reg}^{\frac{1}{2}} \quad \text{in} \quad L^{r}\left(\Omega;L^{2}\left([0,T];W^{1,6}\left( \O \right)\right)\right), \quad r \geqslant2.
\end{equation}
\end{corollary}
\begin{remark}\label{convergence of q=rho u}
 $\left\|\Pi_{m}\left[\bar{\rho}_{m}\bar{\u}_{m}\right]\right\|_{L^{\frac{3}{2}}\left( \O \right)}$ is continuous in $[0,T]$ uniformly in $m$, $\bar{\mathbb{P}}$ {\rm a.s.}
The compactness of $\{\Pi_{m}\left[\bar{\rho}_{m}\bar{\u}_{m}\right]\}$ yields that there exists $\mathbf{q}_{\rm{reg}}$ in $C\left([0,T];L^{\frac{3}{2}}\left( \O \right)\right)$ such that
 \begin{equation}
\left\|\Pi_{m}\left[\bar{\rho}_{m}\bar{\u}_{m}\right]\right\|_{L^{\frac{3}{2}}\left( \O \right)}\rightarrow \left\| \mathbf{q}_{\rm reg} \right\|_{L^{\frac{3}{2}}\left( \O \right)} \text{ in }  C([0,T]), \quad \bar{\mathbb{P}}\text{ {\rm a.s.}}
\end{equation}
Since $L^{\frac{3}{2}}\left( \O \right)$ is an uniformly convex space, therefore we have
\begin{equation}
\Pi_{m}\left[\bar{\rho}_{m}\bar{\u}_{m}\right]\rightarrow \mathbf{q}_{\rm reg} \text{ in }  C\left([0,T];L^{\frac{3}{2}}\left( \O \right)\right), \quad \bar{\mathbb{P}}  \text{ {\rm a.s.}}
\end{equation}
Hence $\mathcal{L}\left[\bar{\rho}_{m}\bar{\u}_{m}\right]$ is tight in $\mathcal{X}_{\rho\u}$. Since $\left\{\bar{\rho}_{m}^{\frac{1}{2}}\bar{\u}_{m}\right\}$ converges weakly in $L^{r}\left(\Omega;L^{\infty}\left([0,T];L^{2}\left( \O \right)\right)\right)$, \eqref{energy squa rho and energy quarter rho},
 $\bar{\rho}_{m}^{\frac{1}{2}}\ra \rho_{\rm reg}^{\frac{1}{2}}$ in $L^{r}\left(\Omega;L^{2}\left([0,T];W^{1,6}\left( \O \right)\right)\right)$, in $L^{r}\left(\Omega;L^{\infty}\left([0,T];L^{6}\left( \O \right)\right)\right)$ $\left\{\bar{\rho}_{m}^{\frac{1}{2}}\right\}$ converges weakly to $\rho_{\rm reg}^{\frac{1}{2}}$,
  $\mathbb{E}\left[r_{0}^{r}\left\| \bar{\u}_{m}\right\|^{4r}_{L_t^{4}L_x^{4}}\right]$$\ls C$, $C$ depends on $T$,$r$ and $r_{0}$,
  so $\bar{\rho}_{m}\bar{\u}_{m}\rightharpoonup \rho_{\rm reg} \u_{\rm reg} $ in $L^{r}\left(\Omega;C\left([0,T];L^{2}\left( \O \right)\right)\right)$. Therefore $\mathbf{q}_{\rm reg} =\rho_{\rm reg} \u_{\rm reg} $ in $ C\left([0,T];L^{\frac{3}{2}}\left( \O \right)\right)$, $\bar{\mathbb{P}}$ {\rm a.s.}
\end{remark}

It is worth mentioning is that whether the stochastic integral $\int_{0}^{T} \rho_{\rm reg} \mathbb{F}(\rho_{\rm reg},\u_{\rm reg})\d W_{\rm reg}$ makes sense or not. Firstly, we know
$\mathbb{E}\left[\int_{0}^{T}\left(\int_{\O}\rho_{\rm reg}  \mathbb{F}_{k}(\rho_{\rm reg} ,\u_{\rm reg} )\d x\right)^{2} \d t\right]<+\infty$. $\rho_{\rm reg} \mathbf{F}_{k}\left(\rho_{\rm reg},\u_{\rm reg}\right)$ is adapted in the canonical filtration $\tilde{\mathcal{F}}=\sigma\left(\sigma[\rho_{\rm reg} (t)]\bigcup\sigma[\rho_{\rm reg}\u_{\rm reg} (t) ]\bigcup \sigma[W_{\rm reg} (t)]\right)$; by the property of $\rho_{\rm reg}  \in  C\left([0,T];L^{2}\left( \O \right)\right)$ and $\left|\rho_{\rm reg}  \mathbf{F}_{k}(\rho_{\rm reg} ,\u_{\rm reg}  )\right|\leqslant f_{k}\left(\left|\rho_{\rm reg}\right|  +\left|\rho_{\rm reg} \u_{\rm reg} \right|\right)$, we have that the integral $\int_{\O}\rho_{\rm reg}  \mathbf{F}_{k}(\rho_{\rm reg}  ,\u_{\rm reg}  )\d x$ is continuous in time, therefore it is progressively measurable, see \S \ref{progressive measurability} in Appendix. Hence the the stochastic integral $\int_{0}^{T} \rho_{\rm reg} \mathbb{F}(\rho_{\rm reg},\u_{\rm reg})\d W_{\rm reg}$ is well-defined.

In this $m\rightarrow \infty$ layer, the damping terms provide $\bar{\u}_{m}\rightharpoonup \u_{\rm reg} $ weakly in $L^{4}([0,T]\times \O)$,
$\bar{\rho}_{m}\bar{\u}_{m}\ra\rho_{\rm reg} \u_{\rm reg} $ strongly in $C\left([0,T]; L^{\frac{3}{2}}\left( \O \right)\right)$,
$\bar{\rho}_{m}\rightharpoonup  \rho_{\rm reg} $ weakly in $L^{2}\left([0,T];H^{9}\left( \O \right)\right)$,
so $\bar{\rho}_{m},\bar{\u}_{m}$ satisfy the continuity equation in a weak sense by Vitali's convergence theorem. With the weak convergence of
 $ p(\bar{\rho}_{m}), \bar{\rho}_{0,m}^{-10}$, $\nabla\triangle^{4}\bar{\rho}_{0,m}$,
 $\nabla\bar{\rho}_{m}^{-5}$, $\triangle^{5}\bar{\rho}_{m}$,
  $\nabla\bar{\rho}_{m}^{\frac{\gamma}{2}}$, $\bar{\rho}_{m}\nabla^{2}\ln\bar{\rho}_{m}$ in the corresponding spaces, we have the convergence of deterministic terms in approximated momentum equation.
We mainly consider the convergence of $\int_{\O}\int_{0}^{T}\Pi_{m}\left(\bar{\rho}_{m}\Pi_{m}\left(\mathbb{F}(\bar{\rho}_{m},\bar{\u}_{m})\right)\right)\d \bar{W}_{m}\d x.$
 \begin{lemma}\label{Conver of sto int for m layer}
   $\int_{\O}\int_{0}^{T}\Pi_{m}\left(\bar{\rho}_{m}\Pi_{m}\left(\mathbb{F}(\bar{\rho}_{m},\bar{\u}_{m})\right)\right)\d \bar{W}_{m}\d x
\ra \int_{\O}\int_{0}^{T}\rho_{\rm reg} \mathbb{F}(\rho_{\rm reg} ,\u_{\rm reg} )\d W_{\rm reg} \d x$  $\bar{\mathbb{P}}$ {\rm a.s.}
\end{lemma}
\begin{proof}
We first claim that $\rho_{\rm reg} >0$ for almost everywhere $\left(\omega, t, x\right)\in\Omega\times (0,T)\times \O$.

 Since $\mathbb{E}\left[\left(\left\| \rho_{\rm reg} ^{-1} \right\|_{L_{t}^{\infty}L_{x}^{\infty}}\right)^{r}\right]\ls \mathbb{E}\left[\left(\left\| D^{2}\rho_{\rm reg} ^{-1}\right\|_{L_{t}^{\infty} L_{x}^{2}}\right)^{r}\right]$, computing
\begin{equation}
D^{2}\rho_{\rm reg} ^{-1}=D\left(D\left(\rho_{\rm reg}^{-1}\right)\right)=D\left(-\rho_{\rm reg} ^{-2}D\rho_{\rm reg} \right)=2\rho_{\rm reg} ^{-3}D\rho_{\rm reg} \otimes D\rho_{\rm reg} -\rho_{\rm reg} ^{-2}D^{2}\rho_{\rm reg} ,\\
\end{equation}
so by H\"older's inequalities, we have
\begin{align}\label{lower bound of rho}
&\quad \mathbb{E}\left[\left\| \rho_{\rm reg} ^{-1} \right\|_{L_t^{\infty}L_x^{\infty}}^{r}\right]
\ls C_{r}\mathbb{E}\left[\left\|\rho_{\rm reg} ^{-3}|D\rho_{\rm reg} |^{2}\right\|_{L_t^{\infty}L_x^{2}}^{r}+\left\|\rho_{\rm reg} ^{-2}|D^{2}\rho_{\rm reg} |\right\|_{L_t^{\infty}L_x^{2}}^{r}\right]\notag\\
&\ls C_{r}\mathbb{E}\left[\left\|\rho_{\rm reg} ^{-1}\right\|_{L_t^{\infty}L_x^{10}}^{3r} \left\| \rho_{\rm reg} \right\|_{L_t^{\infty}H_x^{9}}^{2r}\right]+C_{r}\mathbb{E}\left[\left\|\rho_{\rm reg} ^{-1}\right\|_{L_t^{\infty}L_x^{10}}^{2r}\left\| \rho_{\rm reg} \right\|_{L_t^{\infty}H_x^{9}}^{r}\right]\\
&\ls C_{r}\mathbb{E}\left[e_{\rm reg}(t)^{\frac{3r}{10}}e_{\rm reg}(t)^{r}\right]+C_{r}\mathbb{E}\left[e_{\rm reg}(t)^{\frac{2r}{10}}e_{\rm reg}(t)^{\frac{r}{2}}\right]\notag\\
&\ls C_{r}\left(\mathbb{E}\left[e_{\rm reg}(0)^{\frac{13r}{10}}\right]+\left(\mathbb{E}\left[e_{\rm reg}(0)^{\frac{13r}{10}}\right]\right)^{\frac{7}{13}}\right),\notag
\end{align}
for some $r>4$, where $C_{r}$ may be different constants in different inequality. The last inequality holds up to a constant $K^{r}$, which comes from Burkholder-Davis-Gundy's inequality, which is equivalent to $e^{r}$, see \S \ref{BDG inequa} in Appendix. By Chebyshev's inequality, it holds
\begin{align}
\bar{\mathbb{P}}\left[\left\{\omega\in \Omega\left|\left\| \rho_{\rm reg} ^{-1}(\omega) \right\|_{L_{t}^{\infty}L_{x}^{\infty}}>L\right.\right\}\right] \ls \frac{C_{r}\left(\mathbb{E}\left[e_{\rm reg}(0)^{\frac{13r}{10}}\right]+\left(\mathbb{E}\left[e_{\rm reg}(0)^{\frac{13r}{10}}\right]\right)^{\frac{7}{13}}\right)}{L^{r}}\ra 0,
\end{align}
as $L\ra +\infty$,
which shows that $\rho_{\rm reg} >0$ for almost everywhere in $\Omega\times (0,T)\times \O$ $\mathbb{P}$ a.s.

Now we turn to the proof of lemma \ref{Conver of sto int for m layer}.
Recalling the definition of $\mathbf{F}_{k}(\rho_{\rm reg},\u_{\rm reg})$, when $\rho_{\rm reg}\neq 0$,
 we write $\mathbf{F}_{k}(\rho_{\rm reg},\u_{\rm reg})=\mathbf{F}_{k}\left(\rho_{\rm reg},\frac{\rho_{\rm reg}\u_{\rm reg}}{\rho_{\rm reg}}\right)$.
Recall that $|\rho_{\rm reg}\mathbf{F}_{k}|\leqslant f_{k}(\left|\rho_{\rm reg}\right|+\left|\rho_{\rm reg}\u_{\rm reg}\right|)$, $\sum\limits_{k=1}^{\infty} f_{k}^{2}<+\infty$.
  Defining $O^{\iota}=\left\{(t, x)|\rho_{\rm reg} <\iota\right\}$,
 for $\rho_{\rm reg} $ in $(O^{\iota})^{c}$, when $m$ large enough, it holds $\rho_{m}>\frac{\iota}{2}$.
From Remark \ref{convergence of q=rho u}, we know that $\bar{\rho}_{m}\bar{\u}_{m}$ converges strongly to $\rho_{\rm reg} \u_{\rm reg} $ in $C\left([0,T];L^{\frac{3}{2}}\left( \O \right)\right)$, $\bar{\rho}_{m}\ra \rho_{\rm reg} $ strongly in $C\left([0,T];W^{1,3}\left( \O \right)\right)$, therefore they converge almost everywhere up to a subsequence.
 We claim that $\mathbf{F}_{k}\left(\bar{\rho}_{m},\frac{\bar{\rho}_{m}\bar{\u}_{m}}{\bar{\rho}_{m}}\right)\ra \mathbf{F}_{k}\left(\rho_{\rm reg} ,\frac{\rho_{\rm reg} \u_{\rm reg} }{\rho_{\rm reg} }\right)$ due to the strong convergence of $\bar{\rho}_{m}\mathbf{F}_{k}(\bar{\rho}_{m},\frac{\bar{\rho}_{m}\bar{\u}_{m}}{\bar{\rho}_{m}})\ra \rho_{\rm reg} \mathbf{F}_{k}(\rho_{\rm reg} ,\frac{\rho_{\rm reg} \u_{\rm reg} }{\rho_{\rm reg} })$.
 Actually, for the situation in $O^{\iota}$, $\rho_{\rm reg} \neq 0$, we have
\begin{align}
&\bar{\rho}_{m}\mathbf{F}_{k}\left(\bar{\rho}_{m},\frac{\bar{\rho}_{m}\bar{\u}_{m}}{\bar{\rho}_{m}}\right)-\rho_{\rm reg}\mathbf{F}_{k}\left(\rho_{\rm reg},\frac{\rho_{\rm reg}\u_{\rm reg}}{\rho_{\rm reg} }\right)\notag\\
=&\bar{\rho}_{m}\mathbf{F}_{k}\left(\bar{\rho}_{m},\frac{\bar{\rho}_{m}\bar{\u}_{m}}{\bar{\rho}_{m}}\right)-\rho_{\rm reg} \mathbf{F}_{k}\left(\bar{\rho}_{m},\frac{\bar{\rho}_{m}\bar{\u}_{m}}{\bar{\rho}_{m}}\right)\\
& +\rho_{\rm reg}\mathbf{F}_{k}\left(\bar{\rho}_{m},\frac{\bar{\rho}_{m}\bar{\u}_{m}}{\bar{\rho}_{m}}\right) -\rho_{\rm reg}\mathbf{F}_{k}\left(\rho_{\rm reg},\frac{\rho_{\rm reg}\u_{\rm reg}}{\rho_{\rm reg} }\right),\notag
\end{align}
so we get
\begin{align}
&\left|\rho_{\rm reg}\mathbf{F}_{k}\left(\bar{\rho}_{m},\frac{\bar{\rho}_{m}\bar{\u}_{m}}{\bar{\rho}_{m}}\right)-\rho_{\rm reg}\mathbf{F}_{k}\left(\rho_{\rm reg},\frac{\rho_{\rm reg}\u_{\rm reg}}{\rho_{\rm reg} }\right)\right|\notag\\
\ls & \left|\bar{\rho}_{m}\mathbf{F}_{k}\left(\bar{\rho}_{m},\frac{\bar{\rho}_{m}\bar{\u}_{m}}{\bar{\rho}_{m}}\right)-\rho_{\rm reg} \mathbf{F}_{k}\left(\bar{\rho}_{m},\frac{\bar{\rho}_{m}\bar{\u}_{m}}{\bar{\rho}_{m}}\right)\right|\\
&+\left|\bar{\rho}_{m}\mathbf{F}_{k}\left(\bar{\rho}_{m},\frac{\bar{\rho}_{m}\bar{\u}_{m}}{\bar{\rho}_{m}}\right)-\rho_{\rm reg} \mathbf{F}_{k}\left(\bar{\rho}_{m},\frac{\bar{\rho}_{m}\bar{\u}_{m}}{\bar{\rho}_{m}}\right)\right|\notag\\
\ls & f_{k} \left(\left|\bar{\rho}_{m}-\rho_{\rm reg}\right|+\left|\bar{\rho}_{m}\bar{\u}_{m}-\rho_{\rm reg}\u_{\rm reg}\right|\right)+\left|\bar{\rho}_{m}-\rho_{\rm reg}\right| \left|\mathbf{F}_{k}\left(\bar{\rho}_{m},\frac{\bar{\rho}_{m}\bar{\u}_{m}}{\bar{\rho}_{m}}\right)\right|. \notag
\end{align}
Therefore, for $\rho_{\rm reg}>0$, we have
\begin{align}\label{conver of Fk}
\left|\mathbf{F}_{k}\left(\bar{\rho}_{m},\frac{\bar{\rho}_{m}\bar{\u}_{m}}{\bar{\rho}_{m}}\right)-\mathbf{F}_{k}\left(\rho_{\rm reg},\frac{\rho_{\rm reg}\u_{\rm reg}}{{\rho} }\right)\right|
\rightarrow 0, \text{ almost everywhere as } m \rightarrow +\infty.
\end{align}
When $\rho_{\rm reg}=0$, $\left|\bar{\rho}_{m}\mathbf{F}_{k}\left(\bar{\rho}_{m},\frac{\bar{\rho}_{m}\bar{\u}_{m}}{\bar{\rho}_{m}}\right)\right|\ls f_{k}\left(|\bar{\rho}_{m}|+\left|\bar{\rho}_{m}\bar{\u}_{m}\right|\right)$,
this implies $\left|\mathbf{F}_{k}\left(\bar{\rho}_{m},\frac{\bar{\rho}_{m}\bar{\u}_{m}}{\bar{\rho}_{m}}\right)-0\right|\ls f_{k}\left(1+\left|\bar{\u}_{m}\right|\right)$, hence the integral is bounded, which gives that
\begin{align}\label{conver of Fk on vacuum}
\int_{0}^{T}\int_{\rho_{\rm reg}=0}\left|\mathbf{F}_{k}\left(\bar{\rho}_{m},\bar{\u}_{m}\right)-\mathbf{F}_{k}(\rho_{\rm reg},\u_{\rm reg})\right|^{2}\d x\d t
= 0 \text{ as }m\rightarrow +\infty.
\end{align}
So \eqref{conver of Fk on vacuum} along with \eqref{conver of Fk} yields
\begin{align}
&\int_{0}^{T}\int_{O^{\iota}}|\mathbf{F}_{k}(\bar{\rho}_{m},\bar{\u}_{m})-\mathbf{F}_{k}(\rho_{\rm reg},\u_{\rm reg})|^{2}\d x\d t\notag\\
= &\int_{0}^{T}\int_{O^{\iota}\cap\{\rho_{\rm reg}=0\}}|\mathbf{F}_{k}(\bar{\rho}_{m},\bar{\u}_{m})-\mathbf{F}_{k}(\rho_{\rm reg},\u_{\rm reg})|^{2}\d x\d t \\ &+\int_{0}^{T}\int_{O^{\iota}\cap\{\rho_{\rm reg} \neq0\}}|\mathbf{F}_{k}(\bar{\rho}_{m},\bar{\u}_{m})-\mathbf{F}_{k}(\rho_{\rm reg},\u_{\rm reg})|^{2}\d x\d t \rightarrow 0 \text{ as }m\rightarrow +\infty.\notag
\end{align}
Therefore, it holds
\begin{align}
&\int_{0}^{T}\int_{\O}|\mathbf{F}_{k}(\bar{\rho}_{m},\bar{\u}_{m})-\mathbf{F}_{k}(\rho_{\rm reg},\u_{\rm reg})|^{2}\d x\d t\\
=&\int_{O^{\iota}} |\mathbf{F}_{k}(\bar{\rho}_{m},\bar{\u}_{m})-\mathbf{F}_{k}(\rho_{\rm reg},\u_{\rm reg})|^{2}\d x\d t +\int_{(O^{\iota})^{c}} |\mathbf{F}_{k}(\bar{\rho}_{m},\bar{\u}_{m})-\mathbf{F}_{k}(\rho_{\rm reg},\u_{\rm reg})|^{2} \d x\d t\notag\\
&\rightarrow 0 \text{ as }m\rightarrow +\infty.\notag
\end{align}
Consequently, we have
\begin{align}
\mathbf{F}_{k}(\bar{\rho}_{m},\bar{\u}_{m})\ra \mathbf{F}_{k}(\rho_{\rm reg},\u_{\rm reg})\quad \text{in} \quad L^{2}\left([0,T]\times \O\right).
\end{align}
Therefore, $\Pi_{m}\left(\mathbf{F}_{k}(\bar{\rho}_{m},\bar{\u}_{m})\right)$ converges to $\mathbf{F}_{k}(\rho_{\rm reg},\u_{\rm reg})$ in $L^{2}\left([0,T]\times \O\right)$ due to the continuity of projection operator $\Pi_{m}$.
On the other hand, we have the strong convergence of $\bar{\rho}_{m}$ in $L^{2}\left([0,T]\times \O\right)$, so
\begin{equation}\label{con. rhoF in m}
\Pi_{m}\left(\bar{\rho}_{m}\Pi_{m}\left(\mathbf{F}_{k}(\bar{\rho}_{m},\bar{\u}_{m})\right)\right)
\ra \rho_{\rm reg}\mathbf{F}_{k}(\rho_{\rm reg},\u_{\rm reg})\text{ in }L^{1}\left([0,T]\times \O\right).
\end{equation}
The strong convergence of
$\Pi_{m}\left(\bar{\rho}_{m}\Pi_{m}\left(\mathbf{F}_{k}\left(\bar{\rho}_{m},\bar{\u}_{m}\right)\right)\right)$ in $L^{2}\left([0,T]\times \O\right)$ is required in the convergence of the stochastic integral. Since
$\int_{0}^{T}\int_{\O}\Pi_{m}\left(\bar{\rho}_{m}\mathbb{F}_{k}\left(\bar{\rho}_{m},\bar{\u}_{m}\right)\right)^{2}\d x\d t $ is bounded,
so we get $\bar{\rho}_{m}\Pi_{m}\left(\mathbf{F}_{k}\left(\bar{\rho}_{m},\bar{\u}_{m}\right)\right)\ra \rho_{\rm reg}\mathbf{F}_{k}\left(\rho_{\rm reg},\u_{\rm reg}\right)$ in $L^{r}\left(\Omega;L^{2}\left([0,T]\times \O\right)\right)$.
By dominated convergence theorem, the strong convergence of $\bar{\rho}_{m}$, and $\left\|\rho_{\rm reg}\right\|_{L_t^{\infty}L_x^{\infty}} \ls \left\|\rho_{\rm reg} \right\|_{L_t^{\infty}H_x^{9}}$, we have
\begin{align}
&\int_{0}^{T}\int_{\O}\left|\bar{\rho}_{m}\Pi_{m}\left(\mathbf{F}_{k}(\bar{\rho}_{m},\bar{\u}_{m})\right)-
\rho_{\rm reg}\mathbf{F}_{k}(\rho_{\rm reg},\u_{\rm reg})\right|^{2}\d x\d t\notag\\
\ls&2\int_{0}^{T}\int_{\O}\left|\bar{\rho}_{m}\Pi_{m}\left[\mathbf{F}_{k}(\bar{\rho}_{m},\bar{\u}_{m})\right]-\rho_{\rm reg} \Pi_{m}\left(\mathbf{F}_{k}(\bar{\rho}_{m},\bar{\u}_{m})\right)\right|^{2}\d x\d t\notag\\
 &\quad\quad+2\int_{0}^{T}\int_{\O}\left| \rho_{\rm reg}\Pi_{m}\left(\mathbf{F}_{k}(\bar{\rho}_{m},\bar{\u}_{m})\right)-\rho_{\rm reg}\mathbf{F}_{k}(\rho_{\rm reg},\u_{\rm reg})\right|^{2}\d x\d t\\
\ls& 2f_{k}^{2}\int_{0}^{T}\int_{\O}|\bar{\rho}_{m}-\rho_{\rm reg}|^{2}\d x \d t\notag\\
&+\left\|\rho_{\rm reg} \right\|_{L_t^{\infty}L_x^{\infty}}\int_{0}^{T}\int_{\O}\left|\Pi_{m}\left(\mathbf{F}_{k}(\bar{\rho}_{m},\bar{\u}_{m})\right)-\mathbf{F}_{k}(\rho_{\rm reg},\u_{\rm reg})\right|^{2}\d x\d t\notag\\
&\rightarrow 0 \text{ as }m\rightarrow +\infty.\notag
\end{align}
This shows that
\begin{equation}\label{L2 con. rhoF in m}
\Pi_{m}\left(\bar{\rho}_{m}\Pi_{m}\left(\mathbf{F}_{k}(\bar{\rho}_{m},\bar{\u}_{m})\right)\right)\ra \rho_{\rm reg}\mathbf{F}_{k}(\rho_{\rm reg},\u_{\rm reg})\text{ in }L^{2}\left([0,T]\times \O\right).
\end{equation}

Now combined with (\ref{L2 con. rhoF in m}) and $\bar{W}_{m}\ra W_{\rm reg}$, by lemma \ref{lemma sto int convergence}, it holds\\
  $\int_{\O}\int_{0}^{T}\Pi_{m}\left(\bar{\rho}_{m}\Pi_{m}\left(\mathbf{F}_{k}(\bar{\rho}_{m},\bar{\u}_{m})\right)\right)\d \bar{W}_{m}\d x
\ra \int_{\O}\int_{0}^{T}\rho_{\rm reg}\mathbf{F}_{k}(\rho_{\rm reg},\u_{\rm reg})\d W_{\rm reg}\d x$ as $m\ra +\infty$ $\bar{\mathbb{P}}$ a.s. This completes the proof of the lemma. \hfill$\square$
\end{proof}

Notice that even $\eta$ and $\delta$ will vanish finally, by the Lipschitz continuity of $\rho \mathbf{F}_{k}(\rho,\u )$, the convergence of $\rho \mathbf{F}_{k}(\rho,\u )$ still holds without good regularity of $\rho$.\\

\textbf{Step 4:} $\left(\left(\bar{\Omega},\bar{\mathcal{F}},\bar{\mathbb{P}} \right),\rho_{\rm reg}, \u_{\rm reg}, W_{\rm reg}\right)$ is a martingale solution  after taking the limit.\\
%

By Vitali's convergence theorem and lemma \ref{Conver of sto int for m layer}, we have the following proposition:
\begin{proposition}
$\left(\left(\bar{\Omega},\bar{\mathcal{F}},\bar{\mathbb{P}} \right),\rho_{\rm reg}, \u_{\rm reg}, W_{\rm reg}\right)$ is a martingale solution to the system
\begin{equation}\label{system involving eps}
\left\{\begin{array}{l}\vspace{1.2ex}
\left(\rho_{\rm reg}\right)_{t}+\operatorname {div}(\rho_{\rm reg}\u_{\rm reg})= \eps \triangle \rho_{\rm reg},\\
\d  \left(\rho_{\rm reg}\u_{\rm reg}\right)+\left(\operatorname{div}(\rho_{\rm reg}\u_{\rm reg} \otimes \u_{\rm reg})+ \nabla\left( a\rho_{\rm reg}^{\gamma}\right)-\operatorname{div}(\rho_{\rm reg}
\mathbb{D} \u_{\rm reg})\right)\d t\\
=\left(-r_{0}|\u_{\rm reg}|^{2}\u_{\rm reg}-r_{1} \rho|\u_{\rm reg}|^{2} \u_{\rm reg}-r_{2}\u_{\rm reg}\right)\d t\\
\quad +\left(\frac{11}{10}\eta \nabla \rho_{\rm reg}^{-10}+\delta \rho_{\rm reg}\nabla \triangle^{9}\rho_{\rm reg}+\kappa\rho_{\rm reg}\left(\nabla\left(\frac{\triangle \sqrt{\rho_{\rm reg}}}{\sqrt{\rho_{\rm reg}}}\right)\right)\right)\d t\\
\quad -\left(\eps \nabla \rho_{\rm reg}\nabla \u_{\rm reg}\right)\d t-\left(\eps \triangle^{2} \u_{\rm reg}\right)\d t +\left(\rho_{\rm reg} \mathbb{F}(\rho_{\rm reg},\u_{\rm reg})\right)\d W_{\rm reg}.
\end{array}\right.
\end{equation}
\end{proposition}
We denote the energy as
\begin{align}
e_{\rm reg}(t)=\int_{\O}\left(\frac{1}{2}\rho_{\rm reg}|\u_{\rm reg}|^{2}+\frac{a}{\gamma}\rho_{\rm reg}^{\gamma} +\frac{\eta}{10}\rho_{\rm reg}^{-10}+\frac{\kappa}{2}|\nabla\sqrt{\rho_{\rm reg}}|^{2}+\frac{\delta}{2}|\nabla \triangle^{4} \rho_{\rm reg}|^{2}\right)\d x,
\end{align}
then we compute the estimates
\begin{equation}\label{energy estimate eps 1}
\mathbb{E}\left[e_{\rm reg}(t)^{r}\right]\ls \mathbb{E}\left[e_{\rm reg}(0)^{r}\right], r> 4,
\end{equation}
and
\begin{align}\label{energy estimate eps 2}
&\mathbb{E}\left[\left(\eps \frac{4a}{\gamma}\int_{0}^{t}\int_{\O}\left|\nabla \rho_{\rm reg}^{\frac{\gamma}{2}}\right|^{2}\d x\d s+\eps\eta \frac{11}{25} \int_{0}^{t}\int_{\O}\left|\nabla \rho_{\rm reg}^{-5}\right|^{2}\d x\d s+\eps\delta \int_{0}^{t}\int_{\O}\left|\triangle^{5} \rho_{\rm reg}\right|^{2}\d x\d s\right.\right.\notag\\
&\quad\left.\left.+\eps \kappa \int_{0}^{t}\int_{\O} \rho_{\rm reg}\left|\nabla^{2}\ln\rho_{\rm reg}\right|^{2}\d x\d s\int_{0}^{t}+\eps\int_{\O}\left|\triangle \u_{\rm reg}\right|^{2}\d x\d s+\int_{\O}\rho_{\rm reg}\left|\mathbb{D}\u_{\rm reg}\right|^{2}\d x\d s\right.\right.\notag\\
&\left.\left.\quad+\int_{0}^{t}\int_{\O}\left(r_{0}\left|\u_{\rm reg}\right|^{4}+r_{1}\rho_{\rm reg}\left|\u_{\rm reg}\right|^{4}+r_{2}\left|\u_{\rm reg}\right|^{2}\right)\d x\d s\right)^{r}\right]\\
&\ls C\left(1+\mathbb{E}\left[e_{\rm reg}(0)^{r}\right]\right),\notag
\end{align}
 for $C$ depends on $r$ and $T$.
\section{Vanishing artificial viscosity limit: $\eps\ra 0$}
  Before sending $\eps$ to $0$, we need to develop the stochastic B-D entropy with uniform energy estimates independent of $\kappa,\eps$. In order to construct the energy estimates uniformly in $\eps$, we first give the stochastic B-D entropy.

\subsection{The stochastic B-D entropy estimates}
In this subsection we define the approximated B-D entropy
\begin{align}\label{BD entropy with REG}
\tilde{e}_{\rm reg}=&\int_{\O}\left(\frac{1}{2}\rho_{\rm reg}\left|\u_{\rm reg}+\nabla \ln\rho_{\rm reg}\right|^{2}
+\int_{1}^{\rho_{\rm reg}}\frac{p(z)}{z}\d z\right.\\
&\left.\qquad+\frac{\eta}{10}\rho_{\rm reg}^{-10}+\frac{\delta}{2}\left|\nabla\triangle^{4}\rho_{\rm reg}\right|^{2}+\frac{\kappa}{2}\left|\nabla\sqrt{\rho_{\rm reg}}\right|^{2}+r_{2}\int_{\O}\ln_{-}\rho_{\rm reg}\right)\d x,\notag
\end{align}
and we have the following lemma.
\begin{lemma}
We conclude the following estimates: there hold
\begin{equation}
\begin{aligned}
\mathbb{E}\left[\sup\limits_{t\in[0,T]}\tilde{e}_{\rm reg}(t)^{r}\right]&\ls C,
\end{aligned}
\end{equation}
\begin{align}\label{BD estimate2}
&\mathbb{E}\left[\left(\eps \frac{4a}{\gamma}\int_{0}^{T}\int_{\O}\left|\nabla \rho_{\rm reg}^{\frac{\gamma}{2}}\right|^{2}\d x\d s\right)^{r}\right]\ls C,
\quad \mathbb{E}\left[\left(\eps\eta \frac{11}{25} \int_{0}^{t}\int_{\O}\left|\nabla \rho_{\rm reg}^{-5}\right|^{2}\d x\d s\right)^{r}\right]\ls C,\notag\\
&\mathbb{E}\left[\left(\eps \delta \int_{0}^{t}\int_{\O}\left|\triangle^{5} \rho_{\rm reg}\right|^{2}\d x\d s\right)^{r}\right]\ls C,
\quad \mathbb{E}\left[\left(\frac{1}{2}\int_{0}^{T}\int_{\O}\rho_{\rm reg}\left|\nabla\u_{\rm reg}-\nabla\u_{\rm reg}^{\top}\right|^{2}\right)^{r}\right]\ls C,\notag\\
 &\mathbb{E}\left[\left(\eps \kappa \int_{0}^{t}\int_{\O} \rho_{\rm reg}\left|\nabla^{2}\ln\rho_{\rm reg}\right|^{2}\d x\d s\right)^{r}\right]\ls C,
\quad \mathbb{E}\left[\left(r_{0}\int_{0}^{T}\int_{\O}\left|\u_{\rm reg}\right|^{4}\d x\d s\right)^{r}\right]\ls C,\notag\\
 &\mathbb{E}\left[\left(r_{1}\int_{0}^{T}\int_{\O}\rho_{\rm reg} \left|\u_{\rm reg}\right|^{4}\d x\d s\right)^{r}\right]\ls C,
\quad \mathbb{E}\left[\left(r_{2}\int_{0}^{T}\int_{\O}\left|\u_{\rm reg}\right|^{2}\d x\d s\right)^{r}\right]\ls C,\\
&\mathbb{E}\left[\left(\int_{0}^{t}\int_{\O}\eta \frac{11}{25} \left|\nabla \rho_{\rm reg}^{-5}\right|^{2}\d x\d s\right)^{r}\right]\ls C,
\quad \mathbb{E}\left[\left(\int_{0}^{t}\int_{\O}\gamma\rho_{\rm reg}^{\gamma}\left|\nabla\ln\rho_{\rm reg}\right|^{2}\d x\d s\right)^{r}\right]\ls C,\notag\\
&\mathbb{E}\left[\left(\int_{0}^{t}\int_{\O}\delta\left|\triangle^{5}\rho_{\rm reg}\right|^{2}\d x\d s\right)^{r}\right] \ls C,
\quad \mathbb{E}\left[\left(\int_{0}^{t}\int_{\O}\kappa \rho_{\rm reg}\left|\nabla^{2}\ln\rho_{\rm reg}\right|^{2}\d x\d s\right)^{r}\right]\ls C,\notag
\end{align}
where $C$ depends on $\delta,\eta,r,T,r_{0},r_{1}, \mathbb{E}\left[e_{\rm reg}(0)^{r}\right],\mathbb{E}\left[e_{\rm reg}(0)^{\frac{6r}{5}}\right]$ and $\mathbb{E}\left[e_{\rm reg}(0)^{k_{{\rm I}}r}\right]$, $\mathbb{E}\left[\tilde{e}_{\rm reg}(0)^{r}\right]$, $i=1,2,3$, $r>4$.
\end{lemma}
\begin{proof} For convenience, we denote $\rho_{\rm reg}$ as $\rho$ and denote $\u_{\rm reg}$ as $\u$ in this process of proof.
 Denoting $f\triangleq f(\zeta,q)=\rho\u\cdot \nabla\ln\rho=q\zeta$, $q=\rho\u$, $f_{\zeta}=q=\rho\u,f_{q}=\zeta$, $f_{qq}=0$, by It\^o's formula, it holds
\begin{equation}\label{Ito formula in eps layer}
\d\left(\int_{\O}\rho\u \cdot \nabla\ln\rho \d x\right)=\int_{\O}\d \left(\rho\u\right)\cdot \nabla\ln\rho \d x+\int_{\O}\rho\u \cdot\d \nabla\ln\rho \d x.\\
\end{equation}
 The first term reduces to
\begin{align}
&\int_{\O} \d (\rho\u)\cdot\nabla\ln\rho \d x \notag\\
=& \left(\int_{\O}-\odiv\left(\rho\u\otimes \u\right)\cdot\nabla\ln\rho \d x-\int_{\O}\nabla\rho^{\gamma}\cdot\nabla\ln\rho \d x+\int_{\O}\odiv\left(\rho \mathbb{D}\u\right)\cdot\nabla\ln\rho \d x \right.\notag\\
 &\left.+\frac{11}{10}\eta\int_{\O}\nabla\rho^{-10}\cdot\nabla\ln\rho \d x+\delta\int_{\O}\rho\nabla\triangle^{9}\cdot\nabla\ln\rho \d x +\kappa\int_{\O}\odiv\left(\rho \nabla^{2}\ln\rho\right)\cdot\nabla\ln\rho \d x\right.\notag\\
 &\left.+\eps \int_{\O}\left(\nabla\rho\nabla\u\right)\cdot\nabla\ln\rho \d x-\eps \int_{\O}\triangle^{2}\u\cdot\nabla\ln\rho \d x\right. \\
 &\left.-r_{0}\int_{\O}|\u|^{2}\u\cdot \nabla\ln\rho \d x-r_{1}\int_{\O}\rho|\u|^{2}\u\cdot \nabla\ln\rho \d x-r_{2}\int_{\O}\u\cdot \nabla\ln\rho \d x\right)\d t\notag\\
 & + \int_{\O}\rho \mathbb{F}(\rho,\rho\u)\d W \cdot \nabla\ln\rho \d x. \notag
\end{align}
The second term in \eqref{Ito formula in eps layer} is written as
\begin{equation}
\int_{\O}\rho\u \cdot\d(\nabla\ln\rho) \d x=\int_{\O}\odiv\left(\rho\u\right)\frac{\odiv\left(\rho\u\right)-\eps \triangle\rho}{\rho}\d x\d t.
\end{equation}
From the mass equation, it holds
\begin{align}
&\quad \d \left(\int_{\O}\frac{1}{2}\rho\left|\nabla \ln\rho\right|^{2}\d x\right)\\
&= \left(-\int_{\O} \rho\nabla\u:\nabla\ln\rho\otimes\nabla\ln\rho \d x - \int_{\O}\triangle\rho\odiv\u \d x+ \eps \int_{\O}\frac{\nabla\rho\cdot\nabla\triangle\rho}{\rho}\d x\right)\d t. \notag
\end{align}
After the simplification of the deterministic part\cite{Vasseur-Yu-q2016,BD2006}, we have
\begin{align}
&\quad \int_{\O} \d \left(\rho\u\right)\cdot\nabla\ln\rho \d x + \int_{\O}\rho\u \cdot\d \ln\rho \d x + \d \left(\int_{\O}\frac{1}{2}\rho\left|\nabla \ln\rho\right|^{2}\d x\right)\notag\\
&=\left(\int_{\O}\rho \left|\mathbb{D}\u\right|^{2} \d x-\frac{1}{2}\int_{\O} \frac{\rho}{2}\left|\nabla\u-\nabla\u^{\top}\right|^{2}\d x-\int_{\O}\nabla\rho^{\gamma}\cdot\nabla\ln\rho \d x\right.\notag\\
&\quad\left. + \frac{11}{10}\eta\int_{\O}\nabla\rho^{-10}\cdot\nabla\ln\rho \d x+\delta\int_{\O}\rho\nabla\triangle^{9}\rho\cdot\nabla\ln\rho \d x\right.\\
&\quad\left.+ \kappa\int_{\O}\odiv\left(\rho \nabla^{2}\ln\rho\right)\cdot\nabla\ln\rho \d x+\eps \int_{\O}\left(\nabla\rho\nabla\u\right)\nabla\ln\rho \d x  \right.\notag\\
&\quad\left.+\eps \int_{\O}\frac{\nabla\rho\cdot\nabla\triangle\rho}{\rho} \d x-\eps \int_{\O}\frac{\odiv\left(\rho\u\right)\triangle\rho}{\rho}\d x-\eps \int_{\O}\triangle^{2}\u \cdot \nabla\ln\rho \d x\right.\notag\\
&\quad\left. - r_{0}\int_{\O}\left|\u\right|^{2}\u\cdot \nabla\ln\rho \d x-r_{1}\int_{\O}\rho\left|\u\right|^{2}\u\cdot \nabla\ln\rho \d x
 -r_{2}\int_{\O}\u\cdot \nabla\ln\rho \d x  \right)\d t\notag\\
&\quad+ \left(\int_{\O}\rho \mathbb{F}(\rho,\rho\u)\cdot \nabla\ln\rho \d x\right)\d W . \notag
\end{align}
Together with the energy balance of $\d \left(\int_{\O}\frac{1}{2}\rho\u^{2}\d x\right)$ \eqref{energy balance}, it follows
\begin{align}\label{sto B-D balance}
&\d\left(\int_{\O}\left[\frac{1}{2}\rho\left|\u+\nabla \ln\rho\right|^{2}+\frac{\eta}{10}\rho^{-10}+\frac{\delta}{2}\left|\nabla\triangle^{4}\rho\right|^{2}+\frac{\kappa}{2}\left|\nabla\sqrt{\rho}\right|^{2}+\int_{1}^{\rho}\frac{p(z)}{z}\d z\right]\d x\right)\notag\\
&+\left(\eps \frac{4a}{\gamma}\int_{\O}\left|\nabla\rho^{\frac{\gamma}{2}}\right|^{2}\d x+\frac{11}{25}\eps\eta\int_{\O}\left|\nabla \rho^{-5}\right|^{2}\d x+\eps\int_{\O}\left|\triangle \u\right|^{2}\d x+\eps \kappa \int_{\O}\rho\left|\nabla^{2}\ln\rho\right|^{2}\d x\right.\notag\\
&\left.\quad +r_{0}\int_{\O}\left|\u\right|^{4}\d x+r_{1}\int_{\O}\rho\u^{4}\d x+r_{2}\int_{\O}\left|\u\right|^{2}\d x+\frac{1}{2}\int_{\O}\frac{\rho}{2}\left|\nabla\u-\nabla\u^{\top}\right|^{2}\d x\right)\d t\\
&+\left(\int_{\O}\gamma\rho^{\gamma}\left|\nabla\ln\rho\right|^{2}\d x+\int_{\O}\frac{11}{25}\eta\left|\nabla \rho^{-5}\right|^{2}\d x+\kappa\int_{\O}\rho\left|\nabla^{2}\ln\rho\right|^{2}\d x\d t+\delta\int_{\O}\left|\triangle^{5}\rho\right|^{2}\d x\right)\d t \notag\\
&= \left(\int_{\O}\rho\mathbb{F}\left(\rho,\u\right)\cdot \u \d x\right)\d W + \left(\int_{\O}\frac{1}{2}\rho\left|\mathbb{F}\left(\rho,\u\right)\right|^{2}\d x\right)\d t
+\left(\int_{\O}\rho\mathbb{F}(\rho,\rho\u)\cdot \nabla\ln\rho \d x\right)\d W  \notag\\
&+\left(\eps \int_{\O}\frac{\nabla\rho\nabla\triangle\rho}{\rho}\d x+\eps\int_{\O}\nabla\rho\nabla\u\nabla\ln\rho \d x\d t-\eps \int_{\O}\odiv (\rho\u)\frac{\triangle\rho}{\rho}\d x-\eps \int_{\O}\triangle\u\cdot\nabla\triangle\ln\rho \d x\right.\notag\\
&\left.\quad-r_{0}\int_{\O}\left|\u\right|^{2}\u\cdot\nabla\ln\rho \d x
-r_{1}\int_{\O}\rho\left|\u\right|^{2}\u\cdot\nabla\ln\rho \d x-r_{2}\int_{\O}\u\cdot \nabla\ln\rho \d x\right)\d t\notag\\
\triangleq & I_{1}+I_{2}+I_{3}+I_{4}+I_{5}+I_{6}+I_{7}+I_{8}+I_{9}+I_{10},\notag
\end{align}
where $\displaystyle \left(\int_{\O}\gamma\rho^{\gamma}\left|\nabla\ln\rho\right|^{2}\d x\right)\d t=\left(\int_{\O} \left|\nabla\rho^{\frac{\gamma}{2}}\right|^{2}\d x\right) \d t=\left(\int_{\O} \nabla\rho^{\gamma}\cdot \nabla\ln\rho \d x\right)\d t$.

For $I_{1}=\left(\int_{\O}\rho\mathbb{F}\left(\rho,\u\right)\cdot \u \d x\right)\d W$, similarly as \eqref{estimate of rho F u}, we have
\begin{align}
 \mathbb{E}\left[\left|\int_{0}^{t}\int_{\O}\rho\mathbb{F}\left(\rho,\u\right)\cdot \u \d x\d W\right|^{r} \right]
\ls C\left(\mathbb{E}\left[ e_{\rm reg}(0)^{r}\right]+1\right), \quad r>4,
\end{align}
with $C$ depends on $r,t$ and $\sum\limits_{k=1}^{+\infty} f^{2}_{k}$.

For term $\displaystyle I_{2}=\left(\int_{\O}\frac{1}{2}\rho|\mathbb{F}\left(\rho,\u\right)|^{2}\d x\right)\d s$, similarly as \eqref{Jensen used example}, we estimate
\begin{align}
\quad \mathbb{E}\left[\left|\int_{0}^{t}\int_{\O}\frac{1}{2}\rho|\mathbb{F}\left(\rho,\u\right)|^{2}\d x\d s\right|^{r}\right]
\ls C \left(\mathbb{E}\left[e_{\rm reg}(0)^{r}\right]+1\right),
\end{align}
where $C$ depends on $r,t$ and $\sum\limits_{k=1}^{+\infty} f^{2}_{k}$.

For $\displaystyle I_{3}=\left(\int_{\O}\rho\mathbb{F}(\rho,\u)\cdot\nabla\ln\rho \d x\right)\d W$, it holds
\begin{align}
&\quad \mathbb{E}\left[\left|\int_{0}^{t}\int_{\O}\rho\mathbb{F}(\rho,\u)\cdot\nabla\ln\rho \d x\d W\right|^{r}\right]
\ls C\mathbb{E}\left[\left(\int_{0}^{t}\left|\int_{\O}\rho\mathbb{F}(\rho,\u)\cdot\nabla\ln\rho \d x\right|^{2}\d s\right)^{\frac{r}{2}}\right]\\
&\ls C\mathbb{E}\left[\left(\int_{0}^{t}\sum\limits_{k=1}^{+\infty}f_{k}^2\left(\left\| \rho\nabla\ln\rho\right\|_{L^{1}\left( \O \right)}^2 +\left\| \left|\rho\u\right| \left|\nabla\ln\rho\right|\right\|_{L^{1}\left(\O\right)}^2 \right)\d s\right)^{\frac{r}{2}}\right], \notag
\end{align}
where $C$ depends on $r$ and $\sum\limits_{k=1}^{+\infty}f^{2}_{k}$. 
 Based on \eqref{energy estimate eps 1}, by H\"older's inequality and Cauchy's inequality, it holds
\begin{align}
&\int_{0}^{t}\left|\int_{\O}\sqrt{\rho}\sqrt{\rho}\nabla\ln\rho \d x\right|^{2}\d s
\ls  \int_{0}^{t} \int_{\O}\rho \d x\int_{\O}\rho\left|\nabla\ln\rho\right|^{2}\d x \d s\\
\ls & C \sup\limits_{s\in[0,t]} (e_{\rm reg}(s)+1)\int_{0}^{t} \int_{\O}\rho|\nabla\ln\rho|^{2}\d x\d s, \notag
\end{align}
so we have
\begin{align}
&\mathbb{E}\left[\left(\int_{0}^{t}\left|\int_{\O}\sqrt{\rho}\sqrt{\rho}\nabla\ln\rho \d x\right|^{2}\d s\right)^{\frac{r}{2}}\right]\\
\ls & C(\mathbb{E}[e_{\rm reg}(0)]^{r}+1)+\mathbb{E}\left[ \left(\int_{0}^{t} \int_{\O}\rho|\nabla\ln\rho|^{2}\d x\d s \right)^{r}\right],\notag
\end{align}
where $C$ depends on $r,t$ and $\sum\limits_{k=1}^{+\infty} f^{2}_{k}$, $r>4$.
\begin{align}
&\int_{0}^{t}\left(\int_{\O}\left|\rho\u\right| \left|\nabla\ln\rho\right| \d x\right)^{2}\d s
\ls  \int_{0}^{t}\left( \int_{\O}\rho\left|\u\right|^{2}\d x\right)\left(\int_{\O}\rho\left|\nabla\ln\rho\right|^{2}\d x\right) \d s \\
\ls &C \sup\limits_{s\in[0,t]} (e_{\rm reg}(s)+1)\int_{0}^{t} \int_{\O}\rho|\nabla\ln\rho|^{2}\d x\d s, \notag
\end{align}
Hence we have
\begin{align}
\mathbb{E}\left[\left(\int_{0}^{t}I_{3}\right)^{r}\right]
\ls & C(\mathbb{E}[e_{\rm reg}(0)]^{r}+1)+\mathbb{E}\left[ \left(\int_{0}^{t} \int_{\O}\rho|\nabla\ln\rho|^{2}\d x\d s \right)^{r}\right].\notag
\end{align}
For the term $\displaystyle  I_{4}=\eps \left(\int_{\O}\frac{\nabla\rho\cdot\nabla\triangle\rho}{\rho}\d x\right)\d s$, 
for $\kappa>0$, 
since $\sqrt{\eps \delta}\triangle^{5}\rho\in L^{2}([0,t]\times \O)$,
$H^{10}\left( \O \right)\hookrightarrow C^{8,\gamma}\left( \O \right)$, 
we do the following estimate
\begin{align}
 &\mathbb{E}\left[\left|\int_{0}^{t}\eps \int_{\O}\frac{\nabla\rho\nabla\triangle\rho}{\rho}\d x\d s\right|^{r}\right]\notag\\
\ls & \mathbb{E}\left[\left(\left\|\rho^{-1}\right\|_{L_t^{\infty}L_x^{10}}\int_{0}^{t} \eps \left(\int_{\O} \left|\nabla\rho\cdot\nabla\triangle\rho\right|^{\frac{10}{9}}\d x\right)^{\frac{9}{10}}\d s\right)^{r}\right]\notag\\
\ls & \mathbb{E} \left[\left(\left\|\rho^{-1}\right\|_{L_t^{\infty}L_x^{10}}\left\|\nabla\sqrt{\rho}\right\|_{L_t^{\infty}L_x^{2}}\eps\int_{0}^{t}\left(\int_{\O}\left|\nabla\triangle\rho\right|^{\frac{5}{2}}\d x\right)^{\frac{2}{5}}\d s\right)^{r}\right]\notag\\
\ls &\mathbb{E} \left[\left(\left\|\rho^{-1}\right\|_{L_t^{\infty}L_x^{10}}\left\|\nabla\sqrt{\rho}\right\|_{L_t^{\infty}L_x^{2}}\right)^{r}\left(\eps\int_{0}^{t}\left(\int_{\O}\left|\nabla\triangle\rho\right|^{5}\d x\right)^{\frac{1}{5}}\d s\right)^{r}\right]\\
\ls &\mathbb{E} \left[\left(\left\|\rho^{-1}\right\|_{L_t^{\infty}L_x^{10}}\left\|\nabla\sqrt{\rho}\right\|_{L_t^{\infty}L_x^{2}}\right)^{r}\left(\eps^{\frac{1}{2}}\delta^{-\frac{1}{2}}
\int_{0}^{t}\left(\int_{\O}\left|(\eps\delta)^{\frac{1}{2}}\triangle^{5}\rho\right|^{2}\d x\right)^{\frac{1}{2}}\d s\right)^{r}\right]\notag\\
\ls & \eps^{\frac{r}{2}}\delta^{-\frac{r}{2}}\mathbb{E} \left[\left(\left\|\rho^{-1}\right\|_{L_t^{\infty}L_x^{10}}\left\|\nabla\sqrt{\rho}\right\|_{L_t^{\infty}L_x^{2}}\right)^{r}\left(\int_{0}^{t}\left(\int_{\O}\left|(\eps\delta)^{\frac{1}{2}}\triangle^{5}\rho\right|^{2}\d x\right)\d s\right)^{\frac{r}{2}}\right]\notag\\
\ls &\eps^{\frac{r}{2}}\delta^{-\frac{r}{2}} \left(\mathbb{E}\left[\sup\limits_{s\in [0,t]}e_{\rm reg}(s)^{\frac{6r}{5}}\right]\right)^{\frac{1}{2}}\left(\mathbb{E}\left[\left(\int_{0}^{t}\int_{\O}\left|(\eps\delta)^{\frac{1}{2}}\triangle^{5}\rho\right|^{2}\d x\d s\right)^{r}\right]\right)^{\frac{1}{2}}\notag\\
\ls & C\eps^{\frac{r}{2}},\notag
\end{align}
where $C$ depends on $\delta,\eta,\kappa,r,t$ and $ \mathbb{E}\left[e_{\rm reg}(0)^{\frac{6r}{5}}\right]$, with the upper bound $$\left(\mathbb{E}\left[\sup\limits_{s\in [0,t]} e_{\rm reg}(s)^{\frac{6r}{5}}\right]\right)^{\frac{1}{2}}\ls C \left(\left(\mathbb{E}\left[e_{\rm reg}(0)^{\frac{6r}{5}}\right]\right)^{\frac{1}{2}}+1\right).$$
 Notice that the quantum term is indispensable for the regularized system because it is necessary for the estimate of $I_{4}$, which arises from the artificial viscosity $\eps \triangle \rho$.

For $\displaystyle I_{5}=\eps\left(\int_{\O}\nabla\rho\nabla\u\nabla\ln\rho \d x\right)\d s$, we treat it similarly to \cite{Zatorska2012}, the process is still given here. By Sobolev-Poincar\'e's inequalities on torus, $\displaystyle \int_{\O}\nabla\u\d x=0$, we have $\left\|\nabla\u\right\|_{L^{6}\left( \O \right)}\ls C \left\|\triangle \u\right\|_{L^{2}\left( \O \right)}$,
\begin{equation}
\begin{aligned}
&\quad \left|\int_{\O}\eps\nabla\rho\nabla\u\nabla\ln\rho\d x\right|\ls \eps C\left\|\nabla\u\right\|_{L^{6}\left( \O \right)}\left\|\rho^{-1}\right\|_{L^{10}\left( \O \right)}\left\|\rho\right\|^2_{W^{1,\frac{30}{11}}\left( \O \right)}\\
&\ls \eps C\left\|\triangle \u\right\|_{L^{2}\left( \O \right)}\left\|\rho^{-1}\right\|_{L^{10}\left( \O \right)}\left\|\rho\right\|^2_{W^{1,\frac{30}{11}}\left( \O \right)},
\end{aligned}
\end{equation}
by $C_{r}$-inequality.
Hence by H\"older's inequality and Cauchy's inequality, we obtain
\begin{align}
&\quad \mathbb{E}\left[\left(\int_{0}^{t}\int_{\O}\eps\nabla\rho\nabla\u\nabla\ln\rho \d x\d s\right)^{r}\right]\notag\\
&\ls  \mathbb{E}\left[\left(\int_{0}^{t}\eps C\left\|\triangle \u\right\|_{L^{2}\left( \O \right)}\left\|\rho^{-1}\right\|_{L^{10}\left( \O \right)}\left\|\rho\right\|^2_{W^{1,\frac{30}{11}}\left( \O \right)}\d s\right)^{r}\right]\\
&\ls  \mathbb{E}\left[\left(\eps \left(\int_{0}^{t}\left\|\triangle \u\right\|^2_{L^{2}\left( \O \right)}\d s\right)^{\frac{1}{2}}\left(\int_{0}^{t}\left\|\rho^{-1}\right\|^2_{L^{10}\left( \O \right)}\left\|\rho\right\|^{3}_{H^{9}\left( \O \right)}\left\|\rho\right\|_{H^{10}\left( \O \right)}\d s\right)^{\frac{1}{2}}\right)^{r}\right]\notag\\
&\ls \frac{\eps^{r}}{3} \mathbb{E}\left[\left(\int_{0}^{t}\left\|\triangle\u\right\|^2_{L^{2}\left( \O \right)}\d s\right)^{r}\right]+\eps^{r}3\mathbb{E}\left[\left(\int_{0}^{t}\left\|\rho^{-1}\right\|^2_{L^{10}\left( \O \right)}\left\|\rho\right\|^{3}_{H^{9}\left( \O \right)}\left\|\rho\right\|_{H^{10}\left( \O \right)}\d s\right)^{r}\right]\notag\\
&\ls \frac{\eps^{r}}{3}\mathbb{E}\left[\left(\int_{0}^{t}\left\|\triangle \u\right\|^2_{L^{2}\left( \O \right)}\d s\right)^{r}\right] + C\eps^{\frac{r}{2}} ,\notag
\end{align}
where $C$ depends on $\delta, \eta, r, t$ and $\mathbb{E}\left[e_{\rm reg}(0)^{k_{1}r}\right]$, in which $k_{1}$ is a specific positive constant. The last inequality is obtained similarly to the estimate for $I_{4}$. 

For $\displaystyle I_{6}=-\eps \left(\int_{\O}\odiv (\rho\u)\frac{\triangle\rho}{\rho}\d x\right)\d s$ and $\displaystyle I_{7}=-\eps \left(\int_{\O}\triangle\u\cdot\nabla\triangle\ln\rho \d x \right)\d s $, similarly to $I_{4}$, there holds
\begin{align}
&\mathbb{E}\left[\left(\int_{0}^{t}\int_{\O}-\eps\frac{\odiv(\rho\u)\triangle\rho}{\rho}\d x\d s\right)^{r}\right]\ls \frac{\eps^{r}}{3}\mathbb{E}\left[\left(\int_{0}^{t}\left\|\u\right\|^2_{H^{2}\left( \O \right)}\d s\right)^{r}\right]
       +C\eps^{\frac{r}{2}},
\end{align}
where $C$ depends on $\delta, \eta, \kappa, r, t$ and $\mathbb{E}\left[e_{\rm reg}(0)^{k_{2}r}\right]$, in which $k_{2}$ is a specific positive constant.
\begin{align}
&  \mathbb{E}\left[\left(-\eps\int_{0}^{t}\int_{\O}\triangle\u\cdot\nabla\triangle\ln\rho\d x\d s\right)^{r}\right]
\ls  \frac{\eps^{r}}{3}\mathbb{E}\left[\left(\int_{0}^{t}\left\|\triangle \u\right\|^2_{L^{2}\left( \O \right)}\d s\right)^{r}\right]+C\eps^{\frac{r}{2}},
\end{align}
where $C$ depends on $\delta, \eta, \kappa, r, t$ and $\mathbb{E}\left[e_{\rm reg}(0)^{k_{3}r}\right]$, in which $k_{3}$ is a specific positive constant.

About $\displaystyle I_{8}=-r_{0}\left(\int_{\O}|\u|^{2}\u\cdot\nabla\ln\rho \right)\d x\d s$, it holds
\begin{align}
&\left|-r_{0}\int_{0}^{t}\int_{\O}|\u|^{2}\u\cdot\nabla\ln\rho \d x\d s \right|
\ls r_{0}\int_{0}^{t}\int_{\O}\rho^{\frac{3}{4}}|\u|^{3} \left|\frac{\nabla\rho}{\rho\rho^{\frac{3}{4}}}\right|\d x\d s  \notag\\
\ls & r_{0}\left\|\rho^{-1-\frac{3}{4}}\right\|_{L_{t}^{\infty}L_{x}^{\infty}} \left\|\rho^{\frac{3}{4}}|\u|^{3}\right\|_{L_{t}^{\frac{4}{3}}L_{x}^{\frac{4}{3}}}\left\|\nabla\rho\right\|_{L_{t}^{\infty}L_{x}^{\infty}} \\
\ls & r_{0}\left\|\rho^{-1}\right\|_{L_{t}^{\infty}L_{x}^{\infty}}^{\frac{7}{4}} \left\| \rho^{\frac{3}{4}} |\u|^{3}\right\|_{L_{t}^{\frac{4}{3}}L_{x}^{\frac{4}{3}}}\left\|\nabla\triangle^{4}\rho\right\|_{L_{t}^{\infty}L_{x}^{2}}. \notag
\end{align}
Therefore, there holds
\begin{equation}
\mathbb{E}\left[\left|-r_{0}\int_{0}^{t}\int_{\O}|\u|^{2}\u\cdot \nabla\ln\rho \d x\d s\right|^{r}\right] \ls C_{4},
\end{equation}
where $C_{4}$ depends on $r,t,r_{0},r_{1},\delta, \eta$ and $\mathbb{E}\left[e_{\rm reg}(0)^{k_{4}r}\right]$, in which $k_{4}$ is a specific positive constant. This requires the terms $\eta \nabla\rho^{-10}$ and $\delta\nabla \triangle^{9}\rho$ to be left after $r_{0} |\u|^{2}\u$ vanishes.

As for $\displaystyle I_{9}=-r_{1}\left(\int_{\O}\rho|\u|^{2}\u\cdot \nabla\ln\rho \d x\right)\d s$, estimating
\begin{align}
&\left|-r_{1}\int_{\O}\rho|\u|^{2}\u\cdot \nabla\ln\rho \d x\right|=\left|-r_{1}\int_{\O}|\u|^{2}\u\cdot \nabla\rho \d x\right|\\
\ls& C\int_{\O}r_{1}\rho|\u|^{2}|\nabla\u|\d x\ls C \int_{\O}r_{1}\rho|\u|^{4}\d x+\frac{1}{8}\int_{\O}\rho|\nabla\u|^{2}\d x,\notag
\end{align}
$\frac{1}{8}\int_{\O}\rho|\nabla\u|^{2}\d x$ is controlled by $\int_{\O}\rho|\mathbb{D}\u|^{2}\d x$, by the energy estimate \eqref{energy estimate eps 2}, 
we have
\begin{equation}
 \mathbb{E}\left[\left|\int_{0}^{t}-r_{1}\int_{\O}\rho|\u|^{2}\u\cdot \nabla\ln\rho \d x\d s\right|^{r}\right]\ls C(\mathbb{E}\left[e_{\rm reg}(0)^{r}\right]+1),
\end{equation}
where $C$ depends on $r,t$ and $\sum\limits_{k=1}^{+\infty} f^{2}_{k}$.

About $\displaystyle I_{10}=-r_{2}\left(\int_{\O}\u\cdot\nabla\ln\rho \d x\right)\d s$, we calculate
\begin{equation}
-r_{2}\int_{\O}\u\cdot\nabla\ln\rho \d x=-r_{2}\int_{\O}\frac{\u\cdot\nabla\rho}{\rho}\d x
=r_{2}\int_{\O}(\ln\rho)_{t}\d x-\eps r_{2}\int_{\O}\frac{\triangle\rho}{\rho}\d x.
\end{equation}
We denote $\ln_{+}\rho=\ln \max\{\rho,1\}$, $\ln_{-}\rho=-\ln \min\{\rho,1\}$, therefore we have
\begin{equation}
 \int_{\O}\ln_{+}\rho\d x\ls \int_{\O}\rho \mathrm{I}_{\{\rho>1\}}\d x
  \ls\int_{\O}\rho^{\gamma}\mathrm{I}_{\{\rho>1\}}\d x\ls \int_{\O}\rho^{\gamma}\d x.
\end{equation}
So we have
\begin{align}
&\quad -r_{2}\int_{0}^{t}\int_{\O}\u\cdot\nabla\ln\rho\d x\d s\\
&\ls r_{2} \int_{\O} \left(\ln_{+}\rho-\ln_{-}\rho\right)(t) \d x + r_{2}\int_{\O}\left(\ln_{-}\rho_{0}-\ln_{+}\rho_{0}\right)\d x +\eps r_{2}\int_{0}^{t}\int_{\O}\frac{\left|\triangle\rho\right|}{\rho} \d x\d s\notag\\
&\ls r_{2}\int_{\O}\ln_{+}\rho(t)\d x - r_{2}\int_{\O}\ln_{-}\rho (t)\d x + r_{2}\int_{\O}\ln_{-}\rho_{0}\d x + \eps r_{2}\int_{0}^{t}\int_{\O}\frac{\left|\triangle\rho\right|}{\rho} \d x\d s\notag\\
&\ls r_{2}\int_{\O}\rho^{\gamma}(t)\d x- r_{2}\int_{\O}\ln_{-}\rho(t)\d x + r_{2}\int_{\O}\ln_{-}\rho_{0}\d x + \eps r_{2}\int_{0}^{t}\int_{\O}\frac{\left|\triangle\rho\right|}{\rho} \d x\d s.\notag
\end{align}
We move the term $\displaystyle -r_{2}\int_{\O}\ln_{-}\rho\d x$ to the left-hand side. Next we take the expectation of its $r$-th power
\begin{align}
&\quad \mathbb{E}\left[\left|r_{2}\int_{\O}\rho^{\gamma}\d x + r_{2}\int_{\O}\ln_{-}\rho\d x+\eps r_{2}\int_{0}^{t}\int_{\O}\frac{\left|\triangle\rho\right|}{\rho} \d x \d s\right|^{r}\right]\notag\\
&\ls r_{2}C_{r}\mathbb{E}\left[\left|\int_{\O}\rho^{\gamma}\d x\right|^{r}\right]+ r_{2}C_{r}\mathbb{E}\left[\left|\int_{\O}\ln_{-}\rho_{0}\d x\right|^{r}\right]
 +C_{r}\mathbb{E}\left[\left|\eps r_{2}\int_{0}^{t}\int_{\O}\frac{\left|\triangle\rho\right|}{\rho} \d x\d s\right|^{r}\right]\\
&\ls C_{1}+C_{r}\mathbb{E}\left[\left|r_{2}\int_{\O}\ln_{-}\rho_{0}\d x\right|^{r}\right]+C_{2}\eps^{\frac{r}{2}} ,\notag
\end{align}
 where $C_{1}$ depends on $r, r_{2}, t$ and $\mathbb{E}\left[e_{\rm reg}(0)^{r}\right]$, and $C_{2}$ depends on $r,r_{2},t,\delta,\eta$ and $\mathbb{E}[e_{\rm reg}(0)^{r}]$.
Finally, with recalling to the definition \eqref{BD entropy with REG}, we obtain the stochastic B-D entropy:
\begin{align}\label{origi-sto-BD entro}
&\mathbb{E}\left[\left(\tilde{e}_{\rm reg}(t)+\int_{0}^{t}\int_{\O}\left(\eps \frac{4a}{\gamma}\left|\nabla \rho_{\rm reg}^{\frac{\gamma}{2}}\right|^{2}+\eps\eta \frac{11}{25} \left|\nabla \rho_{\rm reg}^{-5}\right|^{2}+\eps\delta\left|\triangle^{5}\rho_{\rm reg}\right|^{2}+\eps \kappa \rho_{\rm reg}\left|\nabla^{2}\ln\rho_{\rm reg}\right|^{2}\right.\right.\right.\notag\\
&\quad \quad\quad \quad\quad \quad \left.\left.\left.+\frac{\rho_{\rm reg}}{2}\left|\nabla\u_{\rm reg}-\nabla\u_{\rm reg}^{\top}\right|^{2}+\eps\left|\triangle \u_{\rm reg}\right|^{2}+r_{0}\left|\u_{\rm reg}\right|^{4}+r_{1}\rho_{\rm reg} \left|\u_{\rm reg}\right|^{4}+r_{2}\left|\u_{\rm reg}\right|^{2}\right.\right.\right.\notag\\
&\quad \quad\quad \quad\quad \left.\left.\left.+\eta \frac{11}{25} \left|\nabla \rho_{\rm reg}^{-5}\right|^{2}+\gamma\rho_{\rm reg}^{\gamma}\left|\nabla\ln\rho_{\rm reg}\right|^{2}+\delta\left|\triangle^{5}\rho_{\rm reg}\right|^{2}+\kappa \rho_{\rm reg}\left|\nabla^{2}\ln\rho_{\rm reg}\right|^{2}\right)\d x\d s\right)^{r}\right]\notag\\
&\ls C_{3}\eps^{\frac{r}{2}}+C_{4}+C_{5}\left(\mathbb{E}\left[e_{\rm reg}(0)^{r}\right]+1\right)+C_{r}\mathbb{E}\left[\tilde{e}_{\rm reg}(0)^{r}\right]\\
&\quad +\mathbb{E}\left[\left( \int_{0}^{t}\int_{\O}\rho_{\rm reg}\left|\nabla\ln\rho_{\rm reg}\right|^{2}\d x\d s \right)^{r}\right] +\eps^{r}\mathbb{E}\left[\left( \int_{0}^{t}\left\|\u_{\rm reg}\right\|^2_{H^{2}\left( \O \right)}\d s\right)^{r}\right]\notag
\end{align}
with recovered notations, where $C_{3}$ depends on $\delta,\eta,\kappa,r,t$, $\mathbb{E}\left[e_{\rm reg}(0)^{r}\right]$, $\mathbb{E}\left[e_{\rm reg}(0)^{\frac{6r}{5}}\right]$,
 $\mathbb{E}\left[e_{\rm reg}(0)^{k_{1}r}\right]$, $\mathbb{E}\left[e_{\rm reg}(0)^{k_{2}r}\right]$, and $\mathbb{E}\left[e_{\rm reg}(0)^{k_{3}r}\right])$, $C_{4}$ depends on $r,t,r_{0},r_{1},\eta, \delta$ and $\mathbb{E}\left[e_{\rm reg}(0)^{k_{4}r}\right]$, $C_{5}$ depends on $r,t$ and $\sum\limits_{k=1}^{+\infty} f^{2}_{k}$.

This proves that
\begin{align}
&\mathbb{E}\left[\left( \tilde{e}_{\rm reg}(t)+\int_{0}^{t}\int_{\O}\left(\eps \frac{4a}{\gamma}\left|\nabla \rho_{\rm reg}^{\frac{\gamma}{2}}\right|^{2}+\eps\eta \frac{11}{25} \left|\nabla \rho_{\rm reg}^{-5}\right|^{2}+\eps\delta\left|\triangle^{5}\rho_{\rm reg}\right|^{2}+\eps \kappa \rho_{\rm reg}\left|\nabla^{2}\ln\rho_{\rm reg}\right|^{2}\right.\right.\right.\notag\\
&\quad \quad\quad \quad\qquad \qquad \left.\left.\left.+\frac{\rho_{\rm reg}}{2}\left|\nabla\u_{\rm reg}-\nabla\u_{\rm reg}^{\top}\right|^{2}+r_{0}\left|\u_{\rm reg}\right|^{4}+r_{1} \rho_{\rm reg} \left|\u_{\rm reg}\right|^{4}+r_{2}\left|\u_{\rm reg}\right|^{2}\right.\right.\right.\\
&\quad \quad\quad \quad\quad \left.\left.\left.+\eta \frac{11}{25} \left|\nabla \rho_{\rm reg}^{-5}\right|^{2}+\gamma\rho_{\rm reg}^{\gamma}\left|\nabla\ln\rho_{\rm reg}\right|^{2}+\delta\left|\triangle^{5}\rho_{\rm reg}\right|^{2}+\kappa \rho_{\rm reg}\left|\nabla^{2}\ln\rho_{\rm reg}\right|^{2}\right)\d x\d s\right)^{r}\right]\notag\\
&\ls C_{3}\eps^{\frac{r}{2}}+C_{4}+C_{5}\left(\mathbb{E}\left[e_{\rm reg}(0)^{r}\right]+1\right)+\mathbb{E}\left[\left( \int_{0}^{t}\int_{\O}\rho_{\rm reg}\left|\nabla\ln\rho_{\rm reg}\right|^{2}\d x\d s \right)^{r}\right]+C_{r}\mathbb{E}\left[\tilde{e}_{\rm reg}(0)^{r}\right].\notag
\end{align}
 Therefore, by the definition of $\tilde{e}_{\rm reg}(t)$, we have
\begin{align}
 \mathbb{E}\left[\tilde{e}_{\rm reg}(t)^{r}\right] \ls & C_{3}\eps^{\frac{r}{2}}+C_{4}+C_{5}\left(\mathbb{E}\left[e_{\rm reg}(0)^{r}\right]+1\right)+C_{r}\mathbb{E}\left[\tilde{e}_{\rm reg}(0)^{r}\right]+ \int_{0}^{t} \mathbb{E}\left[\tilde{e}_{\rm reg}(s)^{r}\right]\d s.
\end{align}
\end{proof}
\subsection{Passing to the limit $\eps\ra 0$}
\begin{flushleft}
\textbf{Step 1:} Choose the path space of the density and the velocity.
\end{flushleft}
From (\ref{BD estimate2}), if we assume that $\mathbb{E}\left[e_{\rm reg}(0)^{r}\right]$ and $\mathbb{E}\left[\tilde{e}_{\rm reg}(0)^{r}\right]$ are bounded,
recalling lemma \ref{Jungle inequality} for $\rho_{\rm reg}(x)$, it holds
\begin{equation}
\mathbb{E}\left[\kappa^{r}\left\|\rho_{\rm reg}^{\frac{1}{2}}\right\|_{L_{t}^{2}H_{x}^{2}}^{2r}\right] \ls C, \quad \mathbb{E}\left[\kappa^{r}\left\|\nabla\left(\rho_{\rm reg}^{\frac{1}{4}}\right)\right\|_{L_t^4 L_x^4}^{4r}\right] \ls C,\quad r>4,
\end{equation}
uniformly in $\eps, \delta, \eta, \kappa$. The following estimates hold independent of $\eps$:
\begin{align}\label{the no eps upper bound of the corresponding integral of terms}
 &\mathbb{E}\left[\sup\limits_{t\in[0,T]}\eta^{r}\left\|\rho_{\rm reg}^{-1}\right\|^{10r}_{L_{x}^{10}}\right]\ls C, \quad
 \mathbb{E}\left[\sup\limits_{t\in[0,T]}\left(\frac{1}{2}\right)^{r}\left\|\nabla\sqrt{\rho_{\rm reg}}\right\|^{2r}_{L_{x}^{2}}\right]\ls C, \notag\\
 &\mathbb{E}\left[\sup\limits_{t\in[0,T]}\delta^{r}\left\|\nabla\triangle^{4}\rho_{\rm reg}\right\|^{2r}_{L_{x}^{2}}\right]\ls C,\quad \mathbb{E}\left[\eta^{r}\left\|\nabla \rho_{\rm reg}^{-5}\right\|^{2r}_{L_{t}^{2}L_{x}^{2}}\right]\ls C,\\
 &\mathbb{E}\left[\gamma^{r} \left\|\nabla \rho_{\rm reg}^{\frac{\gamma}{2}}\right\|^{2r}_{L_{t}^{2}L_{x}^{2}}\right]\ls C,\quad
 \mathbb{E}\left[\delta^{r}\left\|\triangle^{5} \rho_{\rm reg}\right\|^{2r}_{L_{t}^{2}L_{x}^{2}}\right]\ls C, \notag
\end{align}
where $r>4$ and $C$ depends on $\delta, \eta, \kappa, r, T, \mathbb{E}\left[e_{\rm reg}(0)^{r}\right],\mathbb{E}\left[e_{\rm reg}(0)^{\frac{6r}{5}}\right]$ and $\mathbb{E}\left[e_{\rm reg}(0)^{k_{{\rm I}}r}\right]$, $i=1,2,3$, $r>4$.
Similarly as in the last section, we give the following estimates from the mass equation without proof here:
\begin{equation}
 \mathbb{E}\left[\left\|\left(\rho_{\rm reg}\right)_{t}\right\|_{L_t^2 L_x^2}^{r}\right]\ls  C,\quad \mathbb{E}\left[\left\|\rho_{\rm reg}\u_{\rm reg}\right\|_{L_t^2 L_x^2}^{r}\right] \ls C,
\end{equation}
\begin{equation}\label{independent on delta after eps ra 0}
\begin{aligned}
 \mathbb{E}\left[\left\|\left(\rho_{\rm reg}\right)_{t}\right\|_{L_t^2 L_x^{\frac{3}{2}}}^{r}\right]\ls  C,\quad \mathbb{E}\left[\left\|\left(\rho_{\rm reg}^{\frac{1}{2}}\right)_{t}\right\|_{L_t^2 L_x^2}^{r}\right] \ls C, \quad \mathbb{E}\left[\left\|\rho_{\rm reg}^{\frac{1}{2}}\right\|_{L_t^2 H_x^{2}}^{r}\right] \ls C,
\end{aligned}
\end{equation}
where $C$ depends on $\delta,\eta,\kappa,r,t,\mathbb{E}\left[e_{\rm reg}(0)^{r}\right],\mathbb{E}\left[e_{\rm reg}(0)^{\frac{6r}{5}}\right]$ and $\mathbb{E}\left[e_{\rm reg}(0)^{k_{{\rm I}}r}\right]$, $r>4$. Note that \eqref{independent on delta after eps ra 0} will be independent on $\delta$ and $\eta$ after $\eps\rightarrow 0$.

 Next, we choose the path space 
\begin{equation}
\mathcal{X}_{3}=\mathcal{X}_{\rho_{0}} \times \mathcal{X}_{\mathbf{q}_{0}}^{3} \times \mathcal{X}_{\frac{\mathbf{q}_{0}}{\sqrt{\rho_{0}}}}^{3} \times \mathcal{X}_{\rho}^{3} \times \mathcal{X}_{\u}^{3} \times \mathcal{X}_{\rho\u}^{3} \times \mathcal{X}_{W}^{3}, \\
\end{equation}
where $\mathcal{X}_{\rho_{0}}^{3}=L^{\gamma}\left(\O\right)\cap L^{1}\left(\O\right) \cap L^{-10}\left(\O\right)\cap H^{9}\left(\O\right)$, 
$\mathcal{X}_{\mathbf{q}_{0}}^{3}=L^{1}\left(\O\right))$,
$ \mathcal{X}_{\frac{\mathbf{q}_{0}}{\sqrt{\rho_{0}}}}^{3}=L^{2}\left(\O\right)$,
$\mathcal{X}_{\rho}^{3}=L^{2}\left([0,T];H^{10}\left( \O \right)\right)\cap L^{2}\left([0,T];W^{1,3}\left( \O \right)\right) \cap L^{\frac{5}{3}\gamma}\left([0,T]\times\O\right)\cap C\left([0,T];W^{1,3}\left( \O \right)\right)$,
$\mathcal{X}_{\u}^{3}= L^{2}\left([0,T]\times\O\right)\cap L^{4}\left([0,T]\times\O\right)$,
$ \mathcal{X}_{\rho\u}^{3}= L^{2}\left([0,T]; W^{1,\frac{3}{2}}\left( \O \right)\right)\cap C\left([0,T];L^{\frac{3}{2}}\left( \O \right)\right)$,
$\mathcal{X}_{W}^{3}= C\left([0,T];\mathcal{H}\right)$.\\

\begin{flushleft}
\textbf{Step 2:} The tightness of the laws and the limit $\eps\ra 0$.
\end{flushleft}

Similarly as in last section, we have the following proposition and skip a proof here.
\begin{proposition}
$\left\{\mathcal{L}\left[\rho_{0,{\rm reg}}, \mathbf{q}_{0,{\rm reg}}, \frac{\mathbf{q}_{0,{\rm reg}}}{\sqrt{\rho_{0,{\rm reg}}}}, \rho_{\rm reg}, \u_{\rm reg},\rho_{\rm reg}\u_{\rm reg}, W_{\rm reg}\right]\right\}$ is tight on $\mathcal{X}_{3}$.
\end{proposition}
\begin{proposition}
There exist two families of $\mathcal{X}_{3}$-valued Borel measurable random variables
 $$
 \left\{\bar{\rho}_{0,{\rm reg}},\bar{\mathbf{q}}_{0,{\rm reg}}, \frac{\bar{\mathbf{q}}_{0,{\rm reg}}}{\sqrt{\bar{\rho}_{0,{\rm reg}}}}, \bar{\rho}_{\rm reg}, \bar{\u}_{\rm reg}, \bar{\rho}_{\rm reg}\bar{\u}_{\rm reg}, \bar{W}_{\rm reg} \right\},
 $$
  and $\left\{\rho_{0, {\rm I}}, \mathbf{q}_{0, {\rm I}},\frac{\mathbf{q}_{0, {\rm I}}}{\sqrt{\rho_{0, {\rm I}}}}, \rho_{{\rm I}}, \u_{{\rm I}},\rho_{{\rm I}}\u_{{\rm I}}, W\right\}$, defined on a new complete probability space, we still denote it as $\left(\bar{\Omega},\bar{\mathcal{F}},\bar{\mathbb{P}}\right)$, for convenience we still use the same notation in the following sections, such that (up to a subsequence):
\begin{enumerate}
  \item
  $\mathcal{L}\left[\bar{\rho}_{0,{\rm reg}}, \bar{\mathbf{q}}_{0,{\rm reg}}, \frac{\bar{\mathbf{q}}_{0,{\rm reg}}}{\sqrt{\bar{\rho}_{0,{\rm reg}}}}, \bar{\rho}_{\rm reg}, \bar{\u}_{\rm reg}, \bar{\rho}_{\rm reg}\bar{\u}_{\rm reg}, \bar{W}_{\rm reg} \right],$
   coincides with\\ $\mathcal{L}\left[\rho_{0,{\rm reg}}, \mathbf{q}_{0,{\rm reg}}, \frac{\mathbf{q}_{0,{\rm reg}}}{\sqrt{\rho_{0,{\rm reg}}}}, \rho_{\rm reg}, \u_{\rm reg},\rho_{\rm reg}\u_{\rm reg}, W_{\rm reg}\right] $ on $\mathcal{X}_{3}$;
  \item $\mathcal{L}\left[\rho_{0, {\rm I}}, \mathbf{q}_{0, {\rm I}},\frac{\mathbf{q}_{0, {\rm I}}}{\sqrt{\rho_{0, {\rm I}}}}, \rho_{{\rm I}}, \u_{{\rm I}},\rho_{{\rm I}}\u_{{\rm I}}, W_{{\rm I}}\right]$ on $\mathcal{X}_{3}$ is a Radon measure;
  \item The family of random variables $\left\{\bar{\rho}_{0,{\rm reg}}, \bar{\mathbf{q}}_{0,{\rm reg}}, \frac{\bar{\mathbf{q}}_{0,{\rm reg}}}{\sqrt{\bar{\rho}_{0,{\rm reg}}}}, \bar{\rho}_{\rm reg}, \bar{\u}_{\rm reg}, \bar{\rho}_{\rm reg}\bar{\u}_{\rm reg}, \bar{W}_{\rm reg} \right\}$
   converges to
    $\left\{\rho_{0, {\rm I}}, \mathbf{q}_{0, {\rm I}},\frac{\mathbf{q}_{0, {\rm I}}}{\sqrt{\rho_{0, {\rm I}}}}, \rho_{{\rm I}}, \u_{{\rm I}},\rho_{{\rm I}}\u_{{\rm I}}, W_{{\rm I}}\right\}$ in the topology of $\mathcal{X}_{3}$, $\bar{\mathbb{P}}$ {\rm a.s.} as $\eps\ra 0$.
    \end{enumerate}
\end{proposition}
\par
\begin{flushleft}
\textbf{Step 3:} The system after taking the limit.
\end{flushleft}
\par
Similarly as the estimates in Lemma \ref{strong conver of rho}, by Aubin-Lion's lemma, we have the strong convergence of $\rho_{\rm reg}$:
\begin{equation}
\bar{\rho}_{\rm reg}\rightarrow \rho_{{\rm I}} \quad \text{in} \quad L^{r}\left(\Omega; L^{2}\left([0,T];W^{1,3}\left( \O \right)\right)\right), \quad r>4 ;
\end{equation}
 uniformly in $\eps$, $\eta$ and $\delta$.
And similarly there holds the strong convergence of $\bar{\rho}_{\rm reg}^{\frac{1}{2}}$:
\begin{equation}
\bar{\rho}_{\rm reg}^{\frac{1}{2}}\rightarrow \rho_{{\rm I}}^{\frac{1}{2}} \quad \text{in} \quad L^{r}\left(\Omega;L^{2}\left([0,T];W^{1,6}\left(\O\right)\right)\right), \quad r>4,
\end{equation}
 uniformly in $\eps$, $\eta$ and $\delta$. $  W^{1,6}\left( \O \right)\hookrightarrow C^{0,\frac{1}{2}}\left(\O\right)$, $ \bar{\rho}_{\rm reg}^{\frac{1}{2}}\bar{\u}_{\rm reg} \in L^{r}\left(\Omega;L^{\infty}\left([0,T];L^{2}\left(\O\right)\right)\right)$. So $\{\bar{\rho}_{\rm reg}\bar{\u}_{\rm reg}\}$ converges weakly to $\rho_{{\rm I}}\u_{{\rm I}}$ in $  L^{r}\left(\Omega; L^{2}\left([0,T]\times\O\right)\right)$, $r>4$.

For this $  \eps\rightarrow 0$ layer, by the strong convergence of $\bar{\rho}_{\rm reg}$ and $\bar{\rho}_{\rm reg}\bar{\u}_{\rm reg}$,
 we have the almost everywhere convergence of $ \bar{\rho}_{\rm reg}$ and $  \bar{\rho}_{\rm reg}\bar{\u}_{\rm reg}$, according to Vitali's convergence theorem,
 $\rho_{{\rm I}} $ satisfies the new equation
\begin{equation}
\left(\rho_{{\rm I}}\right)_{t}+\operatorname {div}(\rho_{{\rm I}} \u_{{\rm I}})=0,
\end{equation}
in distribution $\bar{\mathbb{P}}$ a.s.
\begin{proposition}
 $\left(\rho_{{\rm I}}, \u_{{\rm I}}\right)$ satisfies the new system involving $\eta, \delta, \kappa$:
\begin{equation}\label{eps layer system}
\left\{\begin{array}{l}
\left(\rho_{{\rm I}}\right)_{t}+\operatorname {div}(\rho_{{\rm I}} \u_{{\rm I}})=0,\\
\d \left(\rho_{{\rm I}} \u_{{\rm I}}\right)+\left(\operatorname{div}(\rho_{{\rm I}} \u_{{\rm I}} \otimes \u_{{\rm I}})+ \nabla \left(a\rho_{{\rm I}}^{\gamma}\right)-\operatorname{div}(\rho_{{\rm I}}
\mathbb{D} \u_{{\rm I}})\right)\d t\\
=\left(-r_{0}|\u_{{\rm I}}|^{2}\u_{{\rm I}}-r_{1} \rho_{{\rm I}}|\u_{{\rm I}}|^{2} \u_{{\rm I}}-r_{2}\u_{{\rm I}}\right)\d t+\left(\delta \rho_{{\rm I}}\nabla \triangle^{9}\rho_{{\rm I}}+\frac{11}{10}\eta \nabla \rho_{{\rm I}}^{-10}\right)\d t\\
\quad+\kappa \rho_{{\rm I}}\left(\nabla\left(\frac{\triangle \sqrt{\rho_{{\rm I}}}}{\sqrt{\rho_{{\rm I}}}}\right)\right)\d t+\rho_{{\rm I}} \mathbb{F}(\rho_{{\rm I}},\u_{{\rm I}})\d W_{{\rm I}}
\end{array}\right.
\end{equation}
in distribution $\bar{\mathbb{P}}$ {\rm a.s.}
\end{proposition}
\begin{proof}
For all $\varphi(t)\in C_{c}^{\infty}\left([0,T)\right)$, and all $\psi(x)\in C^{\infty}\left(\O\right)$, we have
\begin{align}
\int_{0}^{T}\int_{\O}\eps \triangle\bar{\rho}_{\rm reg} \varphi(t)\psi(x) \d x\d t
\ls \eps \delta^{-\frac{1}{2}}\left\|\sqrt{\delta}\bar{\rho}_{\rm reg}\right\|_{L_t^{2}H_x^{10}}^{r}\left\|\varphi(t)\right\|_{L_t^{2}}\left\|\psi(x)\right\|_{H_x^{9}}\ra 0,  
\end{align}
as $\eps\ra 0$, $\bar{\mathbb{P}}$ a.s.
Correspondingly, we get
\begin{align}
&\quad \int_{0}^{T}\int_{\O}\eps \nabla\bar{\rho}_{\rm reg}\nabla\bar{\u}_{\rm reg}  \varphi(t)\psi(x) \d x\d t \\
&\ls \eps 2\left\|\nabla\sqrt{\bar{\rho}_{\rm reg}}\right\|_{L_s^{\infty}L_x^{2}}\left\|\sqrt{\bar{\rho}_{\rm reg}}\nabla\bar{\u}_{\rm reg}\right\|_{L_t^{2}L_x^{2}}\left\| \varphi(t)\right\|_{L_t^{2}}\left\|\psi(x) \right\|_{H_x^{9}}\ra 0\notag
\end{align}
 as $\eps\ra 0$, $\bar{\mathbb{P}}$ a.s., and we have
\begin{equation}
\int_{0}^{T}\int_{\O}\eps \triangle^{2}\bar{\u}_{\rm reg} \varphi(t)\psi(x) \d x\d t
\ls \eps r_{0}^{-\frac{1}{2}}\left\|r_{2}^{\frac{1}{2}}\bar{\u}_{\rm reg}\right\|_{L_t^{2}L_x^{2}}\left\| \varphi(t)\right\|_{L_t^{2}}\left\|\psi(x) \right\|_{H_x^{9}}\ra 0,
\end{equation}
as $\eps\ra 0$. The weak convergence of $\nabla\bar{\rho}_{\rm reg}^{-5}$, $\triangle^{5}\bar{\rho}_{\rm reg}$, $  \nabla\bar{\rho}_{\rm reg}^{\frac{\gamma}{2}}$,
 $\bar{\rho}_{\rm reg}\nabla^{2}\ln\bar{\rho}_{\rm reg}$ still hold in the corresponding spaces. Next, we check the convergence of the term
  associated with stochastic forces in the momentum equation. Since $ \bar{\rho}_{\rm reg}\bar{\u}_{\rm reg}\ra \rho_{{\rm I}} \u_{{\rm I}} $
  strongly in $ L^{2}\left([0,T]\times \O\right)$ uniformly in $ \eps, \eta, \delta$ and $\bar{\mathbb{P}}$ a.s., $  \bar{\rho}_{\rm reg}\ra \rho_{{\rm I}} $ strongly in $  L^{2}\left([0,T];W^{1,3}\left( \O \right)\right)$
  uniformly in $  \eps, \eta, \delta$ and $\bar{\mathbb{P}}$ a.s., by \eqref{property of rho F}, so we have the almost everywhere convergence
  of $  \bar{\rho}_{\rm reg}\mathbf{F}_{k}\left(\bar{\rho}_{\rm reg},\bar{\u}_{\rm reg}\right)$. Hence we still have the weak convergence of the stochastic term.
More precisely, we obtain
\begin{equation}
\int_{0}^{T}\int_{\O}\bar{\rho}_{\rm reg}\mathbf{F}_{k}\left(\bar{\rho}_{\rm reg},\bar{\u}_{\rm reg}\right)\varphi(t)\psi(x)  \d \bar{W}_{\rm reg}\d x\ra \int_{0}^{T}\int_{\O}\rho_{{\rm I}} \mathbf{F}_{k}\left(\rho_{{\rm I}} ,\u_{{\rm I}} \right)\varphi(t)\psi(x) \d W_{{\rm I}}\d x
\end{equation}
 $\bar{\mathbb{P}}$ a.s. \hfill$\square$
 \end{proof}

\subsection{Energy estimates for the approximated solutions after $\eps\ra 0$}
$\u_{\rm reg}$ satisfies the energy estimates (\ref{BD estimate1}) and (\ref{BD estimate2}). We also need to take the limit $\eps \rightarrow 0$ to
get the new energy estimates of $\u_{{\rm I}}$.
 By the lower semi-continuity of the integrals of convex
 functions, we pass to the limit $\eps\rightarrow 0$, then \eqref{the no eps upper bound of the corresponding integral of terms} still holds. Let $\eps \ra 0$ in \eqref{origi-sto-BD entro}, then we gain
\begin{align}
 \mathbb{E}\left[\tilde{e}_{{\rm I}}(t)^{r}\right]\ls & C_{4}+C_{5}\left(\mathbb{E}\left[e_{{\rm I}}(0)^{r}\right]+1\right)+C_{r} \mathbb{E}\left[\tilde{e}_{{\rm I}}(0)^{r}\right]+
 \int_{0}^{t} \mathbb{E}\left[\tilde{e}_{{\rm I}}(s)^{r}\right]\d s
 \notag \\
 \ls & C + C_{r}\mathbb{E}\left[\tilde{e}_{{\rm I}}(0)^{r}\right] + \int_{0}^{t} \mathbb{E}\left[\tilde{e}_{{\rm I}}(s)^{r}\right]\d s.
\end{align}
$C_{4}+C_{5}\left(\mathbb{E}\left[e_{{\rm I}}(0)^{r}\right]+1\right)+C_{r}\mathbb{E}\left[\tilde{e}_{{\rm I}}(0)^{r}\right]$ is nondecreasing with
 respect to $t$, by Gr\"onwall's inequality, we obtain
\begin{equation}\label{BD estimate1}
\begin{aligned}
 \mathbb{E}\left[\tilde{e}_{{\rm I}}(t)^{r}\right] &\ls \left( C + C_{r}\mathbb{E}\left[\tilde{e}_{{\rm I}}(0)^{r}\right]\right)e^{t}\ls C\left(1+\mathbb{E}\left[\tilde{e}_{{\rm I}}(0)^{r}\right] \right),
\end{aligned}
\end{equation}
where $C$ depends on $\delta,\eta,r,t,r_{0},r_{1},\mathbb{E}\left[e_{{\rm I}}(0)^{r}\right],\mathbb{E}\left[e_{{\rm I}}(0)^{\frac{6r}{5}}\right]$ and $\mathbb{E}\left[e_{{\rm I}}(0)^{k_{{\rm I}}r}\right]$, $i=1,2,3$,$4$. In conclusion, we write
\begin{align}\label{BD entropy conclusion after vanish eps}
\mathbb{E}\left[\sup\limits_{t\in[0,T]}\tilde{e}_{{\rm I}}(t)^{r}\right] \ls C\left(1+\mathbb{E}\left[\tilde{e}_{{\rm I}}(0)^{r}\right] \right), r>4,
\end{align}
where $C$ depends on $\delta,\eta,r,T,r_{0},r_{1},\mathbb{E}\left[e_{{\rm I}}(0)^{r}\right],\mathbb{E}\left[e_{{\rm I}}(0)^{\frac{6r}{5}}\right]$ and $\mathbb{E}\left[e_{{\rm I}}(0)^{k_{{\rm I}}r}\right]$, $i=1,2,3,4$.
Thereupon, we deduce
\begin{align}\label{no eps BD estimate2}
\mathbb{E}\left[\left(\int_{0}^{t}\int_{\O}\frac{\rho_{{\rm I}}}{2}\left|\nabla\u_{{\rm I}}-\nabla\u_{{\rm I}}^{\top}\right|^{2}\d x\d s\right)^{r}\right]\ls C,
\quad \mathbb{E}\left[\left(r_{0}\int_{0}^{t}\int_{\O}\left|\u_{{\rm I}}\right|^{4}\d x\d s\right)^{r}\right]\ls C,\\
\mathbb{E}\left[\left(r_{1}\int_{0}^{t}\int_{\O}\rho_{{\rm I}} \left|\u_{{\rm I}}\right|^{4}\d x\d s\right)^{r}\right]\ls C,
\mathbb{E}\left[\left(r_{2}\int_{0}^{t}\int_{\O}\left|\u_{{\rm I}}\right|^{2}\d x\d s\right)^{r}\right]\ls C,\\
\quad \mathbb{E}\left[\left(\int_{0}^{t}\int_{\O}\eta \frac{11}{25} \left|\nabla \rho_{{\rm I}}^{-5}\right|^{2}\d x\d s\right)^{r}\right]\ls C,
\quad \mathbb{E}\left[\left(\int_{0}^{t}\int_{\O}\gamma\rho_{{\rm I}}^{\gamma}\left|\nabla\ln\rho_{{\rm I}}\right|^{2}\d x\d s\right)^{r}\right]\ls C,\\
\mathbb{E}\left[\left(\int_{0}^{t}\int_{\O}\delta\left|\triangle^{5}\rho_{{\rm I}}\right|^{2}\d x\d s\right)^{r}\right]\ls C,
\quad \mathbb{E}\left[\left(\int_{0}^{t}\int_{\O}\kappa \rho_{{\rm I}}\left|\nabla^{2}\ln\rho_{{\rm I}}\right|^{2}\d x\d s\right)^{r}\right]\ls C,
\end{align}
where $C$ depends on $r,T,\mathbb{E}\left[e_{{\rm I}}(0)^{r}\right]$ and $\mathbb{E}\left[\tilde{e}_{{\rm I}}(0)^{r}\right]$, $r>4$.

\section{Global existence of martingale solutions}

In this section, when we take $\kappa\ra 0$, the high regularity of $\rho$ will be lost, which will cause the lack of regularity of $\rho\u$. Our target turns to derive the Mellet-Vasseur \cite{Mellet-Vasseur2007} type inequality in the stochastic version, which guarantees the strong convergence of $\rho^{\frac{1}{2}}\u$.
\subsection{Approximation of Mellet-Vasseur type inequality}

As in \cite{Vasseur-Yu2016}, we define a $C^{\infty}$ cut-off functions $\phi_{K}$:
\begin{equation}\label{cut off functions of K}
\phi_{K}(\rho)=\left\{\begin{array}{l}
1, \text{ for any }\rho<K,\\
0, \text{ for any }\rho>2K,\\
\end{array}\right.
\end{equation}
where $K>0$ is any real number, $\left|\phi_{K}^{\prime}(\rho)\right|\leqslant \frac{4}{K}$. For any $\rho>0$, there exists $C>0$ such that
$$
\left|\phi_{K}^{\prime}(\rho) \sqrt{\rho}\right|+\left|\frac{\phi_{K}(\rho)}{\sqrt{\rho}}\right| \leqslant C,
$$
where $C$ depends on $K$.
For any $\mathbf{u}\in \mathbb{R}^{3}$, we set $\mathbf{v}=\phi_{K}(\rho)\mathbf{u}$. Define $\varphi_{n}(\u)=\tilde{\varphi}_{n}\left(|\u|^{2}\right)\in C^{1}\left(\mathds{R}^{3}\right)$,
where $\tilde{\varphi}_{n}$ is given on $\mathds{R}^{+}$ by
\begin{equation}
\tilde{\varphi}_{n}(y)=\left\{\begin{array}{ll}
(1+y) \ln (1+y), &  0 \leqslant y<n, \\
2(1+\ln (1+n)) y-(1+y) \ln (1+y)+2(\ln (1+n)-n), & n \leqslant y \leqslant C_{n}, \\
e(1+n)^{2}-2 n-2, & y \geq C_{n},
\end{array}\right.
\end{equation}
with  $\tilde{\varphi}_{n}(0)=0$, and $C_{n}=e(1+n)^{2}-1$. So we have
\begin{equation}\label{cut-off functions of u prime}
\tilde{\varphi}_{n}'(y)=\left\{\begin{array}{ll}
1+\ln (1+y), & \text { if } 0 \leqslant y \leqslant n, \\
1+2\ln(1+n)-\ln (1+y)\geq 0, & \text { if } n<y\leqslant C_{n}, \\
0, & \text { if } y > C_{n}.
\end{array}\right.
\end{equation}
 and $\tilde{\varphi}_{n}'(y) \leqslant 1+\ln (1+n)$ for $n<y \leqslant C_{n}$. And we have
\begin{equation}\label{cut off functions of u double prime}
\tilde{\varphi}_{n}''(y)=\left\{\begin{array}{ll}
\frac{1}{1+y}, & \text { if } 0 \leqslant y \leqslant n, \\
-\frac{1}{1+y}, & \text { if } n<y\leqslant C_{n}, \\
0, & \text { if } y > C_{n}.
\end{array}\right.
\end{equation}
$\tilde{\varphi}_{n}(y)$ is a nondecreasing function with respect to $y$ for any fixed $n$, and it is a nondecreasing function with respect to $n$ for any fixed $y$. Further, the convergence holds:
\begin{equation}
\tilde{\varphi}_{n}(y) \rightarrow(1+y) \ln (1+y) \text { almost everywhere as } n \rightarrow \infty.
\end{equation}
We denote $\nabla_{\u}$ as taking derivative with respect to $\u$, we have
\begin{align}
\nabla_{\u}\varphi_{n}(\mathbf{u})=& 2 \tilde{\varphi}_{n}^{\prime}\left(|\mathbf{u}|^{2}\right)\mathbf{u},\\
\nabla_{\u}^{2}\varphi_{n}(\mathbf{u})=& 2\left(2 \tilde{\varphi}_{n}^{\prime \prime}\left(|\mathbf{u}|^{2}\right) \mathbf{u} \otimes \mathbf{u}+\mathbb{I}_{3} \tilde{\varphi}_{n}^{\prime}\left(|\mathbf{u}|^{2}\right)\right).
\end{align}

 In last section, we know that $\rho_{{\rm I}}$, $\rho_{{\rm I}}\u_{{\rm I}}$ satisfies \eqref{eps layer system} 
weakly $\bar{\mathbb{P}}$ a.s.
We approximate \eqref{eps layer system} by
\begin{equation}\label{system for v}
\d(\rho_{{\rm I}} \mathbf{v}_{{\rm I}})+\left(\operatorname{div}\left(\rho_{{\rm I}} \u_{{\rm I}} \otimes \mathbf{v}_{{\rm I}}\right)-\operatorname{div} \mathbb{S}_{{\rm I}}+ \mathbf{R}_{{\rm I}}\right)\d t=\phi_{K}(\rho_{{\rm I}})\rho_{{\rm I}} \mathbb{F}(\rho_{{\rm I}},\u_{{\rm I}})\d W_{{\rm I}},
\end{equation}
where
\begin{align}
 \mathbb{S}_{{\rm I}}=& \rho_{{\rm I}}\phi_{K}(\rho_{{\rm I}})\left(\mathbb{D} \u_{{\rm I}}+\kappa \frac{\Delta \sqrt{\rho_{{\rm I}}}}{\sqrt{\rho_{{\rm I}}}} \mathbb{I}_{3}\right),  \notag\\
 \mathbf{R}_{{\rm I}}=& \rho_{{\rm I}}^{2} \u_{{\rm I}} \phi_{K}^{\prime}(\rho_{{\rm I}}) \operatorname{div} \u_{{\rm I}}+2 \rho_{{\rm I}}^{\frac{\gamma}{2}} \nabla \rho_{{\rm I}}^{\frac{\gamma}{2}} \phi_{K}(\rho_{{\rm I}})+\rho_{{\rm I}} \nabla \phi_{K}(\rho_{{\rm I}}) \mathbb{D} \u_{{\rm I}}\notag\\
&+r_{0}|\u_{{\rm I}}|^{2} \u_{{\rm I}} \phi_{K}(\rho_{{\rm I}})+r_{1} \rho_{{\rm I}}|\u_{{\rm I}}|^{2} \u_{{\rm I}} \phi_{K}(\rho_{{\rm I}})+r_{2} \u_{{\rm I}} \phi_{K}\left(\rho_{{\rm I}}\right) \\
&-\eta \frac{11}{10}\nabla\rho_{{\rm I}}^{-10}\phi_{K}\left(\rho_{{\rm I}}\right) - \delta \rho_{{\rm I}}\nabla \triangle^{9}\rho_{{\rm I}}\phi_{K}(\rho_{{\rm I}}) + \kappa \sqrt{\rho_{{\rm I}}} \nabla \phi_{K}\left(\rho_{{\rm I}}\right)\triangle \sqrt{\rho_{{\rm I}}}\notag\\
&+2 \kappa \phi_{K}\left(\rho_{{\rm I}}\right)\nabla \sqrt{\rho_{{\rm I}}} \triangle \sqrt{\rho_{{\rm I}}}.\notag
\end{align}

Let
$
\hbar(\rho_{{\rm I}},\rho_{{\rm I}}\mathbf{v}_{{\rm I}})=\rho_{{\rm I}}\left(1+\left|\mathbf{v}_{{\rm I}}\right|^{2}\right)\ln \left(1+ \left|\mathbf{v}_{{\rm I}}\right|^{2}\right)=\left(\rho_{{\rm I}}+\frac{\left|\rho_{{\rm I}}\mathbf{v}_{{\rm I}}\right|^{2}}{\rho_{{\rm I}}}\right)\ln \left(1+\frac{\left|\rho_{{\rm I}}\mathbf{v}_{{\rm I}}\right|^{2}}{\rho_{{\rm I}}^{2}}\right).
$
Since
\begin{equation}
\langle \d \left(\rho_{{\rm I}}\mathbf{v}_{{\rm I}}\right), \d \left(\rho_{{\rm I}}\mathbf{v}_{{\rm I}}\right)\rangle = \rho_{{\rm I}}^{2}\phi_{K}^{2}(\rho_{{\rm I}})\mathbb{F}(\rho_{{\rm I}},\u_{{\rm I}})\otimes \mathbb{F}(\rho_{{\rm I}},\u_{{\rm I}})\d t,
\end{equation}
by It\^o's formula, we have
\begin{align}\label{it in Page 36}
\d \hbar\left(\rho_{{\rm I}},\rho_{{\rm I}}\mathbf{v}_{{\rm I}}\right)
=&\hbar_{\rho_{{\rm I}}} \d \rho_{{\rm I}} +\nabla_{(\rho_{{\rm I}}\mathbf{v}_{{\rm I}})}\hbar \d \left(\rho_{{\rm I}} \mathbf{v}_{{\rm I}}\right)+\frac{\nabla_{(\rho_{{\rm I}}\mathbf{v}_{{\rm I}})}^{2}\hbar}{2} \langle \d \left(\rho_{{\rm I}}\mathbf{v}_{{\rm I}}\right), \d \left(\rho_{{\rm I}}\mathbf{v}_{{\rm I}}\right)\rangle  \notag\\
                                   = & \left(\left(1-|\mathbf{v}_{{\rm I}}|^{2}\right)\ln \left(1+|\mathbf{v}_{{\rm I}}|^{2}\right)-2|\mathbf{v}_{{\rm I}}|^{2}\right)\d\rho_{{\rm I}}\notag\\
                                     & + \left(2\mathbf{v}_{{\rm I}}\ln \left(1+|\mathbf{v}_{{\rm I}}|^{2}\right)+2\mathbf{v}_{{\rm I}}\right) \d \left(\rho_{{\rm I}}\mathbf{v}_{{\rm I}}\right)
                                   +\frac{\nabla_{(\rho_{{\rm I}}\mathbf{v}_{{\rm I}})}^{2}\hbar}{2} \langle \d \left(\rho_{{\rm I}}\mathbf{v}_{{\rm I}}\right), \d \left(\rho_{{\rm I}}\mathbf{v}_{{\rm I}}\right)\rangle \notag\\
                                    = & -\left(\left(1-|\mathbf{v}_{{\rm I}}|^{2}\right)\ln \left(1+|\mathbf{v}_{{\rm I}}|^{2}\right)-2|\mathbf{v}_{{\rm I}}|^{2}\right)\odiv\left(\rho_{{\rm I}}\u_{{\rm I}}\right)\d t\\
                                    &+ \left(2\mathbf{v}_{{\rm I}}\ln \left(1+|\mathbf{v}_{{\rm I}}|^{2}\right)+2\mathbf{v}_{{\rm I}}\right) \left(-\odiv\left(\rho_{{\rm I}}\u_{{\rm I}}\otimes \mathbf{v}_{{\rm I}}\right)\right)\d t\notag\\
                                    & + \left(2\mathbf{v}_{{\rm I}}\ln \left(1+|\mathbf{v}_{{\rm I}}|^{2}\right)+2\mathbf{v}_{{\rm I}}\right) \left(\odiv \mathbb{S}_{{\rm I}}- \mathbf{R}_{{\rm I}}\right)\d t \notag\\
                                    & + \left(2\mathbf{v}_{{\rm I}}\ln \left(1+|\mathbf{v}_{{\rm I}}|^{2}\right)+2\mathbf{v}_{{\rm I}}\right) \rho_{{\rm I}} \mathbb{F}(\rho_{{\rm I}},\u_{{\rm I}})\d W_{{\rm I}} \notag\\
                                    &+ \rho_{{\rm I}}\phi_{K}^{2}(\rho_{{\rm I}})\mathbb{F}(\rho_{{\rm I}},\u_{{\rm I}})\left(\left(1+\ln \left(1+|\mathbf{v}_{{\rm I}}|^{2}\right)\right)\mathbb{I}_{3}+2\frac{\rho_{{\rm I}}\mathbf{v}_{{\rm I}}\otimes \rho_{{\rm I}}\mathbf{v}_{{\rm I}}}{\left|\rho_{{\rm I}}\mathbf{v}_{{\rm I}}\right|^{2}+|\rho_{{\rm I}}|^{2}}\right)\mathbb{F}(\rho_{{\rm I}},\u_{{\rm I}})\d t.\notag
\end{align}
To simplify \eqref{it in Page 36}, we compute $  \odiv \left(\rho_{{\rm I}}\u_{{\rm I}}\otimes \mathbf{v}_{{\rm I}}\right)=\odiv \left(\rho_{{\rm I}}\u_{{\rm I}}\right)\mathbf{v}_{{\rm I}}+\left(\rho_{{\rm I}}\u_{{\rm I}}\cdot\nabla\right)\mathbf{v}_{{\rm I}}$.
Taking integration on $\O$ with respect to $t$ and $x$, we have
\begin{align}
&\int_{\O} \hbar\left(\rho_{{\rm I}},\rho_{{\rm I}}\mathbf{v}_{{\rm I}}\right) \d x \notag\\
= &- \int_{\O} \int_{0}^t \left(-|\mathbf{v}_{{\rm I}}|^{2}-1 \right)\ln \left(1+|\mathbf{v}_{{\rm I}}|^{2}\right)\odiv\left(\rho_{{\rm I}}\u_{{\rm I}}\right)\d s \d x \notag\\
 & - \int_{\O} \int_{0}^t 2\mathbf{v}_{{\rm I}}\cdot \left(\left(\rho_{{\rm I}}\u_{{\rm I}}\cdot\nabla\right)\mathbf{v}_{{\rm I}}\right)\d s \d x - \int_{\O} \int_{0}^t 2\mathbf{v}_{{\rm I}}\ln \left(1+ |\mathbf{v}_{{\rm I}}|^{2}\right)\cdot \left(\left(\rho_{{\rm I}}\u_{{\rm I}}\cdot\nabla\right)\mathbf{v}_{{\rm I}}\right) \d s \d x \notag\\
 & + \int_{\O} \int_{0}^t \left(2\mathbf{v}_{{\rm I}}\ln \left(1+ |\mathbf{v}_{{\rm I}}|^{2}\right)+2\mathbf{v}_{{\rm I}}\right)\left(\odiv \mathbb{S}_{{\rm I}}- \mathbf{R}_{{\rm I}}\right) \d s \d x \\
 & +\int_{\O} \int_{0}^t \left(2\mathbf{v}_{{\rm I}} \ln \left(1+ |\mathbf{v}_{{\rm I}}|^{2}\right)+2\mathbf{v}_{{\rm I}}\right)\rho_{{\rm I}}\phi_{K}(\rho_{{\rm I}}) \mathbb{F}(\rho_{{\rm I}},\u_{{\rm I}})\d W_{{\rm I}}\d x \notag\\
 & +\int_{\O} \int_{0}^t \rho_{{\rm I}}\phi_{K}^{2}(\rho_{{\rm I}})\mathbb{F}(\rho_{{\rm I}},\u_{{\rm I}})\left(\left(1+\ln \left(1+|\mathbf{v}_{{\rm I}}|^{2}\right)\right)\mathbb{I}_{3}+2\frac{\rho_{{\rm I}}\mathbf{v}_{{\rm I}}\otimes \rho_{{\rm I}}\mathbf{v}_{{\rm I}}}{\left|\rho_{{\rm I}}\mathbf{v}_{{\rm I}}\right|^{2}+|\rho_{{\rm I}}|^{2}}\right)\mathbb{F}(\rho_{{\rm I}},\u_{{\rm I}}) \d s \d x\notag\\
 & +\int_{\O} \hbar\left(\rho_{0,{\rm I}},\rho_{0,{\rm I}}\mathbf{v}_{0,{\rm I}}\right) \d x. \notag
\end{align}
The first three integrals on the right-hand side will cancel. 
So we have
\begin{align}\label{formula of h}
&\int_{\O}\hbar\left(\rho_{{\rm I}},\rho_{{\rm I}}\mathbf{v}_{{\rm I}}\right) \d x \notag\\
= &  \int_{0}^t\int_{\O} \left(2\mathbf{v}_{{\rm I}}\ln |\mathbf{v}_{{\rm I}}|^{2}+2\mathbf{v}_{{\rm I}}\right)\left(\odiv \mathbb{S}_{{\rm I}}- \mathbf{R}_{{\rm I}}\right)\d x \d s  \notag\\
 & +\int_{\O} \int_{0}^t \left(2\mathbf{v}_{{\rm I}} \ln |\mathbf{v}_{{\rm I}}|^{2}+2\mathbf{v}_{{\rm I}}\right)\phi_{K}(\rho_{{\rm I}})\rho_{{\rm I}} \mathbb{F}(\rho_{{\rm I}},\u_{{\rm I}})\d W_{{\rm I}}\d x \\
 & +\int_{0}^t \int_{\O}  \rho_{{\rm I}}\phi_{K}^{2}(\rho_{{\rm I}})\mathbb{F}(\rho_{{\rm I}},\u_{{\rm I}})\left(\left(1+\ln \left(1+|\mathbf{v}_{{\rm I}}|^{2}\right)\right)\mathbb{I}_{3}+2\frac{\rho_{{\rm I}}\mathbf{v}_{{\rm I}}\otimes \rho_{{\rm I}}\mathbf{v}_{{\rm I}}}{\left|\rho_{{\rm I}}\mathbf{v}_{{\rm I}}\right|^{2}+|\rho_{{\rm I}}|^{2}}\right)\mathbb{F}(\rho_{{\rm I}},\u_{{\rm I}})\d x \d s \notag\\
 & +\int_{\O} \hbar\left(\rho_{0,{\rm I}},\rho_{0,{\rm I}}\mathbf{v}_{0,{\rm I}}\right) \d x.\notag
\end{align}
Similarly formula \eqref{formula of h} holds for $|\mathbf{v}_{{\rm I}}|^{2}\geqslant n^{2}$, with setting
\begin{align}
& \hslash\left(\rho_{{\rm I}},\rho_{{\rm I}}\mathbf{v}_{{\rm I}}\right) \\
=& \rho_{{\rm I}}\left(2\left(1+\ln (1+n)\right)\frac{\left|\rho_{{\rm I}}\mathbf{v}_{{\rm I}}\right|^{2}}{\rho_{{\rm I}}^{2}}-\left(1+\frac{\left|\rho_{{\rm I}}\mathbf{v}_{{\rm I}}\right|^{2}}{\rho_{{\rm I}}^{2}}\right) \ln \left(1+\frac{\left|\rho_{{\rm I}}\mathbf{v}_{{\rm I}}\right|^{2}}{\rho_{{\rm I}}^{2}}\right)+2\left(\ln (1+n)-n\right)\right).\notag
\end{align}
In summary, we have
\begin{align}
\int_{\O}\rho_{{\rm I}} \varphi_{n}(\mathbf{v}_{{\rm I}}) \d x
=& - \int_{0}^{t} \int_{\O} \nabla_{\mathbf{v}_{{\rm I}}}\varphi_{n}(\mathbf{v}_{{\rm I}})  \mathbf{R}_{{\rm I}} \d x \d s -\int_{0}^{t} \int_{\O} \mathbb{S}_{{\rm I}}: \nabla\left(\nabla_{\mathbf{v}_{{\rm I}}}\varphi_{n}(\mathbf{v}_{{\rm I}})\right) \d x \d s \notag\\
&+ \int_{0}^{t} \int_{\O}\nabla_{\mathbf{v}_{{\rm I}}}\varphi_{n}(\mathbf{v}_{{\rm I}}) \rho_{{\rm I}} \phi_{K}(\rho_{{\rm I}})\mathbb{F}(\rho_{{\rm I}},\u_{{\rm I}}) \d x\d W_{{\rm I}}\\
& +\int_{0}^t\int_{\O}  \rho_{{\rm I}}\phi_{K}^{2}(\rho_{{\rm I}})\mathbb{F}(\rho_{{\rm I}},\u_{{\rm I}})\nabla_{\mathbf{v}_{{\rm I}}}^{2}\varphi_{n}(\mathbf{v}_{{\rm I}})\mathbb{F}(\rho_{{\rm I}},\u_{{\rm I}})\d x  \d s \notag\\
&+ \int_{\O}\rho_{0,{\rm I}} \varphi_{n}\left(\mathbf{v}_{0,{\rm I}}\right) \d x.\notag
\end{align}

\subsection{Vanishing the quantum effect: $\kappa\ra 0$}
\begin{flushleft}
\textbf{Step 1:} Choose the path space.
\end{flushleft}

We choose the space $  \mathcal{X}_{4}=\mathcal{X}_{\rho_{0}}^{4} \times \mathcal{X}_{\mathbf{q}_{0}}^{4}\times \mathcal{X}_{\frac{\mathbf{q}_{0}}{\sqrt{\rho_{0}}}}^{4}
\times\mathcal{X}_{\rho}^{4}\times\mathcal{X}_{\u}^{4}\times\mathcal{X}_{\rho\u}^{4} \times \mathcal{X}_{W}^{4}$,
 where $  \mathcal{X}_{\rho_{0}}^{4}=L^{\gamma}\left( \O \right)$, $\mathcal{X}_{\mathbf{q}_{0}}^{4}= L^{1}\left(\O\right)$,
 $ \mathcal{X}_{\frac{\mathbf{q}_{0}}{\sqrt{\rho_{0}}}}^{4}=L^{2}\left( \O \right)$,
  $ \mathcal{X}_{\rho}^{4}=L^{2}\left(0,T; L^{\frac{3}{2}}\left(\O\right)\right)\cap L^{\frac{5}{3}\gamma}\left([0,T]\times\O\right)$,
   $\mathcal{X}_{\u}^{4}= L^{2}\left([0,T]\times\O\right)\cap L^{4}\left([0,T]\times\O\right)$,
   $ \mathcal{X}_{\rho\u}^{4}$ $=C\left([0,T];L^{\frac{3}{2}}\left( \O \right)\right)$, $  \mathcal{X}_{W}^{4}=C\left([0,T];\mathcal{H}\right)$.\\

\begin{flushleft}
\textbf{Step 2:} The tightness of the laws.
\end{flushleft}
We then consider the regularity of $\rho_{{\rm I}}\u_{{\rm I}}$ and have the following tightness.
\begin{proposition}
 The set $\left\{\mathcal{L}\left[\rho_{{\rm I}}\u_{{\rm I}}\right]\right\}$ is tight on $\mathcal{X}_{\rho\u}^{4}=C\left([0,T];L^{\frac{3}{2}}\left( \O \right)\right)$ for $\gamma>1$.
\end{proposition}
Proof: We first consider the deterministic part, and we denote
\begin{align}
Y_{{\rm I}}(t)=&\rho_{{\rm I}} \mathbf{v}_{{\rm I}}(0)-\int_{0}^{t}\odiv(\rho_{{\rm I}}\u_{{\rm I}}\otimes  \mathbf{v}_{{\rm I}}) \d s+\int_{0}^{t}\odiv \mathbb{S}_{{\rm I}} \d s-\int_{0}^{t} \mathbf{R}_{{\rm I}}\d s\\
& + \int_{0}^{t} \rho_{{\rm I}}\phi_{K}^{2}(\rho_{{\rm I}})\mathbb{F}(\rho_{{\rm I}},\u_{{\rm I}})\nabla_{\mathbf{v}_{{\rm I}}}^{2}\varphi_{n}(\mathbf{v}_{{\rm I}})\mathbb{F}(\rho_{{\rm I}},\u_{{\rm I}})\d s.\notag
\end{align}
In this layer,  $\kappa$ goes to $0$ later. The high regularity on $\rho_{{\rm I}}^{\frac{1}{2}}$ and $\rho_{{\rm I}}^{\frac{1}{4}}$ will not hold uniformly on $\kappa$, but we use $  \rho_{{\rm I}}^{\gamma}\in L^{r}\left(\Omega ; L^{1}\left([0,T];L^{3}\left( \O \right)\right)\right)$, $  \rho_{{\rm I}} \in L^{r}\left(\Omega ;L^{\infty}\left([0,T];L^{\gamma}\left( \O \right)\right)\right)$ uniformly in $\kappa$ to derive that
 $\frac{\d Y_{{\rm I}}(t)}{\d t}$ is bounded in 
$L^{r}\left(\Omega; L^{\frac{4}{3}}\left([0,T];W^{-l,2}\left(\O\right)\right)\right)$ for $l>\frac{5}{2}$ uniformly in $\kappa, \eta, \delta$.
Similar as in Proposition \ref{compactness of Y},
$\mathbb{E}\left[\left\|Y_{{\rm I}}(t)\right\|_{C^{0,\varsigma}\left([0,T];W^{-l,2}\left( \O \right)\right)}^{r}\right]\ls C$ uniformly in $\kappa$, where $C$ depends on $r$, $ \varsigma<\frac{1}{4}$.
Next, we study the time regularity of stochastic integral.
Applying the Burkh\"older-Davis-Gundy's inequality and H\"older's inequality,
\begin{align}
& \mathbb{E}\left[\left(\left\|\int_{\tau_{1}}^{\tau_{2}} \phi_{K}\left(\rho_{{\rm I}}\right) \rho_{{\rm I}}\mathbb{F}\left(\rho_{{\rm I}},\u_{{\rm I}}\right)\d W_{{\rm I}}\right\|_{L^{\frac{3}{2}}\left(\O \right)}\right)^{r}\right]\notag\\
\ls & \mathbb{E}\left[\left(\int_{\O}\left(\int_{\tau_{1}}^{\tau_{2}}\left|\rho_{{\rm I}} \mathbb{F}\left(\rho_{{\rm I}},\u_{{\rm I}}\right)\right|\d W_{{\rm I}} \right)^{\frac{3}{2}}\d x\right)^{\frac{2}{3}r}\right] \notag\\
\ls & \mathbb{E}\left[\left(\int_{\O}\left(\int_{\tau_{1}}^{\tau_{2}}\left|\rho_{{\rm I}} \mathbb{F}\left(\rho_{{\rm I}},\u_{{\rm I}}\right)\right|^{2}\d t \right)^{\frac{3}{4}}\d x\right)^{\frac{2}{3}r}\right] \notag \\
\ls & \mathbb{E}\left[\left(\int_{\O}\left(\int_{\tau_{1}}^{\tau_{2}}\sum\limits_{k=0}^{\infty} f_{k}^{2}\left(\left|\rho_{{\rm I}}\right| + \left|\rho_{{\rm I}} \u_{{\rm I}} \right|\right)^{2}\d t \right)^{\frac{3}{4}}\d x\right)^{\frac{2}{3}r}\right]\\
= & \left(\sum\limits_{k=0}^{\infty} f_{k}^{2}\right)^{\frac{r}{2}} \mathbb{E}\left[\left(\int_{\O}\left(\int_{\tau_{1}}^{\tau_{2}}\left(\left|\rho_{{\rm I}}\right| + \left|\rho_{{\rm I}} \u_{{\rm I}} \right|\right)^{2}\d t \right)^{\frac{3}{4}}\d x\right)^{\frac{2}{3}r}\right]\notag\\
\ls &C \mathbb{E}\left[\left(\int_{\O}\left(\int_{\tau_{1}}^{\tau_{2}}\left(\left|\rho_{{\rm I}}\right|^{2} + \left|\rho_{{\rm I}} \u_{{\rm I}} \right|^{2}\right)\d t \right)^{\frac{3}{4}}\d x\right)^{\frac{2}{3}r}\right]  \notag\\
\ls & C\mathbb{E}\left[\left(\left(\int_{\O}\int_{\tau_{1}}^{\tau_{2}}\left(\left|\rho_{{\rm I}}\right|^{2} + \left|\rho_{{\rm I}} \u_{{\rm I}} \right|^{2}\right)\d t \d x\right)^{\frac{3}{4}}\left(\int_{\O}\d x\right)^{\frac{1}{4}}\right)^{\frac{2}{3}r}\right] \notag\\
=& C \mathbb{E}\left[\left(\left(\int_{\O}\int_{\tau_{1}}^{\tau_{2}}\left(\left|\rho_{{\rm I}}\right|^{2} + \left|\rho_{{\rm I}} \u_{{\rm I}} \right|^{2}\right)\d t \d x\right)^{\frac{3}{4}}\right)^{\frac{2}{3}r}\right], \notag
\end{align}
where $C$ depends on $r, \sum\limits_{k=0}^{\infty} f_{k}^{2}$. By the B-D entropy estimates, we have
$$\nabla \rho_{{\rm I}}^{\frac{1}{2}}\in L^{r}\left(\Omega; L^{\infty}\left([0,T];L^{2}\left(\O\right)\right)\right),$$
 by Sobolev's embedding, we further have
 \begin{align}
\rho_{{\rm I}}\in L^{r}\left(\Omega; L^{\infty}\left([0,T];L^{3}\left(\O\right)\right)\right).
\end{align}
We estimate
\begin{align}\label{sto estimate in kappa1}
\int_{\O}\int_{\tau_{1}}^{\tau_{2}}\left|\rho_{{\rm I}}\right|^{2} \d t \d x =&\int_{\tau_{1}}^{\tau_{2}}\int_{\O}\left|\rho_{{\rm I}}\right|^{2}\d x \d t
\ls \left\|\rho_{{\rm I}}\right\|_{L_t^{\infty}L_{x}^{3}}^{2}\int_{\tau_{1}}^{\tau_{2}}\left(\int_{\O}\d x\right)^{\frac{1}{3}} \d t \\
\ls & \left\|\rho_{{\rm I}}\right\|_{L_t^{\infty}L_{x}^{3}}^{2}\left(\tau_{2}-\tau_{1}\right). \notag
\end{align}
So we have
\begin{align}
\mathbb{E}\left[\left(\left(\int_{\O}\int_{\tau_{1}}^{\tau_{2}}\left|\rho_{{\rm I}}\right|^{2}\d t \d x \right)^{\frac{3}{4}}\right)^{\frac{2}{3}r}\right]\ls \mathbb{E}\left[\left\|\rho_{{\rm I}}\right\|_{L_t^{\infty}L_{x}^{3}}^{r}\right]\left(\tau_{2}-\tau_{1}\right)^{\frac{r}{2}}.
\end{align}
We employ the regularity of $r_{2}\mathbb{E}\left[\left(\int_{0}^{t}\int_{\O}\rho_{{\rm I}}\left|\u_{{\rm I}}\right|^{4}\d x \d t\right)^{r}\right] \ls C$ and the high regularity of $\rho$ from the B-D entropy to estimate
\begin{align}\label{sto estimate in kappa2}
&\int_{\O}\int_{\tau_{1}}^{\tau_{2}}\left|\rho_{{\rm I}} \u_{{\rm I}} \right|^{2}\d t \d x
\ls \left\|\rho_{{\rm I}}^{\frac{1}{2}} \left|\u_{{\rm I}}\right|^{2}\right\|_{L_{t}^{2}L_{x}^{2}}\left\| \rho_{{\rm I}}^{\frac{3}{2}}\right\|_{L_{t}^{2}L_{x}^{2}}\\
\ls &\left\|\rho_{{\rm I}}^{\frac{1}{2}} \left|\u_{{\rm I}}\right|^{2}\right\|_{L_{t}^{2}L_{x}^{2}}\left\| \rho_{{\rm I}}^{\frac{3}{2}}\right\|_{L_{t}^{\infty}L_{x}^{2}}\left(\tau_{2}-\tau_{1}\right)^{\frac{1}{2}}.\notag
\end{align}

Thereupon, up to a modification it holds
\begin{equation}\label{continuity of stochastic terms in kappa}
\int_{0}^{t} \rho_{{\rm I}} \mathbb{F}(\rho_{{\rm I}},\u_{{\rm I}})\d W_{{\rm I}}\in L^{r}\left(\Omega; C^{0,\varsigma}\left([0,T];L^{2}\left( \O \right)\right)\right),
\end{equation}
where $\varsigma < \frac{1}{4}-\frac{1}{r}$ is a  positive constant. We choose some $r>4$ here and hereafter. 
The analysis of deterministic part $Y_{{\rm I}}(t)$ and (\ref{continuity of stochastic terms in kappa}) imply that $\left\|\rho_{{\rm I}}\u_{{\rm I}}\right\|_{L^{2}\left(\O\right)}$ is continuous in $[0,T]$ uniformly in $\kappa$, $\eta$, $\delta$ and $r_{0}$, up to a modification, i.e.,
\begin{equation}\label{strong cover of rho u}
\mathbb{E}\left[\left\| \rho_{{\rm I}}\u_{{\rm I}}\right\|_{C^{0,\varsigma}\left([0,T];L^{\frac{3}{2}}\left(\O\right)\right)}^{r}\right]\ls C ,
\end{equation}
where C depends on $r$ and the initial data. With a standard argument as in previous section, we have the tightness of $\mathcal{L}\left[\{\rho_{{\rm I}}\u_{{\rm I}}\}\right]$
in $\mathcal{X}_{\rho\u}^{4} $. \hfill$\square$

\begin{flushleft}
\textbf{Step 3:} The convergence $\kappa\ra 0$.
\end{flushleft}

We have the corresponding tightness and Skorokhod's representation theorem.
\begin{proposition}
There exists  two families of $\mathcal{X}_{4}$-valued Borel measurable random variables,
 $$
 \left\{\bar{\rho}_{0, {\rm I}}, \bar{\mathbf{q}}_{0, {\rm I}}, \frac{\bar{\mathbf{q}}_{0, {\rm I}}}{\sqrt{\bar{\rho}_{0, {\rm I}}}}, \bar{\rho}_{{\rm I}}, \bar{\u}_{{\rm I}}, \bar{\rho}_{{\rm I}}\bar{\u}_{{\rm I}}, \bar{W}_{{\rm I}} \right\},
 $$
  and $\left\{\rho_{0, {\rm II}}, \mathbf{q}_{0, {\rm II}},\frac{\mathbf{q}_{0}}{\sqrt{\rho_{0}}}, \rho_{{\rm II}}, \u_{{\rm II}},\rho_{{\rm II}}\u_{{\rm II}}, W_{{\rm II}}\right\}$, defined on a new complete probability space
  $\left(\bar{\Omega},\bar{\mathcal{F}},\bar{\mathbb{P}}\right)$, such that (up to a subsequence):
\begin{enumerate}
  \item
  $\mathcal{L}\left[\bar{\rho}_{0, {\rm I}}, \bar{\mathbf{q}}_{0, {\rm I}}, \frac{\bar{\mathbf{q}}_{0, {\rm I}}}{\sqrt{\bar{\rho}_{0, {\rm I}}}}, \bar{\rho}_{{\rm I}}, \bar{\u}_{{\rm I}}, \bar{\rho}_{{\rm I}}\bar{\u}_{{\rm I}}, \bar{W}_{{\rm I}} \right],$
   coincides with \\
    $\mathcal{L}\left[\rho_{0, {\rm I}}, \mathbf{q}_{0, {\rm I}}, \frac{\mathbf{q}_{0, {\rm I}}}{\sqrt{\rho_{0, {\rm I}}}}, \rho_{{\rm I}}, \u_{{\rm I}},\rho_{{\rm I}}\u_{{\rm I}}, W_{{\rm I}}\right]$ on $\mathcal{X}_{4}$;
  \item $\mathcal{L}\left[\rho_{0, {\rm II}}, \mathbf{q}_{0, {\rm II}},\frac{\mathbf{q}_{0, {\rm II}}}{\sqrt{\rho_{0, {\rm II}}}}, \rho_{{\rm II}}, \u_{{\rm II}},\rho_{{\rm II}}\u_{{\rm II}}, W_{{\rm II}}\right]$ on $\mathcal{X}_{4}$ is a Radon measure;
  \item The family of random variables $\left\{\bar{\rho}_{0, {\rm I}}, \bar{\mathbf{q}}_{0, {\rm I}}, \frac{\bar{\mathbf{q}}_{0, {\rm I}}}{\sqrt{\bar{\rho}_{0, {\rm I}}}}, \bar{\rho}_{{\rm I}}, \bar{\u}_{{\rm I}}, \bar{\rho}_{{\rm I}}\bar{\u}_{{\rm I}}, \bar{W}_{{\rm I}}\right\}$
   converges to\\
    $\left\{\rho_{0, {\rm II}}, \mathbf{q}_{0, {\rm II}},\frac{\mathbf{q}_{0, {\rm II}}}{\sqrt{\rho_{0, {\rm II}}}}, \rho_{{\rm II}}, \u_{{\rm II}},\rho_{{\rm II}}\u_{{\rm II}}, W_{{\rm II}}\right\}$ in the topology of $\mathcal{X}_{4}$, $\bar{\mathbb{P}}$ {\rm a.s.} as $\kappa\ra 0$.
    \end{enumerate}
    \end{proposition}

\begin{flushleft}
\textbf{Step 4:} The estimates after taking the limit.
\end{flushleft}

Corresponding to \cite{Vasseur-Yu2016}, we handle the deterministic terms by taking $K=\kappa^{-\frac{3}{4}}$.
\begin{align}
&\mathbb{E}\left[\left|\int_{\O} \bar{\rho}_{{\rm I}} \varphi_{n}(\bar{\mathbf{v}}_{{\rm I}}) \d x\right|^{r}\right]\notag\\
\ls & C_{r}\mathbb{E}\left[\left|- \int_{0}^{t} \int_{\O} \nabla_{\bar{\mathbf{v}}_{{\rm I}}}\varphi_{n}(\bar{\mathbf{v}}_{{\rm I}})  \mathbf{R}_{{\rm I}} \d x \d s - \int_{0}^{t} \int_{\O} \mathbb{S}_{{\rm I}}: \nabla\left(\nabla_{\bar{\mathbf{v}}_{{\rm I}}}\varphi_{n}(\bar{\mathbf{v}}_{{\rm I}})\right) \d x \d s \right|^{r}\right]\notag\\
&+ C_{r}\mathbb{E}\left[\left|\int_{0}^{t} \int_{\O}\nabla_{\bar{\mathbf{v}}_{{\rm I}}}\varphi_{n}(\bar{\mathbf{v}}_{{\rm I}}) \phi_{K}(\bar{\rho}_{{\rm I}})\bar{\rho}_{{\rm I}} \mathbb{F}(\bar{\rho}_{{\rm I}},\bar{\u}_{{\rm I}}) \d x\d W_{{\rm I}}\right|^{r}\right]\\
&+ C_{r}\mathbb{E}\left[\left|\int_{0}^t\int_{\O}  \bar{\rho}_{{\rm I}}\phi_{K}^{2}(\bar{\rho}_{{\rm I}})\mathbb{F}(\bar{\rho}_{{\rm I}},\bar{\u}_{{\rm I}})\nabla_{\bar{\mathbf{v}}_{{\rm I}}}^{2}\varphi_{n}(\bar{\mathbf{v}}_{{\rm I}})\mathbb{F}(\bar{\rho}_{{\rm I}},\bar{\u}_{{\rm I}}) \d x \d s\right|^{r}\right]\notag\\
&+ C_{r}\mathbb{E}\left[\left| \int_{\O}\rho_{0,{\rm I}} \varphi_{n}\left(\frac{\mathbf{q}_{0,{\rm I}}}{\rho_{0,{\rm I}}}\right) \d x\right|^{r}\right]\notag\\
\ls &C, \notag
\end{align}
where $C$ depends on $\mathbb{E}[(e_{{\rm I}}(0))^{r}],r,t\text{ and }n$, $C_{r}$ is a constant merely dependent of $r$.
We study the convergence of $\bar{\rho}_{{\rm I}}$.
\begin{equation}
\begin{aligned}
 &\quad \mathbb{E}\left[\left\|\bar{\rho}_{{\rm I}}^{\frac{1}{2}}\bar{\rho}_{{\rm I}} ^{\frac{1}{2}}\odiv\bar{\u}_{{\rm I}} \right\|_{L_t^{2} L_x^{\frac{3}{2}}}^{r}\right]  \ls \mathbb{E}\left[\left\|\bar{\rho}_{{\rm I}}^{\frac{1}{2}}\right\|_{L_t^{\infty}L_x^{6}}^{r}\left\|\bar{\rho}_{{\rm I}}^{\frac{1}{2}}\odiv\bar{\u}_{{\rm I}}
       \right\|_{L_t^{2}L_x^{2}}^{r}\right] \ls C,\\
\end{aligned}
\end{equation}
together with
\begin{equation}
\begin{aligned}
&\quad \mathbb{E}\left[\left\| \nabla \bar{\rho}_{{\rm I}}^{\frac{1}{2}}\left(\bar{\rho}_{{\rm I}} ^{\frac{1}{2}}\bar{\u}_{{\rm I}}\right)\right\|_{L_t^{2} L_x^{1}}^{r}\right]\ls\mathbb{E}\left[\left\|\nabla \bar{\rho}_{{\rm I}}^{\frac{1}{2}}\right\|_{L_t^{\infty}L_x^{2}}^{r}\left\|\bar{\rho}_{{\rm I}}^{\frac{1}{4}}\right\|_{L_t^{\infty}L_x^{4}}^{r}\left\|\bar{\rho}_{{\rm I}}^{\frac{1}{4}}\bar{\u}_{{\rm I}} \right\|_{L_t^{4}L_x^{4}}^{r}\right]\ls C,
\end{aligned}
\end{equation}
implies
\begin{equation}
\begin{aligned}
&\quad \mathbb{E}\left[\left\|\left(\bar{\rho}_{{\rm I}}\right)_{t}\right\|_{L_t^2 L_x^{1}}^{r}\right]\ls  C,\\
\end{aligned}
\end{equation}
uniformly in $\delta,\eta$, here $C$ depends on $r, r_{1}$, $r>4$. \\
We calculate that
\begin{equation}
|\nabla \bar{\rho}_{{\rm I}}|=2\left|\nabla \bar{\rho}_{{\rm I}} ^{\frac{1}{2}}\right| \bar{\rho}_{{\rm I}} ^{\frac{1}{2}},\\
\end{equation}
so $\{\bar{\rho}_{{\rm I}}\}$ is bounded in $L^{r}\left(\Omega;L^{\infty}\left(0, T; W^{1,1}\left( \O \right)\right)\right)$. By Aubin-Lion's lemma we obtain the strong convergence of $\bar{\rho}_{{\rm I}}$ in $L^{r}\left(\Omega;L^{2}\left(0, T; L^{\frac{3}{2}}\left( \O \right)\right)\right)$. \\
 From the stochastic B-D entropy \eqref{BD entropy conclusion after vanish eps}, we know
 \begin{equation}\label{estimate for log rho}
-r_{2}\mathbb{E}\left[\left|\int_{\O} \ln_{-} \bar{\rho}_{{\rm I}} \d x\right|^{r}\right] \ls C,\text{ for some }r> 4,
\end{equation}
where $C$ depends on $r, T, \mathbb{E}[e_{{\rm I}}(0)]^{r}, \mathbb{E}[\tilde{e}_{{\rm I}}(0)]^{r})$.
Vassuer-Yu \cite{Vasseur-Yu2016} argued that
\begin{equation}
\int_{\Omega}\left(\ln \left(\frac{1}{\rho_{{\rm II}}}\right)\right)_{+} d x  \leq \int_{\Omega} \lim \inf_{\kappa \rightarrow 0} \left(\ln \left(\frac{1}{\bar{\rho}_{{\rm I}}}\right)\right)_{+} d x  \leq \lim \inf_{\kappa \rightarrow 0} \int_{\Omega}\left(\ln \left(\frac{1}{\bar{\rho}_{{\rm I}}}\right)\right)_{+} d x,
\end{equation}
by Fatou's lemma, so $\left(\ln \left(\frac{1}{\rho_{{\rm II}}}\right)\right)_{+}$ is bounded in $L^{r}\left(\Omega; L^{\infty}\left(0, T ; L^{1}\left(\O\right)\right)\right)$. Therefore (\ref{estimate for log rho}) allows us to conclude that
\begin{equation}
\left|\left\{x \mid\rho_{{\rm II}}(t, x)=0\right\}\right|=0, \quad \text { for almost every } t, x, \omega.
\end{equation}
For almost every $(t, x)$, $\bar{\rho}_{{\rm I}}(t, x)\neq 0$, the convergence
$
\bar{\u}_{{\rm I}}(t, x)=\frac{\bar{\rho}_{{\rm I}} \bar{\u}_{{\rm I}}}{\bar{\rho}_{{\rm I}}} \rightarrow \u_{{\rm II}}(t, x)
$
and
$
\bar{\mathbf{v}}_{{\rm I}} \rightarrow \u_{{\rm II}}(t, x)
$
hold $\mathbb{P}$ a.s. as $\kappa \rightarrow 0$. $\phi_{K}(\bar{\rho}_{{\rm I}})\ra 1$ holds almost everywhere as $K\ra \infty$, $\nabla_{\bar{\mathbf{v}}_{{\rm I}}} \varphi_{n }\ls 1+\ln(1+C_{n})$ uniformly in $K$ for fixed $n$, therefore $\nabla_{\bar{\mathbf{v}}_{{\rm I}}}\varphi_{n}(\bar{\mathbf{v}}_{{\rm I}}) \ra \nabla_{\u_{{\rm II}}}\varphi_{n}\left(\u_{{\rm II}}\right)$ holds almost everywhere. $  r_{0}|\bar{\u}_{{\rm I}}|^{2} \bar{\u}_{{\rm I}} \phi_{K}(\bar{\rho}_{{\rm I}})+\eta \nabla\bar{\rho}_{{\rm I}}^{-10}\phi_{K}(\bar{\rho}_{{\rm I}}) + \delta \bar{\rho}_{{\rm I}}\nabla \triangle^{9}\bar{\rho}_{{\rm I}}\phi_{K}(\bar{\rho}_{{\rm I}}) \ra r_{0}|\u_{{\rm II}}|^{2} \u_{{\rm II}}+\eta \nabla \rho_{{\rm II}}^{-10} + \delta \rho_{{\rm II}}\nabla \triangle^{9}\rho_{{\rm II}}$ holds almost everywhere. These three terms are $L^{1}$ integrable. By Lebesgue's dominated convergence theorem, taking the limit $K\ra \infty$, i.e. $\kappa\ra 0$, we have
\begin{equation}
\int_{\O} \bar{\rho}_{{\rm I}} \varphi_{n}( \phi_{K}(\bar{\rho}_{{\rm I}})\bar{\u}_{{\rm I}}) \d x \ra \int_{\O}\rho_{{\rm II}} \varphi_{n}(\u_{{\rm II}}) \d x \quad \bar{\mathbb{P}} \text{ a.s.}
 \end{equation}
For the pressure term, integration by parts implies
\begin{align}
&-\int_0^t\int_{\O} \nabla_{\bar{\mathbf{v}}_{{\rm I}}}\varphi_{n}(\bar{\mathbf{v}}_{{\rm I}})\nabla \bar{\rho}_{{\rm I}}^{\gamma} \phi_{K}(\bar{\rho}_{{\rm I}})\d x\d s\\
=&\int_0^t\int_{\O} \nabla_{\bar{\mathbf{v}}_{{\rm I}}}^{2}\varphi_{n}(\bar{\mathbf{v}}_{{\rm I}}):\nabla \bar{\mathbf{v}}_{{\rm I}}\bar{\rho}_{{\rm I}}^{\gamma}\d x\d s+\int_0^t\int_{\O} \nabla_{\bar{\mathbf{v}}_{{\rm I}}}\varphi_{n}(\bar{\mathbf{v}}_{{\rm I}})\nabla\phi_{K}(\bar{\rho}_{{\rm I}})\bar{\rho}_{{\rm I}}^{\gamma}\d x\d s
\triangleq  P_{1}+P_{2}, \notag
\end{align}
where
\begin{align}
P_{1}
=&\int_0^t\int_{\O} \bar{\rho}_{{\rm I}}^{\gamma}\phi_{K}^{2}(\bar{\rho}_{{\rm I}}) \nabla_{\bar{\mathbf{v}}_{{\rm I}}}^{2}\varphi_{n}(\bar{\mathbf{v}}_{{\rm I}}):\nabla \bar{\u}_{{\rm I}}\d x\d s\\
 &+ \int_0^t\int_{\O} \bar{\rho}_{{\rm I}}^{\gamma}  \phi_{K}(\bar{\rho}_{{\rm I}}) \nabla_{\bar{\mathbf{v}}_{{\rm I}}}^{2}\varphi_{n}(\bar{\mathbf{v}}_{{\rm I}}):\nabla \phi_{K}(\bar{\rho}_{{\rm I}})\otimes \bar{\u}_{{\rm I}} \d x\d s
\triangleq  P_{11}+P_{12}. \notag
\end{align}
Due to $  \phi_{K}^{\prime}(\bar{\rho}_{{\rm I}})\bar{\rho}_{{\rm I}}^{\frac{1}{2}}\ls \frac{4}{K^{\frac{1}{2}}}$, we have
\begin{align}
\mathbb{E}\left[|P_{2}|^{r}\right] \ls & C \mathbb{E}\left[\kappa^{-\frac{1}{4}} \left\|\bar{\rho}_{{\rm I}}^{\frac{1}{2}}\phi_{K}'(\bar{\rho}_{{\rm I}})\right\|_{L_{t}^{\infty}L_{x}^{\infty}}^{r} \left\|\kappa^{\frac{1}{4}} \nabla\bar{\rho}_{{\rm I}}^{\frac{1}{4}}\right\|_{L_{t}^{4}L_{x}^{4}}^{r}\left\|\bar{\rho}_{{\rm I}}^{\gamma+\frac{1}{4}}\right\|_{L_{t}^{\frac{4}{3}}L_{x}^{\frac{4}{3}}}^{r}\right]\\
\ls & C\left(\kappa^{-\frac{1}{4}}\frac{2}{K^{\frac{1}{2}}}\right)^{r}=C\kappa^{\frac{1}{8}r}\ra 0 \text{ as }\kappa \ra 0,\notag
\end{align}
where $C$ depends on $\mathbb{E}[(e_{{\rm I}}(0))^{r}],r,t\text{ and } n $. We estimate
\begin{align}
\mathbb{E}\left[|P_{12}|^{r}\right]
= &\mathbb{E}\left[\left|\int_{\O}\int_0^t \bar{\rho}_{{\rm I}}^{\gamma+\frac{1}{4}}\phi_{K}(\bar{\rho}_{{\rm I}}) \nabla_{\bar{\mathbf{v}}_{{\rm I}}}^{2}\varphi_{n}(\bar{\mathbf{v}}_{{\rm I}}):\phi_{K}'(\bar{\rho}_{{\rm I}})2\bar{\rho}_{{\rm I}}^{\frac{1}{2}}\nabla\bar{\rho}_{{\rm I}}^{\frac{1}{4}} \bar{\u}_{{\rm I}}\d s\d x \right|^{r}\right]\notag\\
\ls & C \mathbb{E}\left[\kappa^{-\frac{1}{4}} \left\|\bar{\rho}_{{\rm I}}^{\frac{1}{2}}\phi_{K}'(\bar{\rho}_{{\rm I}})\right\|_{L_{t}^{\infty}L_{x}^{\infty}}^{r} \left\|\kappa^{\frac{1}{4}} \nabla\bar{\rho}_{{\rm I}}^{\frac{1}{4}}\right\|_{L_{t}^{4}L_{x}^{4}}^{r}\left\|\bar{\rho}_{{\rm I}}^{\gamma+\frac{1}{4}}\right\|_{L_{t}^{\frac{4}{3}}L_{x}^{\frac{4}{3}}}^{r}\right]\\
\ls & C\left(\kappa^{-\frac{1}{4}}\frac{2}{K^{\frac{1}{2}}}\right)^{r}= C\kappa^{\frac{1}{8}r}\ra 0 \text{ as }\kappa \ra 0,\notag
\end{align}
in which we used $  \left| \nabla_{\bar{\mathbf{v}}_{{\rm I}}}^{2}\varphi_{n}\mathbf{v}_{k}\right|\ls C(n)$, with $C$ depending on $\mathbb{E}[(e_{{\rm I}}(0))^{r}],r,t\text{ and } n $.
Since
\begin{equation}
  \nabla_{\bar{\mathbf{v}}_{{\rm I}}}^{2}\varphi_{n}(\bar{\mathbf{v}}_{{\rm I}}):\nabla\bar{\u}_{{\rm I}}=4\tilde{\varphi}_{n}^{\prime \prime}\left(|\bar{\mathbf{v}}_{{\rm I}}|^{2}\right)\bar{\mathbf{v}}_{{\rm I}}\otimes\bar{\mathbf{v}}_{{\rm I}}:\nabla\bar{\u}_{{\rm I}}
   +2\tilde{\varphi}_{n}^{\prime}\left(\left|\bar{\mathbf{v}}_{{\rm I}}\right|^{2}\right)\odiv \bar{\u}_{{\rm I}},
\end{equation}
 we have
\begin{align}\label{Mellet-Vasseur P11 estimates}
   &P_{11}=4\int_0^t \int_{\O} \phi_{K}^{2}(\bar{\rho}_{{\rm I}})\bar{\rho}_{{\rm I}}^{\gamma}\left|\nabla\bar{\u}_{{\rm I}}\right|\d x\d s
           +2\int_0^t \int_{\O} \phi_{K}^{2}(\bar{\rho}_{{\rm I}})\bar{\rho}_{{\rm I}}^{\gamma}\left|\odiv \bar{\u}_{{\rm I}} \right|\tilde{\varphi}_{n}^{\prime}\left(\left|\bar{\mathbf{v}}_{{\rm I}}\right|^{2}\right)\d x\d s\notag\\
\ls & 4\int_0^t \int_{\O} \bar{\rho}_{{\rm I}}\left|\nabla  \bar{\u}_{{\rm I}}\right|^{2}+4 \int_{\O}\int_0^t \bar{\rho}_{{\rm I}}^{2\gamma-1}\d x\d s+2\int_0^t\int_{\O}  \phi_{K}^{2}(\bar{\rho}_{{\rm I}})\bar{\rho}_{{\rm I}}^{\gamma}\left|\mathbb{D} \bar{\u}_{{\rm I}} \right|\tilde{\varphi}_{n}'(\left|\bar{\mathbf{v}}_{{\rm I}}\right|^{2})\d x\d s \\
\ls & C+2\int_0^t \int_{\O} \phi_{K}^{4}(\bar{\rho}_{{\rm I}})\bar{\rho}_{{\rm I}}\left|\mathbb{D}\bar{\u}_{{\rm I}}\right|^{2}\tilde{\varphi}_{n}'\left(\left|\bar{\mathbf{v}}_{{\rm I}}\right|^{2}\right)\d x\d s +2\int_0^t \int_{\O} \tilde{\varphi}_{n}'\left(\left|\bar{\mathbf{v}}_{{\rm I}}\right|^{2}\right)\bar{\rho}_{{\rm I}}^{2\gamma-1}\d x\d s ,\notag
\end{align}
with $C$ depending on $\mathbb{E}[(e_{{\rm I}}(0))^{r}],r,t\text{ and } n $. 
\begin{align}
 \mathbb{S}_{{\rm I}}= \bar{\rho}_{{\rm I}}\phi_{K}(\bar{\rho}_{{\rm I}})\left(\mathbb{D} \bar{\u}_{{\rm I}}+\kappa \frac{\Delta \sqrt{\bar{\rho}_{{\rm I}}}}{\sqrt{\bar{\rho}_{{\rm I}}}} \mathbb{I}_{3}\right)\triangleq \mathbb{S}_{1}+\mathbb{S}_{2},
\end{align}
We estimate term by term,
\begin{align}
&\int_0^t \int_{\O}\nabla  \nabla_{\bar{\mathbf{v}}_{{\rm I}}}\varphi_{n}(\bar{\mathbf{v}}_{{\rm I}}): \mathbb{S}_{1} \d x\d s
=\int_0^t \int_{\O} \nabla \left( \nabla_{\bar{\mathbf{v}}_{{\rm I}}}\varphi_{n}(\bar{\mathbf{v}}_{{\rm I}})\right):\bar{\rho}_{{\rm I}}\phi_{K}(\bar{\rho}_{{\rm I}})\mathbb{D} \bar{\u}_{{\rm I}}  \d x \d s \notag\\
=&\int_0^t \int_{\O} \bar{\rho}_{{\rm I}}\phi_{K}(\bar{\rho}_{{\rm I}}) \nabla_{\bar{\mathbf{v}}_{{\rm I}}}^{2}\varphi_{n}(\bar{\mathbf{v}}_{{\rm I}})\nabla\bar{\u}_{{\rm I}}:\mathbb{D} \bar{\u}_{{\rm I}}  \d x \d s\\
 & +\int_0^t \int_{\O} \bar{\rho}_{{\rm I}}\phi_{K}(\bar{\rho}_{{\rm I}})\bar{\u}_{{\rm I}}^{\top} \nabla_{\bar{\mathbf{v}}_{{\rm I}}}^{2}\varphi_{n}(\bar{\mathbf{v}}_{{\rm I}})\mathbb{D} \bar{\u}_{{\rm I}}\nabla\left(\phi_{K}(\bar{\rho}_{{\rm I}})\right) \d x \d s\notag\\
\triangleq & A_{1}+ A_{2}.\notag
\end{align}
Similarly, $\mathbb{E}\left[|A_{2}|^{r}\right] \ra 0$ as $\kappa \ra 0 $.
\begin{align}
A_{1}= & 2 \int_0^t \int_{\O} \phi_{K}^{2}(\bar{\rho}_{{\rm I}})\tilde{\varphi}_{n}'(|\bar{\mathbf{v}}_{{\rm I}}|^{2}) \bar{\rho}_{{\rm I}}\nabla\bar{\u}_{{\rm I}}: \mathbb{D} \bar{\u}_{{\rm I}} \d x \d s \notag\\
      & + 4\int_0^t \int_{\O} \phi_{K}^{2}(\bar{\rho}_{{\rm I}})\tilde{\varphi}_{n}''\left(|\bar{\mathbf{v}}_{{\rm I}}|^{2}\right)\bar{\mathbf{v}}_{{\rm I}}\otimes \bar{\mathbf{v}}_{{\rm I}}\nabla \bar{\u}_{{\rm I}}: \mathbb{D} \bar{\u}_{{\rm I}} \d x \d s\\
\geq & 2 \int_0^t \int_{\O}\phi_{K}^{4}(\bar{\rho}_{{\rm I}})\tilde{\varphi}_{n}'(|\bar{\mathbf{v}}_{{\rm I}}|^{2})\bar{\rho}_{{\rm I}}\left|\mathbb{D}\bar{\u}_{{\rm I}}\right|^{2} \d s \d x -4 \int_{\O}\int_0^t \phi_{K}^{2}(\bar{\rho}_{{\rm I}})\bar{\rho}_{{\rm I}}\left|\mathbb{D}\bar{\u}_{{\rm I}}\right|^{2} \d x \d s, \notag
\end{align}
combining with the pressure term $P_{11}$, it holds
\begin{align}
-A_{1}+P_{11}\ls& (-2+2) \int_0^t \int_{\O}\phi_{K}^{4}(\bar{\rho}_{{\rm I}})\tilde{\varphi}_{n}'(|\bar{\mathbf{v}}_{{\rm I}}|^{2})\bar{\rho}_{{\rm I}}\left|\mathbb{D}\bar{\u}_{{\rm I}}\right|^{2}\d x \d s \\
&+ 4\int_0^t\int_{\O} \phi_{K}^{2}(\bar{\rho}_{{\rm I}})\bar{\rho}_{{\rm I}}\left|\mathbb{D}\bar{\u}_{{\rm I}}\right|^{2} \d x \d s +2\int_0^t\int_{\O}\tilde{\varphi}_{n}'\left(\left|\bar{\mathbf{v}}_{{\rm I}}\right|^{2}\right)\bar{\rho}_{{\rm I}}^{2\gamma-1} \d x\d s,\notag
\end{align}
and
\begin{equation}
\mathbb{E}\left[\left| -A_{1}+P_{11}\right|^{r}\right]\ls C+ C_{r} \mathbb{E}\left[\left| 2\int_{\O}\int_0^t\tilde{\varphi}_{n}'\left(\left|\bar{\mathbf{v}}_{{\rm I}}\right|^{2}\right)\bar{\rho}_{{\rm I}}^{2\gamma-1}\d s \d x\right|^{r}\right],
\end{equation}
where $C$ depends on $\mathbb{E}[(e_{{\rm I}}(0))^{r}],r,t \text{ and } n $. We also have
\begin{align}
    &\mathbb{E}\left[\left|\int_0^t \int_{\O} 2 \kappa  \nabla_{\bar{\mathbf{v}}_{{\rm I}}}\varphi_{n}(\bar{\mathbf{v}}_{{\rm I}})\phi_{K}(\bar{\rho}_{{\rm I}}) \nabla \sqrt{\bar{\rho}_{{\rm I}}} \triangle \sqrt{\bar{\rho}_{{\rm I}}}\d x\d s\right|^{r}\right]\\
\ls &C  \kappa \kappa^{-\frac{1}{4}}\kappa^{-\frac{1}{2}}\left\|\kappa^{\frac{1}{4}}\nabla\bar{\rho}_{{\rm I}}^{\frac{1}{4}}\right\|_{L_{t}^{4}L_{x}^{4}}^{r}
\left\|\kappa^{\frac{1}{2}}\triangle\bar{\rho}_{{\rm I}}^{\frac{1}{2}}\right\|_{L_{t}^{2}L_{x}^{2}}^{r}\left\|\bar{\rho}_{{\rm I}}^{\frac{1}{4}}\right\|_{L_{t}^{4}L_{x}^{4}}^{r}
\ls  C \kappa^{\frac{1}{4}} \text{ as }\kappa\ra 0,\notag
\end{align}
where $C$ depends on $\mathbb{E}[(e_{{\rm I}}(0))^{r}],r,t \text{ and } n $.
 Similarly, there holds
\begin{equation}
\mathbb{E}\left[\left|\int_0^t\int_{\O}  \nabla_{\bar{\mathbf{v}}_{{\rm I}}}\varphi_{n}(\bar{\mathbf{v}}_{{\rm I}})\mathbb{S}_{2}\d x\d s\right|^{r}\right]\ra 0;
\end{equation}
\begin{equation}
\int_{\O}\int_0^t  \nabla_{\bar{\mathbf{v}}_{{\rm I}}}\varphi_{n}(\bar{\mathbf{v}}_{{\rm I}})\left(r_{0}|\bar{\u}_{{\rm I}}|^{2} \bar{\u}_{{\rm I}}  +r_{1} \bar{\rho}_{{\rm I}}|\bar{\u}_{{\rm I}}|^{2} \bar{\u}_{{\rm I}} + r_{2}\bar{\u}_{{\rm I}} \right) \phi_{K}(\bar{\rho}_{{\rm I}})\d s\d x\geqslant 0;
\end{equation}
and
\begin{align}
& \int_{0}^t\int_{\O} \bar{\rho}_{{\rm I}}\phi_{K}^{2}\left(\bar{\rho}_{{\rm I}}\right)\mathbb{F}\left(\bar{\rho}_{{\rm I}},\bar{\u}_{{\rm I}}\right) \nabla_{\bar{\mathbf{v}}_{{\rm I}}}^{2}\varphi_{n}\left(\bar{\mathbf{v}}_{{\rm I}}\right)\mathbb{F}\left(\bar{\rho}_{{\rm I}},\bar{\u}_{{\rm I}}\right) \d x \d s\\
\ra & \int_{0}^t \int_{\O} \rho_{{\rm II}}\mathbb{F}\left(\rho_{{\rm II}},\u_{{\rm II}}\right) \nabla_{\u_{{\rm II}}}^{2}\varphi_{n}\left(\u_{{\rm II}}\right)\mathbb{F}\left(\rho_{{\rm II}},\u_{{\rm II}}\right) \d x  \d s \text{ as }K\ra \infty, \notag
\end{align}
$\bar{\mathbb{P}}$ a.s. according to Lebesgue's dominated convergence theorem. The terms
\begin{equation}
-\int_0^t \int_{\O} \nabla_{\bar{\mathbf{v}}_{{\rm I}}}^{2}\varphi_{n}(\bar{\mathbf{v}}_{{\rm I}})\bar{\rho}_{{\rm I}}^{2} \bar{\u}_{{\rm I}} \phi_{K}^{\prime}(\bar{\rho}_{{\rm I}}) \operatorname{div} \bar{\u}_{{\rm I}}\d x\d s,
\end{equation}
and
\begin{equation}
-\int_0^t \int_{\O} \nabla_{\bar{\mathbf{v}}_{{\rm I}}}^{2}\varphi_{n}(\bar{\mathbf{v}}_{{\rm I}})\bar{\rho}_{{\rm I}} \nabla \phi_{K}(\bar{\rho}_{{\rm I}}) \mathbb{D} \bar{\u}_{{\rm I}}\d x\d s,
\end{equation}
will go to zero due to $ \phi_{K}^{\prime}(\bar{\rho}_{{\rm I}})\bar{\rho}_{{\rm I}}^{\frac{1}{2}}\ls \frac{4}{K^{\frac{1}{2}}}$.
Lemma \ref{lemma sto int convergence} implies that
\begin{equation}
\int_{0}^{t} \int_{\O} \nabla_{\bar{\mathbf{v}}_{{\rm I}}}\varphi_{n}(\bar{\mathbf{v}}_{{\rm I}})\phi_{K}(\bar{\rho}_{{\rm I}})\bar{\rho}_{{\rm I}} \mathbb{F}(\bar{\rho}_{{\rm I}},\bar{\u}_{{\rm I}}) \d x\d \bar{W}_{{\rm I}}\ra \int_{0}^{t} \int_{\O}  \nabla_{\u_{{\rm II}}}\varphi_{n}(\u_{{\rm II}}) \rho_{{\rm II}} \mathbb{F}(\rho_{{\rm II}},\u_{{\rm II}}) \d x\d W_{{\rm II}}.
\end{equation}
 $\bar{\mathbb{P}}$ a.s. as $\kappa\ra 0$. Therefore, we have the following lemma.
 \begin{lemma}\label{Mellet-Vasseur mid process conclution}
 $\left(\rho_{{\rm II}}, \u_{{\rm II}}\right)$ satisfies
\begin{align}
&\int_{\O}\rho_{{\rm II}} \varphi_{n}(\u_{{\rm II}}) \d x= - \int_{0}^{t} \int_{\O} \nabla_{\u_{{\rm II}}}\varphi_{n}(\u_{{\rm II}}) \mathbf{R}_{{\rm II}} \d x \d s + \int_{\O}\rho_{0, {\rm II}} \varphi_{n}\left(\frac{\mathbf{q}_{0, {\rm II}}}{\rho_{0, {\rm II}}}\right) \d x\\
&+\int_0^t \int_{\O} \rho_{{\rm II}}^{\gamma} \nabla_{\u_{{\rm II}}}^{2}\varphi_{n}(\u_{{\rm II}}):\nabla \u_{{\rm II}}\d x\d s+ \int_0^t\int_{\O} \rho_{{\rm II}} \nabla_{\u_{{\rm II}}}^{2}\varphi_{n}(\u_{{\rm II}})\nabla\u_{{\rm II}}:\mathbb{D}\u_{{\rm II}}  \d x \d s\notag\\
&+\int_{0}^t\int_{\O}  \rho_{{\rm II}}\mathbb{F}(\rho_{{\rm II}},\u_{{\rm II}}) \nabla_{\u_{{\rm II}}}^{2}\varphi_{n}(\u_{{\rm II}}) \mathbb{F}(\rho_{{\rm II}},\u_{{\rm II}})  \d x\d s + \int_{0}^{t} \int_{\O} \nabla_{\u_{{\rm II}}}\varphi_{n}(\u_{{\rm II}})  \rho_{{\rm II}} \mathbb{F}(\rho_{{\rm II}},\u_{{\rm II}}) \d x\d W_{{\rm II}}, \notag
\end{align}
and the estimate
\begin{align}
&  \mathbb{E}\left[\left|  \int_{\Omega} \rho_{{\rm II}}  \varphi_{n} (\u_{{\rm II}} ) \d x  +\int_{0}^{t} \int_{\O} \nabla_{\u_{{\rm II}}}\varphi_{n}(\u_{{\rm II}}) \mathbf{R}_{{\rm II}} \d x \d s \right|^{r}\right]\notag \\
\ls & \bar{C} + C_{r}\mathbb{E}\left[\left|\int_{\Omega} \rho_{0, {\rm II}} \varphi_{n} \left(\frac{\mathbf{q}_{0, {\rm II}}}{\rho_{0, {\rm II}}}\right) \d x\right|^{r}\right]\notag\\
&+ C_{r} \mathbb{E}\left[\left|\int_{0}^t \int_{\O}\rho_{{\rm II}}\mathbb{F}(\rho_{{\rm II}},\u_{{\rm II}})\nabla_{\u_{{\rm II}}}^{2}\varphi_{n}\left(\u_{{\rm II}}\right)\mathbb{F}(\rho_{{\rm II}},\u_{{\rm II}}) \d x \d s\right|^{r}\right]\\
& + C_{r} \mathbb{E}\left[\left|\int_{0}^{t} \int_{\O}\left(1+\tilde{\varphi}_{n}^{\prime}\left(|\u_{{\rm II}} |^{2}\right)\right) \rho_{{\rm II}} ^{2 \gamma-1} \d x \d s\right|^{r}\right]\notag\\
& + C_{r}\mathbb{E}\left[\left| \int_{\O}\int_{0}^{t} \nabla_{\u_{{\rm II}}}\varphi_{n}(\u_{{\rm II}}) \rho_{{\rm II}}  \mathbb{F}(\rho_{{\rm II}} ,\u_{{\rm II}} )\d W_{{\rm II}} \d x\right|^{r}\right],\notag
\end{align}
where
\begin{equation}
\mathbf{R}_{{\rm II}}=\eta \nabla\rho_{{\rm II}}^{-10} + \delta \rho_{{\rm II}}\nabla \triangle^{9}\rho_{{\rm II}} +r_{0}|\u_{{\rm II}}|^{2} \u_{{\rm II}}  +r_{1} \rho_{{\rm II}}|\u_{{\rm II}}|^{2} \u_{{\rm II}} + r_{2}\u_{{\rm II}},
\end{equation}
$\bar{C} $ depends on $r, T$, and $\mathbb{E}\left[e_{{\rm I}}(0)^{r}\right]$, $C_{r}$ merely depends on $r$.
\end{lemma}
\begin{proposition} $\left(\left(\bar{\Omega},\bar{\mathcal{F}},\bar{\mathbb{P}} \right),\rho_{{\rm II}}, \u_{{\rm II}}, W_{{\rm II}}\right)$ is a martingale solution to
\begin{align}\label{system of kappa vanish}
\left\{\begin{array}{l}\vspace{1.2ex}
\left(\rho_{{\rm II}}\right)_{t}+\operatorname {div}(\rho_{{\rm II}} \u_{{\rm II}})= 0,\\
\d \left(\rho_{{\rm II}} \u_{{\rm II}}\right) + \left(\operatorname{div}(\rho_{{\rm II}} \u_{{\rm II}} \otimes \u_{{\rm II}})+ \nabla \left(a\rho_{{\rm II}}^{\gamma}\right)-\operatorname{div}(\rho_{{\rm II}}
\mathbb{D} \u_{{\rm II}})-\frac{11}{10}\eta \nabla \rho_{{\rm II}}^{-10} \right)\d t \\
= (-r_{0}|\u_{{\rm II}}|^{2}\u_{{\rm II}}-r_{1} \rho_{{\rm II}}|\u_{{\rm II}}|^{2} \u_{{\rm II}}-r_{2}\u_{{\rm II}}) \d t
+ \delta \rho_{{\rm II}}\nabla \triangle^{9}\rho_{{\rm II}}  \d t+ \rho_{{\rm II}} \mathbb{F}(\rho_{{\rm II}},\u_{{\rm II}}) \d W_{{\rm II}}.
\end{array}\right.
\end{align}
\end{proposition}
\begin{proof}
 The weak convergence of $\nabla\bar{\rho}_{{\rm I}}^{\frac{\gamma}{2}}$,
 $\bar{\rho}_{{\rm I}}^{\frac{1}{2}}\nabla^{2}\ln\bar{\rho}_{{\rm I}}$ still holds in the corresponding spaces. By Vitali's convergence theorem, the limits of $\bar{\rho}_{{\rm I}}$ and $\bar{\u}_{{\rm I}}$ satisfy the mass conservation equation. We just consider the convergence of stochastic forces.
We show that what is limit of stochastic integral $  \int_{0}^{T}\int_{\O}\bar{\rho}_{{\rm I}} \mathbb{F}(\bar{\rho}_{{\rm I}},\bar{\u}_{{\rm I}})\d \bar{W}_{{\rm I}}\d x$.

By \eqref{sto estimate in kappa1} and \eqref{sto estimate in kappa2}, we have
\begin{align}
 \int_0^t \int_{\O} \left|\bar{\rho}_{{\rm I}} \mathbb{F}(\bar{\rho}_{{\rm I}},\bar{\u}_{{\rm I}})\right|^{2}\d x\d t
 \ls C,
\end{align}
 where $C$ depends on $T,r,\mathbb{E}\left[e(0)^{r}\right]$ and $\mathbb{E}\left[\tilde{e}(0)^{r}\right]$.
In addition, we have derived the strong convergence of $\bar{\rho}_{{\rm I}}$ in $L^{2}\left(0,T; L^{\frac{3}{2}}\left(\O\right)\right)$ and the strong convergence of $\bar{\rho}_{{\rm I}}\bar{\u}_{{\rm I}}$ in $L^{2}\left(0,T; L^{\frac{3}{2}}\left(\O\right)\right)$, hence we know $\bar{\rho}_{{\rm I}}\ra \rho_{{\rm II}} $ almost everywhere, $\bar{\rho}_{{\rm I}}\bar{\u}_{{\rm I}}\ra \rho_{{\rm II}} \u_{{\rm II}} $ almost everywhere up to a subsequence. Therefore $\bar{\rho}_{{\rm I}} \mathbb{F}\left(\bar{\rho}_{{\rm I}},\bar{\u}_{{\rm I}}\right)\ra \rho_{{\rm II}}  \mathbb{F}\left(\rho_{{\rm II}} ,\u_{{\rm II}} \right)$ holds almost everywhere. By dominated convergence theorem, we have
\begin{equation}
\int_0^t \int_{\O} \left|\bar{\rho}_{{\rm I}} \mathbb{F}\left(\bar{\rho}_{{\rm I}},\bar{\u}_{{\rm I}}\right)\right|^{2}\d x\d t \ra \int_0^t \int_{\O} \left|\rho_{{\rm II}} \mathbb{F}\left(\rho_{{\rm II}},\u_{{\rm II}}\right)\right|^{2}\d x\d t, \quad \bar{\mathbb{P}} \text{ a.s.}.
\end{equation}
Then the convergence of the stochastic integral follows according to Lemma \ref{Conver of sto int for m layer}. \hfill$\square$
\end{proof}
 Moreover, letting
 \begin{align}\label{no kappa BD estimate1}
&e_{{\rm II}}(t)=\int_{\O}\left(\frac{1}{2}\rho_{{\rm II}}\left|\u_{{\rm II}}\right|^{2}+\frac{\eta}{10}\rho_{{\rm II}}^{-10}+\frac{\delta}{2}\left|\nabla\triangle^{4}\rho_{{\rm II}}\right|^{2}+\int_{1}^{\rho_{{\rm II}}}\frac{p(z)}{z}\d z\right)\d x, \\
&\tilde{e}_{{\rm II}}(t)=\int_{\O}\left(\frac{1}{2}\rho_{{\rm II}}\left|\u_{{\rm II}}+\nabla \ln\rho_{{\rm II}}\right|^{2}+\frac{\eta}{10}\rho_{{\rm II}}^{-10}+\frac{\delta}{2}\left|\nabla\triangle^{4}\rho_{{\rm II}}\right|^{2}+\int_{1}^{\rho_{{\rm II}}}\frac{p(z)}{z}\d z\right)\d x,
\end{align}
by Gr\"onwall's inequality, we have the following estimates,
\begin{align}\label{no kappa BD estimate1}
\mathbb{E}\left[\sup\limits_{t\in[0,T]}e_{{\rm II}}(t)^{r}\right] \ls C,\quad \mathbb{E}\left[\sup\limits_{t\in[0,T]}\tilde{e}_{{\rm II}}(t)^{r}\right] \ls C,
\end{align}
where $C$ depends on $r,T,\mathbb{E}\left[e_{{\rm II}}(0)^{r}\right]$, $\mathbb{E}\left[\tilde{e}_{{\rm II}}(0)^{r}\right]$, $r>4$.

\subsection{Vanishing the artificial pressure and Rayleigh damping: $ \delta, \eta, r_{0}\ra 0 $}
We take $n\ra +\infty$ to derive a new stochastic Mellet-Vasseur type inequality. Recall that in the process of the stochastic B-D entropy estimates,
estimate of the term $r_{0}\u_{{\rm II}}^{3}$ is dependent on $\delta$ and $\eta$, which will cause trouble if we firstly take $\delta,\eta \ra 0$. So we will take the limit $n\ra \infty, \delta,\eta,r_{0}\ra 0$ at the same time, then we do the energy estimates for the system, which shows the terms like $\mathbb{E}\left[\left(\int_{\O}\delta \mid \nabla\triangle^{4}\rho_{{\rm II}}\mid^{2}\d x\right)^{r}\right]$ in the energy estimates will vanish after we take the limit $\delta\ra 0$. So does the $\eta-$term in the stochastic B-D entropy.\\

\begin{flushleft}
\textbf{Step 1:} Choose the path space and show the tightness of the laws and the convergence.
\end{flushleft}

Take the space $  \mathcal{X}_{5}=\mathcal{X}_{\rho_{0}}^{5} \times \mathcal{X}_{\frac{\mathbf{q}_{0}}{\sqrt{\rho_{0}}}}^{5}
\times\mathcal{X}_{\mathbf{q}_{0}}^{5} \times\mathcal{X}_{\rho}^{5}\times \mathcal{X}_{\rho^{\frac{1}{2}}\u}^{5}\times \mathcal{X}_{W}^{5}$,
where $\mathcal{X}_{\rho_{0}}^{5}=L^{\gamma}\left(\O\right)$, $\mathcal{X}_{\frac{\mathbf{q}_{0}}{\sqrt{\rho_{0}}}}^{5}=L^{2}\left(\O\right)$,
$\mathcal{X}_{\mathbf{q}_{0}}^{5}=L^{1}\left(\O\right)$,
$\mathcal{X}_{\rho}^{5}=L^{2}\left(0,T; L^{\frac{3}{2}}\left(\O\right)\right)\cap L^{\frac{5}{3}\gamma}\left([0,T]\times\O\right)$,
 $\mathcal{X}_{\rho^{\frac{1}{2}}\u}^{5}=L^{2}\left([0,T]\times\O\right)$,
 $\mathcal{X}_{W}^{5}=C\left([0,T];\mathcal{H}\right)$. Then we have the corresponding tightness and convergence  theorem.
\begin{proposition}
There exists two families of $\mathcal{X}_{5}$-valued Borel measurable random variables
 $$\left\{\bar{\rho}_{0, {\rm II}},\frac{\bar{\mathbf{q}}_{0, {\rm II}}}{\sqrt{\bar{\rho}_{0, {\rm II}}}}, \bar{\mathbf{q}}_{0, {\rm II}}, \bar{\rho}_{{\rm II}},\bar{\rho}_{{\rm II}}^{\frac{1}{2}}\bar{\u}_{{\rm II}}, \bar{W}_{{\rm II}}\right\}$$
 and
  $\left\{\rho_{0,\rm{III}},\frac{\mathbf{q}_{0, {\rm III}}}{\sqrt{\rho_{0, {\rm III}}}}, \mathbf{q}_{0, {\rm III}}, \rho_{{\rm III}} , \rho_{{\rm III}}^{\frac{1}{2}}\u_{\rm{\tiny{III}}} , W_{{\rm III}} \right\}$,
  defined on a new complete probability space $\left(\bar{\Omega},\bar{\mathcal{F}},\bar{\mathbb{P}}\right)$, such that (up to a subsequence):
\begin{enumerate}
  \item
  $\mathcal{L}\left[\bar{\rho}_{0, {\rm II}},\frac{\bar{\mathbf{q}}_{0, {\rm II}}}{\sqrt{\bar{\rho}_{0, {\rm II}}}}, \bar{\mathbf{q}}_{0, {\rm II}}, \bar{\rho}_{{\rm II}}, \bar{\rho}_{{\rm II}}^{\frac{1}{2}}\bar{\u}_{{\rm II}}, \bar{W}_{{\rm II}} \right]$ and
   $\mathcal{L}\left[\rho_{0, {\rm II}},\frac{\mathbf{q}_{0, {\rm II}}}{\sqrt{\rho_{0, {\rm II}}}}, \mathbf{q}_{0, {\rm II}}, \rho_{{\rm II}}, \rho_{{\rm II}}^{\frac{1}{2}}\u_{{\rm II}}, W_{{\rm II}} \right]$ \\
   coincides with each other on $\mathcal{X}_{5}$;
  \item $\mathcal{L}\left[\rho_{0, {\rm III}},\frac{\mathbf{q}_{0, {\rm III}}}{\sqrt{\rho_{0, {\rm III}}}}, \mathbf{q}_{0, {\rm III}}, \rho_{{\rm III}} , \rho_{{\rm III}}^{\frac{1}{2}}\u_{{\rm III}} , W_{{\rm III}} \right]$ on $\mathcal{X}_{5}$ is a Radon measure;
  \item Random variables $\left\{\bar{\rho}_{0, {\rm II}},\frac{\bar{\mathbf{q}}_{0, {\rm II}}}{\sqrt{\bar{\rho}_{0, {\rm II}}}}, \bar{\mathbf{q}}_{0, {\rm II}}, \bar{\rho}_{{\rm II}},\bar{\rho}_{{\rm II}}^{\frac{1}{2}}\bar{\u}_{{\rm II}}, \bar{W}_{{\rm II}}\right\}$
   converges to\\
    $\left\{\rho_{0, {\rm III}},\frac{\mathbf{q}_{0, {\rm III}}}{\sqrt{\rho}_{0, {\rm III}}}, \mathbf{q}_{0, {\rm III}}, \rho_{{\rm III}}, \rho_{{\rm III}}^{\frac{1}{2}}\u_{{\rm III}}, W_{{\rm III}} \right\}$ in the topology of $\mathcal{X}_{5}$, $\bar{\mathbb{P}}$ {\rm a.s.} as $n\ra \infty$. \\
\end{enumerate}

\end{proposition}
\begin{flushleft}
\textbf{Step 2:} The estimates for the limit system.
\end{flushleft}

Recall \eqref{BD estimate2} and \eqref{BD estimate1},
 by Gagliardo-Nirenberg's interpolation inequality, it holds
 $\delta^{\frac{9}{19}}\left|\nabla\triangle^{4} \bar{\rho}_{\delta}\right| \in L^{\frac{19}{9}}\left(0,T; L^{3}(\O)\right)$ \cite{Vasseur-Yu-q2016}.
 Therefore, for any $\iota>0$, Chebyshev's inequality shows this highest order derivative of $\rho$ term goes to $0$ $\bar{\mathbb{P}}$ almost everywhere. For all $\varphi(t)\in C_{c}^{\infty}\left([0,T)\right)$, and all $\psi(x)\in C^{\infty}\left(\O\right)$, we have
\begin{align}
&\mathbb{P}\left[\left\{\left|\int_{0}^{t}\int_{\O}\delta \bar{\rho}_{{\rm II}}\nabla\triangle^{9}\bar{\rho}_{{\rm II}} \varphi(t)\psi(x) \d x\d s\right|>\iota\right\}\right]\\
\ls &C \frac{\mathbb{E}\left[\left|\int_{0}^{t}\int_{\O}\delta \triangle^{5}\bar{\rho}_{{\rm II}}\nabla\triangle^{4}\left(\bar{\rho}_{{\rm II}}
                                                                                    \varphi(t)\psi(x)\right) \d x\d s \right|^{r}\right]}{\iota^{r}}. \notag
\end{align}
Taking the highest order derivative of $\bar{\rho}_{{\rm II}}$ for example, since the lower order terms are controlled by this term, we have
\begin{align}
&\frac{\mathbb{E}\left[\left|\int_{0}^{t}\int_{\O}\delta \triangle^{5}\bar{\rho}_{{\rm II}}\nabla\triangle^{4}\left(\bar{\rho}_{{\rm II}}
                                                                                    \right) \varphi(t)\psi(x)\d x\d s \right|^{r}\right]}{\iota^{r}}\notag\\
\ls&\frac{\mathbb{E}\left[\left|C(\varphi,\psi) \int_{0}^{t} \int_{\Omega} \sqrt{\delta}\left|\triangle^{5} \bar{\rho}_{{\rm II}}\right|
            \delta^{\frac{9}{19}}\left|\nabla\triangle^{4} \bar{\rho}_{{\rm II}}\right| \delta^{\frac{1}{38}} \d x \d s \right|^{r}\right]}{\iota^{r}}\\
\ls &\frac{C(\varphi,\psi) \delta^{\frac{1}{38}}\left( \mathbb{E}\left[\sqrt{\delta}\left\|\triangle^{5}\bar{\rho}_{{\rm II}}\right\|_{L_{t}^{2} L_{x}^{2}}\right]^{2r}\right)^{\frac{1}{2}}
               \left( \mathbb{E}\left[ \left\|\delta^{\frac{9}{19}} \nabla\triangle^{4}\bar{\rho}_{{\rm II}}\right\|_{L_{t}^{\frac{19}{9}} L_{x}^{3}}\right]^{2r}\right)^{\frac{1}{2}}}{\iota^{r}}\notag\\
&\ra 0 \quad\text{ as }\delta \ra 0.\notag
\end{align}
In order to handle
\begin{align}
&\mathbb{E}\left[\left|\int_0^t\int_{\O} \delta \bar{\rho}_{{\rm II}}\nabla\triangle^{9}\bar{\rho}_{{\rm II}} \cdot \nabla_{\bar{\u}_{{\rm II}}}\varphi_{n} \left(\bar{\u}_{{\rm II}}\right) \d x \d s\right|^{r}\right]\notag\\
=& \mathbb{E}\left[\left|\int_0^t\int_{\O} \delta \bar{\rho}_{{\rm II}}\sum\limits_{i=1}^{3}\frac{\partial}{\partial x_{i}}\triangle^{9}\bar{\rho}_{{\rm II}} \left(1+\ln\left(1+|\bar{\u}_{{\rm II}}|^{2}\right)\right)2 \bar{u}_{{\rm II},i}\d x \d s \right|^{r}\right]\\
=& 3^{r-1}\sum\limits_{i=1}^{3}\mathbb{E}\left[\left|\int_0^t\int_{\O} \delta \bar{\rho}_{{\rm II}}\frac{\partial}{\partial x_{i}}\triangle^{9}\bar{\rho}_{{\rm II}} \left(1+\ln\left(1+|\bar{\u}_{{\rm II}}|^{2}\right)\right)2\bar{u}_{{\rm II}, i}\d x \d s \right|^{r}\right],\notag
\end{align}
we calculate
\begin{align}
    &\mathbb{E}\left[\left|\int_0^t \int_{\O}\delta \bar{\rho}_{{\rm II}}\frac{\partial}{\partial x_{1}}\triangle^{9}\bar{\rho}_{{\rm II}}\left(1+\ln\left(1+|\bar{\u}_{{\rm II}}|^{2}\right)\right)2\bar{u}_{{\rm II},1}\d x \d s \right|^{r}\right]\notag\\
\ls & C_{n,r}\mathbb{E}\left[\left|\int_0^t \int_{\O}\delta \bar{\rho}_{{\rm II}}\left|\frac{\partial}{\partial x_{1}}\triangle^{9}\bar{\rho}_{{\rm II}}\right|\d x \d s\right|^{r}\right]\notag\\
 =& C_{n,r}\mathbb{E}\left[\left|\int_0^t \int_{\O}\left(I_{\{\frac{\partial}{\partial x_{1}}\triangle^{9}\bar{\rho}_{{\rm II}}>0\}}-I_{\{\frac{\partial}{\partial x_{1}}\triangle^{9}\bar{\rho}_{{\rm II}}\ls 0\}}\right)\delta\bar{\rho}_{{\rm II}}\frac{\partial}{\partial x_{1}}\triangle^{9}\bar{\rho}_{{\rm II}}\d x \d s\right|^{r}\right] \\
 =& C_{n,r}\mathbb{E}\left[\left|\int_0^t \int_{\O}I_{\{\frac{\partial}{\partial x_{1}}\triangle^{9}\bar{\rho}_{{\rm II}}>0\}}\delta\triangle^{5}\bar{\rho}_{{\rm II}}\frac{\partial}{\partial x_{1}}\triangle^{4}\bar{\rho}_{{\rm II}}\d x \d s\right.\right.\notag\\
   &\quad \quad  \quad\left.\left.-\int_0^t \int_{\O} I_{\{\frac{\partial}{\partial x_{1}}\triangle^{9}\bar{\rho}_{{\rm II}}\ls 0\}}\delta\triangle^{5}\bar{\rho}_{{\rm II}}\frac{\partial}{\partial x_{1}}\triangle^{4}\bar{\rho}_{{\rm II}}\d x \d s\right|^{r}\right]\notag\\
\ls & C_{n,r}\mathbb{E}\left[\left|\int_0^t \int_{\O}\delta\left|\triangle^{5}\bar{\rho}_{{\rm II}}\right|\left|\frac{\partial}{\partial x_{1}}\triangle^{4}\bar{\rho}_{{\rm II}}\right|\d x \d s\right|^{r}\right],\notag
\end{align}
where $C_{n,r}=\left(2(e(1+n)^{2}-1)\left(1+\ln n\right)+1\right)^{r}$. So we have
\begin{align}
    &\mathbb{E}\left[\left|\int_0^t\int_{\O} \delta \bar{\rho}_{{\rm II}}\nabla\triangle^{9}\bar{\rho}_{{\rm II}} \cdot \nabla_{\bar{\u}_{{\rm II}}}\varphi_{n} \left(\bar{\u}_{{\rm II}}\right) \d x \d s\right|^{r}\right]\notag\\
\ls &3^{r}C\mathbb{E}\left[\left(\int_0^t \int_{\O} \delta\left|\triangle^{5}\bar{\rho}_{{\rm II}}\right| \left|\nabla\triangle^{4}\bar{\rho}_{{\rm II}}\right| \d x \d s\right)^{r}\right]\notag\\
= &3^{r}C\mathbb{E}\left[\left( \delta^{\frac{1}{38}}\int_0^t\int_{\O} \delta^{\frac{1}{2}} |\triangle^{5}\bar{\rho}_{{\rm II}}| \delta^{\frac{9}{19}} |\nabla\triangle^{4}\bar{\rho}_{{\rm II}}|\d x \d s\right)^{r}\right] \\
\ls &3^{r}C\left( \delta^{\frac{1}{38}} \right)^{r}\left( \mathbb{E}\left[\delta^{\frac{1}{2}}\left\|\triangle^{5}\bar{\rho}_{{\rm II}}\right\|_{L_{t}^{2} L_{x}^{2}}\right]^{2r}\right)^{\frac{1}{2}}
               \left( \mathbb{E}\left[ \left\|\delta^{\frac{9}{19}} \nabla\triangle^{4}\bar{\rho}_{{\rm II}}\right\|_{L_{t}^{\frac{19}{9}} L_{x}^{3}}\right]^{2r}\right)^{\frac{1}{2}}\notag\\
\ls &3^{r}C\left( \delta^{\frac{1}{38}} \right)^{r},\notag
\end{align}
 for any $\delta<n^{-\alpha}$, $\alpha> 76$, where $C$ depends on $r,T,\mathbb{E}\left[(e_{{\rm II}}(0))^{r}\right]$ and $\mathbb{E}\left[(\tilde{e}_{{\rm II}}(0))^{r}\right]$. \\
 $\mathbb{E}\left[\left(\left|\int_0^T\int_{\O} \delta \bar{\rho}_{{\rm II}}\nabla\triangle^{9}\bar{\rho}_{{\rm II}} \nabla_{\bar{\u}_{{\rm II}}}\varphi_{n} \left(\bar{\u}_{{\rm II}}\right) \d x \d t\right|\right)^{r}\right]\ra 0$ as $n\ra \infty$.

Then, we consider the limit $\eta\ra 0$. We claim that $\eta\int_{\O}\bar{\rho}_{{\rm II}}^{-10}\d x \ra 0$ as $\eta\ra 0$. In fact,
\begin{equation}
\left|\left\{x \mid \bar{\rho}_{{\rm II}}(t, x)=0\right\}\right|=0 \quad \text{ for almost every } t,x,\omega.
\end{equation}
By Aubin-Lion's lemma, the strong convergence of $\bar{\rho}_{{\rm II}}$ in $L^{2}(0,T; L^{1}\left( \O \right))$ still holds, so $\bar{\rho}_{{\rm II}}^{-10}\ra \rho_{{\rm III}}^{-10}$ holds almost everywhere. So the $\eta\int_{\O}\bar{\rho}_{{\rm II}}^{-10}\d x$ in energy will vanish as $\eta\ra 0$. Next, we deal with $\mathbb{E}\left[\left(\left|\int_0^t\int_{\O} \eta \nabla\bar{\rho}_{{\rm II}}^{-10}\nabla_{\bar{\u}_{{\rm II}}}\varphi_{n} \left(\bar{\u}_{{\rm II}}\right) \d x \d s\right|\right)^{r}\right]$. Since
$
 \mathbb{E}\left[\left(\int_{\O}\eta \frac{\bar{\rho}_{{\rm II}}^{-10}}{10} \d x\right)^{r}\right]\ls C,
$
$
 \mathbb{E}\left[\left(\int_{0}^{T}\int_{\O}\eta \frac{11}{25} \left|\nabla \bar{\rho}_{{\rm II}}^{-5}\right|^{2}\d x\d s\right)^{r}\right]\ls C,
$
calculating
\begin{equation}
|\nabla \bar{\rho}_{{\rm II}}^{-10}|= 2\left|\nabla \bar{\rho}_{{\rm II}} \right|\bar{\rho}_{{\rm II}}^{-11}= 2\left|\nabla \bar{\rho}_{{\rm II}}^{\frac{1}{2}}\right|\bar{\rho}_{{\rm II}}^{-10-\frac{1}{2}}.
\end{equation}
by H\"older's inequality, we have
\begin{align}
 \eta^{\frac{11}{20}}\left\|\bar{\rho}_{{\rm II}}^{-10-\frac{1}{2}}\right\|_{L_{t}^{1}L_{x}^{\frac{60}{23}}}
\ls \eta^{\frac{1}{2}}\left\|\bar{\rho}_{{\rm II}}^{-10}\right\|_{L_{t}^{1}L_{x}^{3}} \eta^{\frac{1}{20}}\left\|\bar{\rho}_{{\rm II}}^{-\frac{1}{2}}\right\|_{L_{t}^{\infty}L_{x}^{20}}\ls C,
\end{align}
where $C$ depends on $r,T,\mathbb{E}\left[(e_{{\rm II}}(0))^{r}\right]$ and $\mathbb{E}\left[(\tilde{e}_{{\rm II}}(0))^{r}\right]$.
So by H\"older's inequality, it holds
\begin{align}
 &\eta^{\frac{11}{20}} \int_0^t\left|\int_{\O}\nabla\bar{\rho}_{{\rm II}}^{-10}\d x\right|^{\frac{60}{23}}\d s
=\eta^{\frac{11}{20}} \int_0^t \left|\int_{\O}\bar{\rho}_{{\rm II}}^{-10-\frac{1}{2}}\left|\nabla \bar{\rho}_{{\rm II}}^{\frac{1}{2}}\right|\d x\right|^{\frac{60}{23}}\d s \notag\\
\ls  &C\sup\limits_{s\in[0,T]}\left(\int_{\O}\left|\nabla \bar{\rho}_{{\rm II}}^{\frac{1}{2}}\right|^{2}\d x\right)^{\frac{1}{2}} \int_0^t \eta^{\frac{11}{20}}\left(\int_{\O}\bar{\rho}_{{\rm II}}^{(-10-\frac{1}{2})\frac{60}{23}}\d x\right)^{\frac{23}{60}\frac{60}{23}}\d s\\
\ls &C \sup\limits_{s\in[0,T]}\left(\int_{\O}\left|\nabla \bar{\rho}_{{\rm II}}^{\frac{1}{2}}\right|^{2}\d x\right)^{\frac{1}{2}}\eta^{\frac{1}{2}}\left\|\bar{\rho}_{{\rm II}}^{-10}\right\|_{L_{t}^{1}L_{x}^{3}} \eta^{\frac{1}{20}}\left\|\bar{\rho}_{{\rm II}}^{-\frac{1}{2}}\right\|_{L_{t}^{\infty}L_{x}^{20}},\notag
\end{align}
where $C$ depends on the length of the period of the area.
So we conclude that $\eta^{\frac{253}{1200}}\nabla\bar{\rho}_{{\rm II}}^{-10}\in L^{\frac{60}{23}}\left(0,T;L^{1}\left( \O \right)\right)$. Next, we estimate
\begin{align}
    &\mathbb{E}\left[\left|\int_0^t\int_{\O} \eta \nabla\bar{\rho}_{{\rm II}}^{-10} \nabla_{\bar{\u}_{{\rm II}}}\varphi_{n}\left(\bar{\u}_{{\rm II}}\right) \d x \d s\right|^{r}\right]\notag\\
\ls &\mathbb{E}\left[\left(\left|\int_0^t\int_{\O} \eta |\nabla\bar{\rho}_{{\rm II}}^{-10}|  \left(1+\ln\left(1+|\bar{\u}_{{\rm II}}|^{2}\right)\right)2|\bar{\u}_{{\rm II}}| \d x \d s\right|\right)^{r}\right]\notag\\
\ls &\mathbb{E}\left[\left(2(e(1+n)^{2}-1)^{\frac{1}{2}}\left(1+\ln(1+n)\right) \int_0^t \int_{\O} \eta \left|\nabla\bar{\rho}_{{\rm II}}^{-10}\right| \d x \d s\right)^{r}\right]\\
\ls &\mathbb{E}\left[\left(2(e(1+n)^{2}-1)^{\frac{1}{2}}\left(1+\ln(1+n)\right)\eta^{\frac{947}{1200}} \int_0^t \int_{\O} \eta^{\frac{253}{1200}}\left|\nabla\bar{\rho}_{{\rm II}}^{-10}\right|\d x\d s\right)^{r}\right]\notag \\
\ls & C\mathbb{E}\left[\left(2(e(1+n)^{2}-1)^{\frac{1}{2}}\left(1+\ln(1+n)\right)\eta^{\frac{947}{1200}}
\left\|\eta^{\frac{253}{1200}}\nabla\bar{\rho}_{{\rm II}}^{-10}\right\|_{L_{t}^{\frac{60}{23}}L_{x}^{1}}\right)^{r}\right],\notag
\end{align}
 for any $\eta<n^{-\beta}$, $\beta> \frac{2400}{947}$, $\mathbb{E}\left[\left(\left|\int_0^T\int_{\O} \eta \nabla\bar{\rho}_{{\rm II}}^{-10}\nabla_{\bar{\u}_{{\rm II}}}\varphi_{n}\left(\bar{\u}_{{\rm II}}\right)\d x \d t\right|\right)^{r}\right]\ra 0$ as $n\ra \infty$.
For all $\varphi(t)\in C_{c}^{\infty}\left([0,T)\right)$, and all $\psi(x)\in C^{\infty}\left(\O\right)$, for any $\iota>0$, we obtain
\begin{align}
&\quad \mathbb{P}\left[\left\{\left|\int_{0}^{t}\int_{\O}\eta \nabla \bar{\rho}_{{\rm II}}^{-10}\varphi(t)\psi(x) \d x\d s \right|>\iota\right\}\right]\notag\\
&\ls  \frac{\mathbb{E}\left[\left|\int_{0}^{t}\varphi(t)\int_{\O}\eta \bar{\rho}_{{\rm II}}^{-10} \nabla\psi(x) \d x\d s \right|^{r}\right]}{\iota^{r}}
  \ls  \frac{\mathbb{E}\left[\left|\int_{0}^{t}\varphi(t)\int_{\O}\eta \bar{\rho}_{{\rm II}}^{-10} \nabla\psi(x) \d x\d s \right|^{r}\right]}{\iota^{r}}\\
&\ls  \eta^{\frac{r}{2}}\frac{\mathbb{E}\left[\left(\left\|\eta^{\frac{1}{2}}\bar{\rho}_{{\rm II}}^{-10}\right\|_{L_t^{2}L_x^{3}}\left\|\varphi(t)\right\|_{L_t^{2}}\left\|\psi(x)\right\|_{H_x^{10}}\right)^{r}\right]}{\iota^{r}}
\ls  C \eta^{\frac{r}{2}} \ra 0\quad\text{as }\eta\ra 0,\notag
\end{align}
where $C$ depends on $r,T,\mathbb{E}\left[(e_{{\rm II}}(0))^{r}\right]$ and $\mathbb{E}\left[(\tilde{e}_{{\rm II}}(0))^{r}\right]$.

From (\ref{lower bound of rho}), we have
\begin{align}
 \mathbb{E}\left[\left\| \bar{\rho}_{{\rm II}}^{-1} \right\|^r_{L_t^{\infty}L_x^{\infty}}\right]
& \ls C_{r}\eta^{-\frac{3}{10}r}\left(\mathbb{E}\left[\eta^{r}\left\|\bar{\rho}_{{\rm II}}^{-1}\right\|_{L_t^{\infty}L_x^{10}}^{10r}\right]\right)^{\frac{3}{10}}
       \delta^{-r}\left(\mathbb{E}\left[\left(\delta^{r}\left\| \bar{\rho}_{{\rm II}}\right\|_{L_t^{\infty}H_x^{9}}^{2r}\right)^{\frac{10}{7}}\right]\right)^{\frac{7}{10}}\\
&\quad +C_{r}\eta^{-r}\mathbb{E}\left[\eta^{r}\left\|\bar{\rho}_{{\rm II}}^{-1}\right\|_{L_t^{\infty}L_x^{10}}^{10r}\right]^{\frac{1}{5}}\delta^{-\frac{1}{2}r}\mathbb{E}\left[\left(\delta^{\frac{1}{2}r}\left\| \bar{\rho}_{{\rm II}}\right\|_{L_t^{\infty}H_x^{9}}^{r}\right)^{\frac{5}{4}}\right]^{\frac{4}{5}}.\notag
\end{align}
We estimate the added damping term as follows
\begin{align}
&-r_{0}\int_{0}^{t}\int_{\O}|\bar{\u}_{{\rm II}}|^{2}\bar{\u}_{{\rm II}}\cdot\nabla\ln\bar{\rho}_{{\rm II}} \d x\d s
\ls  r_{0}\int_{0}^{t}\int_{\O}\bar{\rho}_{{\rm II}}^{\frac{3}{4}}|\bar{\u}_{{\rm II}}|^{2}\bar{\u}_{{\rm II}} |\frac{\nabla\bar{\rho}_{{\rm II}}}{\bar{\rho}_{{\rm II}}\bar{\rho}_{{\rm II}}^{\frac{3}{4}}}|\d x\d s \\
\ls & r_{0}\left\|\bar{\rho}_{{\rm II}}^{-1}\right\|_{L_{t}^{\infty}L_{x}^{\infty}}^{\frac{7}{4}} \left\||\bar{\rho}_{{\rm II}}^{\frac{3}{4}}\bar{\u}_{{\rm II}}|^{2}\bar{\u}_{{\rm II}}\right\|_{L_{t}^{\frac{4}{3}}L_{x}^{\frac{4}{3}}}\left\|\nabla\triangle^{4}\bar{\rho}_{{\rm II}}\right\|_{L_{t}^{\infty}L_{x}^{2}}.\notag
\end{align}
So for our setting, $\eta=n^{-\alpha}$, $\delta=n^{-\beta}$, roughly we choose $r_{0}=n^{-\left(1+\frac{7}{4}\left(\frac{13}{10}\alpha+\frac{3}{2}\beta\right)+\frac{\beta}{2}\right)}$, then
\begin{align}
&-\mathbb{E}\left[\left|r_{0}\int_{0}^{t}\int_{\O}|\bar{\u}_{{\rm II}}|^{2}\bar{\u}_{{\rm II}}\cdot \nabla\ln\bar{\rho}_{{\rm II}} \d x\d s\right|^{r}\right]
\ls  r_{0}^{r}\left(\eta^{-\frac{3}{10}r}\delta^{-r}+\eta^{-r}\delta^{-\frac{1}{2}r}\right)^{\frac{7}{4}}\delta^{-\frac{r}{2}} C
\ra 0 \text{ as } n \ra \infty, \notag
\end{align}
where $C$ depends on $r,T,\mathbb{E}\left[(e_{{\rm II}}(0))^{r}\right]$ and $\mathbb{E}\left[(\tilde{e}_{{\rm II}}(0))^{r}\right]$.
Therefore we take limit $r_{0}\ra 0$, i.e., $n\ra \infty$, in the Stochastic B-D entropy estimates.

For passing to the limit $n\ra \infty$, we still have
\begin{align}
   &\mathbb{E}\left[\left(\left|\int_0^t\int_{\O}\bar{\rho}_{{\rm II}} \mathbf{F}_{k}\nabla_{\bar{\u}_{{\rm II}} }\varphi_{n}\left(\bar{\u}_{{\rm II}} \right)\d x\d \bar{W}_{{\rm II}} \right|\right)^{r}\right]\notag\\
\ls&\mathbb{E}\left[\left(\left|\int_0^t\left(\int_{\O}\bar{\rho}_{{\rm II}} \mathbf{F}_{k}\nabla_{\bar{\u}_{{\rm II}} }\varphi_{n}\left(\bar{\u}_{{\rm II}} \right)\d x\right)^{2}\d s\right|\right)^{\frac{r}{2}}\right]\notag\\
\ls&\mathbb{E}\left[\left(f_{k}^{2}\int_0^t\left(\int_{\O}\bar{\rho}_{{\rm II}} (1+\ln (1+|\bar{\u}_{{\rm II}} |^{2}))2|\bar{\u}_{{\rm II}} |\d x\right)^{2}\d s \right)^{\frac{r}{2}}\right]\\
\ls&\mathbb{E}\left[\left(f_{k}^{2}\int_0^t\left(\int_{\O}\bar{\rho}_{{\rm II}} |\bar{\u}_{{\rm II}} |^{2}\d x \int_{\O}\bar{\rho}_{{\rm II}} (1+\ln (1+|\bar{\u}_{{\rm II}} |^{2}))^{2}\d x\right) \d s\right)^{\frac{r}{2}}\right]\notag\\
\ls&\mathbb{E}\left[\left(f_{k}^{2}\left\| \bar{\rho}_{{\rm II}} ^{\frac{1}{2}}\bar{\u}_{{\rm II}} \right\|_{L_s^{\infty}L_x^{2}}^{2}\left(t\left\|\bar{\rho}_{{\rm II}} \right\|_{L_t^{\infty}L_x^{2}}+\left\|\left(\bar{\rho}_{{\rm II}} \right)^{\frac{1}{4}}\bar{\u}_{{\rm II}} \right\|_{L_t^{4}L_x^{4}}\right)\right)^{\frac{r}{2}}\right]\ls C,\notag
\end{align}
where $C$ depends on $r, T, \mathbb{E}\left[(e_{{\rm II}}(0))^{r}\right]$ and $\mathbb{E}\left[(\tilde{e}_{{\rm II}}(0))^{r}\right]$. 
When $n \ra +\infty $, for the stochastic part, since
\begin{equation}
\nabla_{\bar{\u}_{{\rm II}}}^{2}\varphi_{n}\left(\bar{\u}_{{\rm II}}\right)\ra \left(\left(2+\ln\left(1+|\u_{{\rm III}}|^{2}\right)\right)\mathbb{I}_{3}+4\frac{\rho_{{\rm III}}\u_{{\rm III}}\otimes\rho_{{\rm III}}\u_{{\rm III}}}{|\rho_{{\rm III}}\u_{{\rm III}}|^{2}}\right),
\end{equation}
we gain
\begin{align}
&\mathbb{E}\left[\left|\int_{\O}\int_{0}^t \bar{\rho}_{{\rm II}}\mathbb{F}(\bar{\rho}_{{\rm II}},\bar{\u}_{{\rm II}})\nabla_{\bar{\u}_{{\rm II}}}^{2}\varphi_{n}\left(\bar{\u}_{{\rm II}}\right)\mathbb{F}(\bar{\rho}_{{\rm II}},\bar{\u}_{{\rm II}}) \d s \d x\right|^{r}\right]\notag\\
\ls &\sum\limits_{k=1}^{+\infty} f_{k}^{2}\left(\mathbb{E}\left[\left|\int_{\O}\int_{0}^t\bar{\rho}_{{\rm II}}|\bar{\u}_{{\rm II}}|^{2}\d s \d x\right|^{r}\right]+\mathbb{E}\left[\left|\int_{\O}\int_{0}^t\bar{\rho}_{{\rm II}}\d s \d x\right|^{r}\right]\right)\ls C,
\end{align}
where $C$ depends on $r,T,\mathbb{E}\left[(e_{{\rm II}}(0))^{r}\right]$ and $\mathbb{E}\left[(\tilde{e}_{{\rm II}}(0))^{r}\right]$. Hence we prove the following lemma.
\begin{lemma}\label{Mellet-Vasseur lemma}
(Mellet-Vasseur's inequality in the stochastic version) For $\varpi\in (0,2)$, where $\varpi$ is small, for some $r> 4$, it holds
\begin{align}
     &\mathbb{E}\left[\left(\int_{\O} \rho_{{\rm III}} \left(1+|\u_{{\rm III}} |^{2}\right) \ln \left(1+|\u_{{\rm III}} |^{2}\right) \d x \right)^{r}\right]\notag\\
\ls & \bar{C}+C_{r}\mathbb{E}\left[\left( \int_{\O} \rho_{0,{\rm III}}\left(1+\frac{\left|\mathbf{q}_{0,{\rm III}}\right|^{2}}{\rho_{0,{\rm III}}^{2}}\right) \ln \left(1+\frac{\left|\mathbf{q}_{0,{\rm III}}\right|^{2}}{\rho_{0,{\rm III}}^{2}}\right) \d x \right)^{r}\right]\\
&+\mathbb{E}\left[\left(8\left(\int_{\O}\left(\frac{\left|\mathbf{q}_{0,{\rm III}}\right|^{2}}{\rho_{0,{\rm III}}}+\frac{\rho_{0,{\rm III}}^{\gamma}}{\gamma-1}+\left|\nabla \sqrt{\rho_{0,{\rm III}}}\right|^{2}\right) \d x+2 e_{0,{\rm III}}\right) \right)^{r}\right]\notag\\
     &+\mathbb{E}\left[\left( C \int_{0}^{t}\left(\int_{\O}\left(\rho_{{\rm III}}^{2 \gamma-1-\frac{\varpi}{2}}\right)^{\frac{2}{2-\varpi}}\right)^{\frac{2-\varpi}{2}}\left(\int_{\O} \rho_{{\rm III}} \left(2+\ln \left(1+|\u_{{\rm III}} |^{2}\right)\right)^{\frac{2}{\varpi}} \d x\right)^{\frac{\varpi}{2}} \d s \right)^{r}\right],\notag
\end{align}
and there exists a constant $C$, such that
 $$\left(\int_{\O} \rho_{{\rm III}}\left(2+\ln \left(1+|\u_{{\rm III}} |^{2}\right)\right)^{\frac{2}{\varpi}} \d x\right)^{\frac{\varpi}{2}}\ls C \int_{\O}\rho_{{\rm III}} \u_{{\rm III}} ^{2}\d x,$$
$\bar{C}$ and $C_{r}$ share the same meaning as in lemma \ref{Mellet-Vasseur mid process conclution}.
\end{lemma}
\begin{proposition} $\left(\left(\bar{\Omega},\bar{\mathcal{F}},\bar{\mathbb{P}} \right),\rho_{{\rm III}}, \u_{{\rm III}}, W_{{\rm III}}\right)$ is a martingale solution to
\begin{equation}\label{system of kappa n vanished}
\left\{\begin{array}{l}
\left(\rho_{{\rm III}}\right)_{t}+\operatorname {div}(\rho_{{\rm III}} \u_{{\rm III}})= 0, \\
\d  \left(\rho_{{\rm III}}\u_{{\rm III}}\right) + \left(\operatorname{div}(\rho_{{\rm III}} \u_{{\rm III}} \otimes \u_{{\rm III}})+ \nabla \left(a\rho_{{\rm III}}^{\gamma}\right)-\operatorname{div}(\rho_{{\rm III}}
\mathbb{D} \u_{{\rm III}})\right)\d t\\
= (-r_{1} \rho_{{\rm III}}|\u_{{\rm III}}|^{2} \u_{{\rm III}}-r_{2}\u_{{\rm III}}) \d t + \rho_{{\rm III}} \mathbb{F}(\rho_{{\rm III}},\u_{{\rm III}}) \d W_{{\rm III}}.
\end{array}\right.
\end{equation}
\end{proposition}
\begin{proof}
 The weak convergence of  $\nabla\bar{\rho}_{{\rm II}}^{\frac{\gamma}{2}}$, still holds in the corresponding spaces. By Vitali's convergence theorem, $\rho_{{\rm III}}$ and $\u_{{\rm III}}$ satisfy the mass conservation equation, it suffices to see whether the convergence of the term associated with stochastic forces holds or not in the momentum equation.
Since we still have the strong convergence of $\bar{\rho}_{{\rm II}}$ in $L^{2}\left(0,T; L^{\frac{3}{2}}\left(\O\right)\right)$ and the strong convergence of $\bar{\rho}_{{\rm II}}^{\frac{1}{2}}\bar{\u}_{{\rm II}}$ in $L^{2}\left([0,T]\times\O\right)$, hence $\bar{\rho}_{{\rm II}}\ra\rho_{{\rm III}} $ holds almost everywhere, and $ \bar{\rho}_{{\rm II}}\bar{\u}_{{\rm II}}\ra \rho_{{\rm III}} \u_{{\rm III}} $ holds almost everywhere up to a subsequence. Therefore $ \bar{\rho}_{{\rm II}} \mathbb{F}(\bar{\rho}_{{\rm II}},\bar{\u}_{{\rm II}}) \ra\rho_{{\rm III}}  \mathbb{F}\left(\rho_{{\rm III}},\u_{{\rm III}}\right)$ almost everywhere. By dominated convergence theorem, we have
\begin{equation}
\int_0^t \int_{\O} |\bar{\rho}_{{\rm II}} \mathbb{F}(\bar{\rho}_{{\rm II}},\bar{\u}_{{\rm II}})|^{2}\d x\d t \ra \int_0^t \int_{\O} |\rho_{{\rm III}} \mathbb{F}(\rho_{{\rm III}},\u_{{\rm III}})|^{2}\d x\d t,
\end{equation}
 $\bar{\mathbb{P}}$ a.s. \hfill$\square$
\end{proof}
 We define the energy $e(t)=\int_{\O}\left(\frac{1}{2}\rho_{{\rm III}}|\u_{{\rm III}} |^{2}+\frac{a}{\gamma}\rho_{{\rm III}} ^{\gamma} \right)\d x$,
 then the following lemma holds.
\begin{lemma}
For system \eqref{system of kappa n vanished}, it holds
\begin{enumerate}
  \item for some $r> \max\left\{6\gamma, \frac{4\gamma}{5\gamma-6}\right\}$,
   \begin{equation}\label{energy estimate1 kappa n vanished}
\mathbb{E}\left[\sup\limits_{t\in[0,T]}e_{{\rm III}}(t)^{r}\right] \ls  C\left(t^{\frac{r}{2}}+t^{r}+\mathbb{E}\left[e_{{\rm III}}(0)^{r}\right]\right)e^{C\left(t^{\frac{r}{2}}+t^{r}\right)}
\ls C\left(\mathbb{E}\left[e_{{\rm III}}(0)^{r}\right]+1\right),
\end{equation}
where $C$ depends on $r$ $T$, and $\sum\limits_{k=1}^{+\infty}f^{2}_{k}$. 
\item Let $\tilde{e}_{{\rm III}}=\int_{\O}\left(\frac{1}{2}\rho_{{\rm III}}|\u_{{\rm III}}+\nabla \ln\rho_{{\rm III}}|^{2}+\int_{1}^{\rho_{{\rm III}}}\frac{p(z)}{z}\d z+r_{2}\int_{\O}\ln_{-}\rho_{{\rm III}}\right)\d x$, then we have
\begin{align}\label{BD estimate kappa n vanished}
\mathbb{E}\left[\sup\limits_{t\in[0,T]}\tilde{e}_{{\rm III}}(t)^{r}\right]&\ls C,\\
\mathbb{E}\left[\left(r_{1}\int_{0}^{t}\int_{\O}\rho_{{\rm III}}|\u_{{\rm III}}|^{4}\d x\d s\right)^{r}\right]&\ls C,
\quad \mathbb{E}\left[\left(r_{2}\int_{0}^{t}\int_{\O}|\u_{{\rm III}}|^{2} \d x\d s\right)^{r}\right]\ls C,
\end{align}
where $C$ depends on $r, T, \mathbb{E}\left[e_{{\rm III}}(0)^{r}\right],\mathbb{E}\left[e_{{\rm III}}(0)^{\frac{6r}{5}}\right]$ and $\mathbb{E}\left[e_{{\rm III}}(0)^{k_{i}r}\right]$, $i=1,2,3$,
 in which $k_{{\rm I}}$ are specific constants.
\end{enumerate}
\end{lemma}
\subsection{Vanishing the drag forces: $r_{1}\ra 0,r_{2}\ra 0$ }
\par
\begin{flushleft}
\textbf{Step 1:} Choose the path space, and show the tightness of the laws and the convergence.
\end{flushleft}

We study the convergence of $\rho_{{\rm III}}$.
\begin{equation}
\begin{aligned}
 &\quad \mathbb{E}\left[\left\|\rho_{{\rm III}}^{\frac{1}{2}}\rho_{{\rm III}} ^{\frac{1}{2}}\odiv\u_{{\rm III}} \right\|_{L_t^{2} L_x^{\frac{3}{2}}}^{r}\right]  \ls \mathbb{E}\left[\left\|\rho_{{\rm III}}^{\frac{1}{2}}\right\|_{L_t^{\infty}L_x^{6}}^{r}\left\|\rho_{{\rm III}}^{\frac{1}{2}}\odiv\u_{{\rm III}}
       \right\|_{L_t^{2}L_x^{2}}^{r}\right] \ls C,\\
\end{aligned}
\end{equation}
together with
\begin{equation}
\begin{aligned}
&\quad \mathbb{E}\left[\left\| \nabla \rho_{{\rm III}}^{\frac{1}{2}}(\rho_{{\rm III}} ^{\frac{1}{2}}\u_{{\rm III}} )\right\|_{L_t^{2} L_x^{1}}^{r}\right]\ls\mathbb{E}\left[\left\|\nabla \rho_{{\rm III}}^{\frac{1}{2}}\right\|_{L_t^{\infty}L_x^{2}}^{r}\left\|\rho_{{\rm III}}^{\frac{1}{2}}\u_{{\rm III}} \right\|_{L_t^{2}L_x^{2}}^{r}\right] \ls C,\\
\end{aligned}
\end{equation}
implies
\begin{equation}
\begin{aligned}
&\quad \mathbb{E}\left[\left\|\left(\rho_{{\rm III}}\right)_{t}\right\|_{L_t^2 L_x^{1}}^{r}\right]\ls  C,\\
\end{aligned}
\end{equation}
uniformly in $r_{1},r_{2}$. \\
Calculating$\left|\nabla \rho_{{\rm III}}\right|=2\left|\nabla \rho_{{\rm III}} ^{\frac{1}{2}}\right| \rho_{{\rm III}} ^{\frac{1}{2}},$
 $\rho_{{\rm III}}$ is bounded in $L^{\infty}\left(0, T; W^{1,1}\left( \O \right)\right)$, by Aubin-Lion's lemma, we know that the strong convergence of $\rho_{{\rm III}}$ still holds in $L^{2}\left(0, T; L^{\frac{3}{2}}\left(\O \right)\right)$.

Take the space $\mathcal{X}_{6}=\mathcal{X}_{\rho_{0}}^{6} \times \mathcal{X}_{\frac{\mathbf{q}_{0}}{\sqrt{\rho_{0}}}}^{6}\times\mathcal{X}_{\rho}^{6}
 \times \mathcal{X}_{\rho^{\frac{1}{2}}\u}^{6}\times \mathcal{X}_{W}^{6}$, where $\mathcal{X}_{\rho_{0}}^{6}=L^{\gamma}\left( \O \right)$,
 $\mathcal{X}_{\frac{\mathbf{q}_{0}}{\sqrt{\rho_{0}}}}^{6}=L^{2}\left( \O \right)$, $\mathcal{X}_{\rho}^{6}=L^{2}\left(0,T; L^{\frac{3}{2}}\left( \O \right)\right)$,
  $\mathcal{X}_{\rho^{\frac{1}{2}}\u}^{6}=L^{2}\left([0,T]\times\O\right)$, $\mathcal{X}_{W}^{6}=C([0,T];\mathcal{H})$. Similarly, we have the following convergence. 
\begin{proposition}
There exist two families of $\mathcal{X}_{6}$-valued Borel measurable random variables
$$
\left\{\bar{\rho}_{0, {\rm III}}, \frac{\bar{\mathbf{q}}_{0, {\rm III}}}{\sqrt{\bar{\rho}_{0, {\rm III}}}}, \bar{\rho}_{{\rm III}},\bar{\rho}_{{\rm III}}^{\frac{1}{2}}\bar{\u}_{{\rm III}}, \bar{W}_{{\rm III}} \right\}
$$
 and
  $\left\{\rho_{0},\frac{\mathbf{q}_{0}}{\sqrt{\rho_{0}}}, \rho,  \rho^{\frac{1}{2}}\u, W \right\}$,
  defined on a new complete probability space $\left(\bar{\Omega},\bar{\mathcal{F}},\bar{\mathbb{P}} \right)$, such that (up to a sequence):
\begin{enumerate}
  \item
  $\mathcal{L}\left[\bar{\rho}_{0, {\rm III}}, \frac{\bar{\mathbf{q}}_{0, {\rm III}}}{\sqrt{\bar{\rho}_{0, {\rm III}}}}, \bar{\rho}_{{\rm III}},  \bar{\rho}_{{\rm III}}^{\frac{1}{2}}\bar{\u}_{{\rm III}}, \bar{W}_{{\rm III}} \right]$ and
   $\mathcal{L}\left[\rho_{0, {\rm III}}, \frac{\mathbf{q}_{0, {\rm III}}}{\sqrt{\rho_{0, {\rm III}}}}, \rho_{{\rm III}},  \rho_{{\rm III}}^{\frac{1}{2}}\u_{{\rm III}}, W_{{\rm III}} \right]$ \\
   coincide with each other on $\mathcal{X}_{6}$;
  \item $\mathcal{L}\left[\rho_{0},\frac{\mathbf{q}_{0}}{\sqrt{\rho_{0}}}, \rho, \rho^{\frac{1}{2}}\u, W \right]$ on $\mathcal{X}_{6}$ is a Radon measure;
  \item The random variables $\left\{\bar{\rho}_{0, {\rm III}}, \frac{\bar{\mathbf{q}}_{0, {\rm III}}}{\sqrt{\bar{\rho}_{0, {\rm III}}}}, \bar{\rho}_{{\rm III}},  \bar{\rho}^{\frac{1}{2}}\bar{\u}_{{\rm III}}, \bar{W}_{{\rm III}} \right\}$
   converges $\bar{\mathbb{P}}$ {\rm a.s.} to\\
    $\left\{\rho_{0},\frac{\mathbf{q}_{0}}{\sqrt{\rho_{0}}}, \rho, \rho \u , W \right\}$ in the topology of $\mathcal{X}_{6}$. 
\end{enumerate}
\end{proposition}
\begin{flushleft}
\textbf{Step 2:} $\left(\left(\bar{\Omega},\bar{\mathcal{F}},\bar{\mathbb{P}} \right),\rho, \u, W\right)$ is a global martingale solution to our system \eqref{sto NS system}, which finally proves our main theorem.
\end{flushleft}
\begin{lemma}
The stochastic B-D entropy estimates become
\begin{align}\label{BD estimate  kappa n vanished}
\tilde{e}=\int_{\O}\left(\frac{1}{2}\rho|\u+\nabla \ln\rho|^{2}+\int_{1}^{\rho}\frac{p(z)}{z}\d z\right)\d x,\quad
\mathbb{E}\left[\sup\limits_{t\in[0,T]}\tilde{e}(t)^{r}\right]&\ls C,
\end{align}
where $C$ depends on $r,T,\mathbb{E}\left[e(0)^{r}\right],\mathbb{E}\left[e(0)^{\frac{6r}{5}}\right]$ and $\mathbb{E}\left[e(0)^{k_{{\rm I}}r}\right]$, $i=1,2,3$, in which $k_{{\rm I}}$ are specific constants.
\end{lemma}
\begin{lemma} $\left(\left(\bar{\Omega},\bar{\mathcal{F}},\bar{\mathbb{P}} \right),\rho, \u, W\right)$ is a global-in-time martingale solution to system \eqref{sto NS system}.
\end{lemma}
\begin{proof}
 The proof is similar to the last section. We just show that the terms involving $r_{1} \bar{\rho}_{{\rm III}}\left|\bar{\u}_{{\rm III}}\right|^{2} \bar{\u}_{{\rm III}}$ and $r_{2} \bar{\u}_{{\rm III}}$ tend to zero in distribution sense as $r_{2} \rightarrow 0$ and $r_{1} \rightarrow 0$ $\bar{\mathbb{P}}$ a.s. 
For $r_{1} \bar{\rho}_{{\rm III}}\left|\bar{\u}_{{\rm III}}\right|^{2} \bar{\u}_{{\rm III}}$,  for all $\varphi(t)\in C_{c}^{\infty}\left([0,T)\right)$, and all $\psi(x)\in C^{\infty}\left(\O\right)$, we have
\begin{align}
    &\int_{0}^{T} \int_{\O} r_{1}\bar{\rho}_{{\rm III}}\left| \bar{\u}_{{\rm III}}\right|^{2} \bar{\u}_{{\rm III}}\varphi(t)\psi(x)\d x \d t \\
\ls &r_{1}^{\frac{1}{2}}\left\|r_{1}^{\frac{1}{2}} \bar{\rho}_{{\rm III}}^{\frac{1}{2}}\left|\bar{\u}_{{\rm III}}\right|^{2}\right\|_{L_t^{2}L_x^{2}}\left\|\bar{\rho}_{{\rm III}}^{\frac{1}{2}}\ \bar{\u}_{{\rm III}}\right\|_{L_t^{2}L_x^{2}}\|\varphi(t)_{L_t^{2}}\|\psi(x)\|_{L_x^{\infty}}
 \ra 0\notag
\end{align}
as $r_{1}\rightarrow 0$.
 The term $r_{2}\bar{\u}_{{\rm III}}$
\begin{align}
    &\int_{0}^{T} \int_{\O} r_{2}\bar{\u}_{{\rm III}}\varphi(t)\psi(x) \d x \d t
\ls \int_{0}^{T} \int_{\O}  r_{2}^{\frac{1}{2}} r_{2}^{\frac{1}{2}}\left|\bar{\u}_{{\rm III}}  \varphi(t)\psi(x)\right| \d x \d t\\
\ls &r_{2}^{\frac{1}{2}}\left\|r_{2}^{\frac{1}{2}} \bar{\u}_{{\rm III}}\right\|_{L_t^{2}L_x^{2}}\|\varphi(t)\|_{L_t^{2}}\|\psi(x)\|_{L_x^{2}} \ra 0\notag
\end{align}
as $r_{2}\rightarrow 0$.
Letting $r_{1}\rightarrow 0$ and $r_{2}\rightarrow 0$, we obtains that
\begin{align}
&-\varphi(0)\int_{\O} \mathbf{q}_{0} \cdot \psi(x) \d x-\int_{0}^{T}\varphi_{t}(t) \int_{\O} \rho\u \psi(x) \d x \d t \notag\\
 &-\int_{0}^{T}\varphi(t) \int_{\O} \rho\u\otimes \mathbf{u}: \nabla \psi(x) \d x \d t-\int_{0}^{T} \varphi(t)\int_{\O} \rho^{\gamma} \odiv \psi(x) \d x \d t\\
 &-\int_{0}^{T}\varphi(t) \int_{\O} \rho \mathbb{D} \mathbf{u}: \nabla \psi(x) \d x \d t =\int_{0}^{T} \varphi(t)\int_{\O} \rho\mathbb{F}(\rho,\u)\psi(x)\d x\d W,
\notag
\end{align}
$\bar{\mathbb{P}}$ a.s. \hfill $\square$
\end{proof}

\section{Appendix: Some preliminaries in stochastic analysis}
 In this section, for the convenience of the reader, we list some basic definitions and theorems in measure theory and  stochastic analysis.
 \subsection{Definitions} \label{def in app}
 We first list some definitions.
\begin{enumerate}
  \item [1.]\label{sto basis} {\bf Stochastic basis.} Let $\left(\Omega, \mathcal{F}, \mathbb{P}\right)$ be a stochastic basis. A filtration $\mathcal{F}=\left(\mathcal{F}_{t}\right)_{t \in \mathbf{T}}$ is a family of $\sigma$-algebras on $\Omega$ indexed by $\mathbf{T}$ such that $\mathcal{F}_{s} \subseteq$ $\mathcal{F}_{t} \subseteq \mathcal{F}$, $s \leq t$, $s, t \in \mathbf{T}$. A measurable space, together with a filtration, is called a filtered space.
  \item [2.]
{\bf Adapted stochastic process.} A stochastic process $X$ is $\mathcal{F}$-adapted if $X_{t}$ is $\mathcal{F}_{t}$-measurable for every $t \in \mathbf{T}$; 
\item [3.] {\bf Wiener Process.} An $\mathbb{R}^{m}$ -valued stochastic process $W$ is called an $\mathcal{F}$-adapted Wiener process, provided:
 $W(t)$ is $\mathcal{F}_{t}$-adapted;\quad  $W(0)=0, \quad  \mathbb{P}$ a.s.;\quad
$W$ has continuous trajectories: $t \mapsto W(t)$ is continuous $\mathbb{P}$ a.s.; \quad
$W$ has independent increments: $W(t)-W(s)$ is independent of $\mathcal{F}_{s}$ for all $0 \leq s \leq$ $t<\infty$.
\item [4.]\label{stopping time def}{\bf Stopping time.} On $\left(\Omega, \mathcal{F}, \mathbb{P}\right)$, a random time is a measurable mapping $\tau: \Omega \rightarrow \mathbf{T} \cup \infty$. A random time is a stopping time if $\{\tau \leq t\} \in \mathcal{F}_{t}$ for every $t \in \mathbf{T}$.
For a process $X$ and a subset $V$ of the state space we define the hitting time of $X$ in $V$ as
\begin{align}
\tau_{V}(\omega)=\inf \left\{\left.t \in \mathbf{T}\right| X_{t}(\omega) \in V\right\}
\end{align}
If $X$ is a continuous adapted process and $V$ is closed, then $\tau_{V}$ is a stopping time.
\item [5.] {\bf Martingale.} A process $X$ is called a martingale if
$X$ is adapted;
$X_{t}$ is integrable for every $t \in \mathbf{T}$;
$X_{s}=\mathbb{E}\left[X_{t} \mid \mathcal{F}_{s}\right]$ whenever $s \leq t, s, t \in \mathbf{T}$.
\item [6.]\label{progressive measurability}
{\bf Progressive measurability.} Let $\left(\Omega,\mathcal{F},\mathbb{P} \right)$ be a filtrated space. Stochastic process $X$ is progressively measurable or simply progressively measurable, if for $\omega\in \Omega$, $(\omega, s) \mapsto X(s,\omega), s \leq t$ is $\mathcal{F}_{t} \otimes \mathcal{B}(\mathbf{T} \cap[0, t])$-measurable for every $t \in \mathbf{T}$, where $\mathcal{B}(\mathbf{T} \cap[0, t])$ is Borel-algebra.
  In particular, progressively measurable processes are automatically adapted. The reciprocal is true if the paths of the process are regular enough. Let $X$ be right-continuous or left-continuous $\mathcal{F}$-adapted process. Then $X$ is progressively measurable. Any measurable and adapted process admits a progressive modification. 
\item [7.] {\bf Stochastic integral.} By $\mathcal{M}_{c}^{2}$ we denote the space of square integrable continuous martingales $M$ such that $M_{0}=0$. For $M$ in $\mathcal{M}_{c}^{2}$, we denote by $\langle M\rangle$ the continuous natural process of bounded variation in the Doob-Meyer decomposition of the sub-martingale $M^{2}$, that is, $M^{2}-\langle M\rangle$ is a martingale. We call $\langle M\rangle$ the quadratic variations of $M$. Let $\mathcal{L}^{2}(M)$ be the space of progressively measurable process $H$ such that $\mathbb{E}\left[\int_{0}^{t} H^{2} \d\langle M\rangle\right]<+\infty$. There exists a unique continuous linear functional $\int \cdot \d M: \mathcal{L}^{2}(M) \rightarrow$ $\mathcal{M}_{2}^{c}$ coinciding with the elementary stochastic integral on $\mathcal{S}$ and for which holds
\begin{align}
\mathbb{E}\left[\left(\int_{0}^{t} H \d M\right)^{2}\right]=\mathbb{E}\left[\int_{0}^{t} H^{2} \d\langle M\rangle\right].
\end{align}
Furthermore, the following properties hold
\begin{itemize}
                                             \item Linearity: $\int(\alpha H+\beta G) \d M=\alpha \int H \d M+\beta \int G \d M$ for $\alpha$ and $\beta$ constants;
                                             \item Stopping property: $\int 1_{\{\cdot \leq \tau\}} H \d M=\int H \d M^{\tau}=\int_{0}^{\cdot \wedge \tau} H \d M$;
                                             \item It\^o-Isometry: for every $t$,
$
\mathbb{E}\left[\left(\int_{0}^{t} H \d M\right)^{2}\right]=\mathbb{E}\left[\int_{0}^{t} H^{2} \d\langle M\rangle\right].
$
\end{itemize}
\item [8.]\label{def semi-martingale}{\bf Semi-Martingale.} A semi-martingale is a process $X$ with decomposition $X=X_{0}+M+A$, where $A$ is the difference of two increasing continuous processes and $M$ is a continuous local martingale. For two semi-martingales $X=X_{0}+M+A$ and $Y=Y_{0}+N+B$, we define
    \begin{itemize}
      \item the quadratic variations: $\langle X\rangle \triangleq \langle M\rangle$;
      \item the co-variations: $\langle X, Y\rangle \triangleq (\langle X+Y\rangle-\langle X-Y\rangle) / 4=(\langle M+N\rangle-\langle M-N\rangle) / 4=\langle M, N\rangle$.
    \end{itemize}
     For every progressive process $H$ such that $\int_{0}^{t}|H| \d|A|<\infty$ and $\int_{0}^{t}|H|^{2} \d\langle M\rangle<\infty$ a.s. for every $t$, we therefore define
\begin{align}
\int H \d X\triangleq\int H \d M+\int H \d A.
\end{align}
  \item [9.] \label{Radon measure} {\bf Radon measure.} Let $\mathcal{M}$ be a measure on the $\sigma$-algebra with its element belonging to a Hausdorff topological space $S$. The measure $\mathcal{M}$ is called inner regular or tight if, for any open set $U$,
\begin{align}
\mathcal{M}(U)=\sup_{K\subseteqq U}\{\mathcal{M}(K): K \quad is \quad compact\}.
\end{align}
The measure $\mathcal{M}$ is called outer regular if for any Borel set $B$,
\begin{align}
\mathcal{M}(B)=\inf_{B\subseteqq U}\{\mathcal{M}(U): U \quad is \quad open\}.
\end{align}
The measure $\mathcal{M}$  is called locally finite if every point of $S$ has a neighborhood $U$ for which $\mathcal{M}(U)$ is finite.
The measure $\mathcal{M}$  is called a Radon measure if it is inner regular, outer regular, and locally finite.
 A measure on $\mathbb {R}$ is a Radon measure if and only if it is a locally finite Borel measure. Since a probability measure is globally finite, hence locally finite, so every probability measure on a Radon space is also a Radon measure. In particular, a separable complete metric space $(S, \mathcal{M})$ is a Radon space. Gaussian measure on Euclidean space $\mathbb {R} ^{n}$ with its Borel $\sigma$-algebra is a Radon measure. 
  \item [10.]\label{def of law}
    {\bf Law of a random variable.} Let $(S, \mathcal{S})$ be a measurable space. An $S$-valued random variable is a measurable mapping
 $\mathbf{U}:(\Omega, \mathcal{F}) \rightarrow(S, \mathcal{S})$. 
 We denote by $\mathcal{L}[\mathbf{U}]$ or also $\mathcal{L}_{S}[\mathbf{U}]$ the law of $\mathbf{U}$ on $S$, that is, $\mathcal{L}[\mathbf{U}]$ is the push-forward probability measure on $(S, \mathcal{S})$ given by
\begin{equation}
\mathcal{L}\left[\mathbf{U}\right](A)=\mathbb{P}\left[\left\{\mathbf{U} \in A\right\}\right], \quad A \in \mathcal{S}.\\
\end{equation} 
In measure theory, a pushforward measure 
is obtained by using a measurable function, transferring a measure from one measurable space to another space.
\end{enumerate}
\subsection{Theorems}\label{appendix 2nd part}
In the paper, the following theorems are employed.
\begin{enumerate}
\item [1.] {\bf Chebyshev's inequality.} Let $X$ be a random variable in probability space $\left(\Omega, \mathcal{F}, \mathbb{P}\right)$, $\varepsilon>0$. For every $0<r<\infty$, Chebyshev's inequality reads
\begin{equation}
\mathbb{P}\left[\left\{|X| \geq \varepsilon\right\}\right] \leq \frac{1}{\varepsilon^r}\mathbb{E}\left[|X|^r\right] .
\end{equation}
  \item [2.]\label{Skorokhod thm}{\bf Jakubowski's extension of Skorokhod's representation theorem.}\cite{Skorokhod1957LimitTF}\cite{Jakubowski} Let $(S, \mathcal{M})$ be a sub-Polish space and let $\mathcal{S}$ be the $\sigma$-field generated by $\left\{f_{n}; n \in \mathbb{N}\right\}$. If $\left(\mu_{n}\right)_{n \in \mathbb{N}}$ is a tight sequence of probability measures on $(S, \mathcal{S})$, then there exists a subsequence $\left(n_{k}\right)$ and $S$-valued Borel measurable random variables $\left(\mathbf{U}_{k}\right)_{k \in \mathbb{N}}$ and $\mathbf{U}$ defined on the standard probability space $([0,1], \overline{\mathcal{B}([0,1])}, \bar{\mathbb{P}})$, such that $\mu_{n_{k}}$ is the law of $\mathbf{U}_{k}$ and $\mathbf{U}_{k}(\omega)$ converges to $\mathbf{U}(\omega)$ in $S$ for every $\omega \in[0,1]$. Moreover, the law of $\mathbf{U}$ is a Radon measure.
  \item [3.]\label{BDG inequa}{\bf Burkholder-Davis-Gundy's inequality.} \cite{BDG-inequality} Let $M$ be a continuous local martingale. Let $M^{\ast}=\max\limits_{0\ls s\ls t}|M(s)|$, for any $m>0$, then there exist constants $K^{m}$ and $K_{m}$ such that
\begin{equation}
 K_{m} \mathbb{E}\left[\left(\langle M\rangle_{T}\right)^{m}\right]\ls \mathbb{E}\left[\left(M^{\ast}_{T}\right)^{2m}\right]\ls K^{m} \mathbb{E}\left[\left(\langle M\rangle_{T}\right)^{m}\right],
\end{equation}
for every stopping time $T$. For $m\geq 1$, $K^{m}=\left(\frac{2m}{2m-1}\right)^{\frac{2m(2m-2)}{2}}$, which is equivalent to $e^{m}$ as $m\ra \infty$.
  \item [4.]\label{Centov thm}{\bf Kolmogorov-Centov's continuity theorem.}\cite{Karatzas1988} Let $(\Omega, \mathcal{F}, \mathbb{P})$ be a probability space and $\bar{X}$ a process on $[0, T]$. Suppose that
\begin{equation}
\mathbb{E}\left[\left|\bar{X}_{t}-\bar{X}_{s}\right|^{\alpha}\right] \leq C|t-s|^{1+\beta},
\end{equation}
for every $s<t \leq T$ and some strictly positive constants $\alpha, \beta$ and $C$. Then $\bar{X}$ admits a continuous modification $X$, $\mathbb{P}\left[\left\{X_{t}=\bar{X}_{t}\right\}\right]=1$ for every $t$, and $X$ is locally H\"older continuous for every exponent $0<\gamma<\frac{\beta }{\alpha},$ namely,
\begin{equation}
\mathbb{P}\left[\left\{\omega\left| \sum_{0<t-s<h(\omega), t, s \leq T} \frac{\left|X_{t}(\omega)-X_{s}(\omega)\right|}{|t-s|^{\gamma}} \leq \delta\right.\right\}\right]=1,
\end{equation}
where $h(\omega)$ is an strictly positive random variable a.s. and the constant $\delta>0$.
\item [5.]\label{Kolmogorov thm}{\bf Kolmogorov's theorem.} \cite{Karatzas1988} The sequence of probability measures $\left\{\mathbb{P}_{n}\right\}_{n=1}^{\infty}$ on  \\
$\left(C[0, \infty), \mathcal{B}\left(C[0, \infty)\right)\right)$,
is tight if and only if
\begin{align}
\lim\limits_{\lambda\ra \infty}\sup_{n\geqslant 1} \mathbb{P}_{n}\left[\left\{\left.\omega\right||\omega(0)|>\lambda\right\} \right]=0,
\end{align}
\begin{align}
\lim\limits_{\delta \ra 0} \sup_{n\geqslant 1} \mathbb{P}_{n}\left[\left\{\omega\left|\max\limits_{|s-t|\ls \delta}|\omega(s)-\omega(t)|>\eps \right.\right\}\right]=0;\quad 0 \ls s,t \ls T; \quad \forall T>0, \eps>0.
\end{align}
  \item [6.]\label{Ito formula}{\bf It\^o's formula.} \cite{Ito1944} Let $X$ and $Y$ be semi-martingales, then it holds
\begin{equation}
\begin{aligned}
X Y &= X_{0} Y_{0}+\int X \d Y+\int Y \d X+\int \d\langle X,Y \rangle\\
&=X_{0} Y_{0}+\int X \d Y+\int Y \d X + \langle X,Y\rangle,\\
\end{aligned}
\end{equation}
$\langle X,Y\rangle$ means the co-variation of $X$ and $Y$. In particular, this formula can be used to calculate integration by part. More precisely, if $X$ is continuous, then $\langle \d X,\d Y\rangle=0 $ and it holds $\int X \d Y=-\int Y \d X+ X Y-X_{0} Y_{0}$.
\end{enumerate}

\bigskip

\bigskip
\noindent{\bf Data Availability Statement} \quad Data sharing is not applicable to this article as no data sets were generated or analysed during the current study.


\begin{thebibliography}{10}

\bibitem{BreitFeireisl2020}
D.~Breit and E.~Feireisl.
\newblock Stochastic {N}avier--{S}tokes--{F}ourier equations.
\newblock {\em Indiana Univ. Math. J.}, 69(3):911--975, 2020.

\bibitem{BreitFeireislHofmanova2016}
D.~Breit, E.~Feireisl, and M.~Hofmanov\'a.
\newblock Incompressible limit for compressible fluids with stochastic forcing.
\newblock {\em Arch. Ration.Mech. Anal.}, 222(2):895--926, 2016.

\bibitem{BreitFeireislHofmanova2017}
D.~Breit, E.~Feireisl, and M.~Hofmanov\'a.
\newblock Compressible fluids driven by stochastic forcing: the relative energy
  inequality and applications.
\newblock {\em Comm. Math. Phys.}, 350(2):443--473, 2017.

\bibitem{BreitFeireislHofmanova-book2018}
D.~Breit, E.~Feireisl, and M.~Hofmanov\'a.
\newblock {\em Stochastically forced compressible fluid flows}.
\newblock Walter de Gruyter GmbH, Berlin, 2018.

\bibitem{BreitFeireislHofmanovaMalowski2019}
D.~Breit, E.~Feireisl, M.~Hofmanov\'a, and B.~Malowski.
\newblock Stationary solutions to the compressible Navier-Stokes system driven
  by stochastic forcing.
\newblock {\em Probab. Theory Relat. Fields}, 174(3--4):981--1032, 2016.

\bibitem{Breit-Hofmanova2016}
D.~Breit and M.~Hofmanov\'a.
\newblock Stochastic {N}avier--{S}tokes equations for compressible fluids.
\newblock {\em Indiana Univ. Math. J.}, 65:1183--1250, 2016.

\bibitem{BD2006}
D.~Bresch and B.~Desjardins.
\newblock On the construction of approximate solutions for the 2d viscous
  shallow water model and for compressible {N}avier--{S}tokes models.
\newblock {\em J. Math. Pures. Appl.}, 86(4):362--368, 2006.

\bibitem{Brzeziak-Dhariwal-Zatorska2022}
Z.~Brze\'ziak, G.~Dhariwal, and E.~Zatorska.
\newblock Sequential stability of weak martingale solutions to stochastic
  compressible Navier-Stokes equations with viscosity vanishing on vacuum.
  https://arxiv.org/abs/2201.02070.
\newblock 2022.

\bibitem{BDG-inequality}
D.~L. Burkholder, B.~J. Davis, and R.~F. Gundy.
\newblock Berkeley symposium on mathematical statistics and probability:
  Integral inequalities for convex functions of operators on martingales.
\newblock {\em Davis, Burgess, and Renming Song, eds. Selected Works of Donald L. Burkholder. Springer Science \& Business Media}, 2:223--240, 1945-1971.

\bibitem{chandrasekhar1957introduction}
S.~Chandrasekhar.
\newblock {\em An introduction to the study of stellar structure, Vol.2}.
\newblock Courier Corporation, 1957.

\bibitem{bookChanpmanCowling}
S.~Chapman and T.~G. Cowling.
\newblock {\em The mathematical theory of non-uniform gases: an account of the
  kinetic theory of viscosity, thermal conduction and diffusion in gases, 3rd
  ed.}
\newblock Cambridge university press, London, 1970.

\bibitem{DEBUSSCHE20111123}
A.~Debussche, N.~Glatt-Holtz, and R.~Temam.
\newblock Local martingale and pathwise solutions for an abstract fluids model.
\newblock {\em Physica D.}, 240(14):1123--1144, 2011.

\bibitem{Dong2010}
J.~Dong.
\newblock A note on barotropic compressible quantum {N}avier--{S}tokes
  equations.
\newblock {\em Nonlinear Anal.}, 73(4):854--856, 2010.

\bibitem{dampingintro}
S.~J. Elliott, M.~G. Tehrani, and R.~S. Langley.
\newblock Nonlinear damping and quasi-linear modelling.
\newblock {\em Philos. Trans. Royal Soc. A}, 373(2051):20140412, 2015.

\bibitem{Feireisl-Maslowski-Novotny2013}
E.~Feireisl, B.~Maslowski, and A.~Novotn\'y.
\newblock Compressible fluid flows driven by stochastic forcing.
\newblock {\em J. Differ. Equ.}, 254(3):1342--1358, 2013.

\bibitem{bookFeireislNovotny}
E.~Feireisl and A.~Novotn\'y.
\newblock {\em Singular limits in thermodynamics of viscous fluids}.
\newblock Springer Science and Business Media, Berlin, 2009.

\bibitem{Feireisl-Novotny-Petzeltova2001}
E.~Feireisl, A.~Novotn\'y, and H.~Petzeltov\'a.
\newblock On the existence of globally defined weak solutions to the
  {N}avier--{S}tokes equations.
\newblock {\em J. Math. Fluid Mech.}, 3:358--392, 2001.

\bibitem{Gisclon-Violet2015}
M.~Gisclon and L.~Violet.
\newblock About the barotropic compressible quantum {N}avier--{S}tokes
  equations.
\newblock {\em Nonlinear Anal.}, 128:106--121, 2015.

\bibitem{Grafakos2008}
L.~Grafakos.
\newblock {\em Classical {F}ourier analysis}.
\newblock Springer-Verlag, New York, 2008.

\bibitem{Hoff1992}
D.~Hoff.
\newblock Spherically symmetric solutions of the {N}avier--{S}tokes equations
  for compressible isothermal flow with large discontinuous initial data.
\newblock {\em Indiana U. Math. J.}, 41:1225--1302, 1992.

\bibitem{Huxianpeng2021}
X.~Hu.
\newblock Weak solutions for compressible isentropic {N}avier--{S}tokes
  equations in dimensions three.
\newblock {\em Arch. Ration. Mech. Anal.}, 242, 2021.

\bibitem{Ito1944}
K.~It\^o.
\newblock Stochastic integral.
\newblock {\em Proc. Imp. Acad. Tokyo}, 20:519--524, 1944.

\bibitem{Jakubowski}
A.~Jakubowski.
\newblock Short communication: The almost sure {S}korokhod representation for
  subsequences in nonmetric spaces.
\newblock {\em Theor. Probab. Appl.}, 42(1):167--174, 1998.

\bibitem{Jiang2011}
F.~Jiang.
\newblock A remark on weak solutions to the barotropic compressible quantum
  {N}avier--{S}tokes equations.
\newblock {\em Nonlinear Anal.}, 12:1733--1735, 2011.

\bibitem{Jiang-Zhang2003}
S.~Jiang and P.~Zhang.
\newblock Axisymmetric solutions of the 3{D} {N}avier--{S}tokes equations for
  compressible isentropic fluids.
\newblock {\em J. Math. Pure. Appl.}, 82:949--973, 08 2003.

\bibitem{Jungle2010}
A.~J\"ungel.
\newblock Global weak solutions to compressible {N}avier--{S}tokes equations
  for quantum fluids.
\newblock {\em SIAM J. Math. Anal.}, 42(3):1025--1045, 2010.

\bibitem{Karatzas1988}
I.~Karatzas and S.~Shreve.
\newblock {\em Brownian motion and stochastic calculus}.
\newblock Springer-Verlag, New York, 1988.

\bibitem{Kazhikhov-Vaigant1995}
A.~V. Kazhikhov and V.~A. Vaigant.
\newblock On the existence of global solutions to two-dimensional
  {N}avier--{S}tokes equations of a compressible viscous fiuid.
\newblock {\em Siberian Math. J.}, 36(6):1108--1141, 1995.

\bibitem{Kazhikhov-Shelukhin1977}
V.~Kazhikhov and V.~V. Shelukhin.
\newblock Unique global solution with respect to time of
  initial--boundary--value problems for one-dimensional equations of a viscous
  gas.
\newblock {\em J. Appl. Math. Mech.}, 41:273--282, 1977.

\bibitem{TatsienLi}
T.~Li and T.~Qin.
\newblock {\em Physics and partial differential equations, {V}ol.2}.
\newblock SIAM, and Higher Education Press, Philadelphia, 2012.

\bibitem{Lions1998}
P.~L. Lions.
\newblock {\em Mathematical topics in fluid dynamics, {V}ol.2}.
\newblock Oxford Science Publication, Oxford, 1998.

\bibitem{Mellet-Vasseur2007}
A.~Mellet and A.~Vasseur.
\newblock On the barotropic compressible {N}avier--{S}tokes equations.
\newblock {\em Commun. Partial. Differ. Equ.}, 32(3):431--452, 2007.


\bibitem{bookRayleigh}
B.~Rayleigh, J.~W. Strutt, Sc.D., and F.R.S.
\newblock {\em The theory of sound, {V}ol.1, 2nd ed.}
\newblock Dover Publications, New York, 1945.

\bibitem{Skorokhod1957LimitTF}
A.~V. Skorokhod.
\newblock Limit theorems for stochastic processes with independent increments.
\newblock {\em Theory Probab. its Appl.}, 2:138--171, 1957.

\bibitem{Smith2017}
S.~Smith.
\newblock Random perturbations of viscous, compressible fluids: Global
  existence of weak solutions.
\newblock {\em SIAM J. Math. Anal.}, 49(6):4521--4578, 2017.

\bibitem{Stroock-Varadhan}
D. W.~Stroock, and S. S.~Varadhan.
\newblock {\em Multidimensional diffusion processes. {V}ol. 233.}
\newblock Springer Science \& Business Media, New York, 1997.

\bibitem{Tornatore2000}
E.~Tornatore.
\newblock Global solution of bi--dimensional stochastic equation for a viscous
  gas.
\newblock {\em Nonlinear Differ. Equ. Appl.}, 7:343--360, 2000.

\bibitem{Vasseur-Yu2016}
A.~Vasseur and C.~Yu.
\newblock Existence of global weak solutions for three-dimensional degenerate
  compressible {N}avier--{S}tokes equations.
\newblock {\em Invent. Math.}, 206:935--974, 2016.

\bibitem{Vasseur-Yu-q2016}
A.~Vasseur and C.~Yu.
\newblock Global weak solutions to the compressible quantum {N}avier--{S}tokes
  with damping.
\newblock {\em SIAM J. Math. Anal.}, 48(2):1489--1511, 2016.

\bibitem{Wang-Wang2015}
D.~Wang and H.~Wang.
\newblock Global existence of martingale solutions to the three-dimensional
  stochastic compressible {N}avier--{S}tokes equations.
\newblock {\em Differ. Integral Equ.}, 28(11-12):1105--1154, 2015.

\bibitem{Zatorska2012}
E.~Zatorska.
\newblock On the flow of chemically reacting gaseous mixture.
\newblock {\em J. Differ. Equ.}, 253:3471--3500, 2012.

\end{thebibliography}
\end{document}